\documentclass{article}

\usepackage{amsmath,amssymb}
\usepackage{amscd}
\usepackage{mathrsfs}
\usepackage{amsthm}

\theoremstyle{plain}

\newtheorem{thm}{\bf Theorem}[section]
\newtheorem{lem}[thm]{\bf Lemma}
\newtheorem{prop}[thm]{\bf Proposition}
\newtheorem{cor}[thm]{\bf Corollary}

\theoremstyle{definition}
\newtheorem{dfn}[thm]{Definition}

\newtheorem{rem}[thm]{\bf Remark}

\newcommand{\M}{\mathfrak{M}}
\newcommand{\cM}{\overline{\mathfrak{M}}}
\newcommand{\mf}[1]{\mathfrak{#1}}
\newcommand{\cN}{\overline{\mf{N}}}

\newcommand{\norm}[1]{{\rm{||}}#1{\rm{||}}}
\newcommand{\mc}[1]{\mathcal{#1}}
\newcommand{\msc}[1]{\mathscr{#1}}
\newcommand{\tens}{\otimes}
\newcommand{\eps}{\varepsilon}
\newcommand{\eqa}[1]{
\begin{align*}
#1
\end{align*}}

\DeclareMathOperator*{\Bca}{\bigcap}
\DeclareMathOperator*{\Bcu}{\bigcup}

\usepackage[all]{xy}

\newcommand{\mbb}[1]{\mathbb{#1}}
\newcommand{\brac}[1]{\left (#1\right )}

\newcommand{\limn}{\lim_{n\to \infty}}

\newcommand{\ol}[1]{\overline{#1}}
\newcommand{\tb}[1]{\textbf{#1}}

\newcommand{\Hs}{\mathcal{H}}
\newcommand{\nai}[2]{\langle #1,#2\rangle}
\newcommand{\dom}[1]{{{\rm{dom}}{(#1)}}}
\newcommand{\ran}[1]{{\rm{ran}}{(#1)}}
\newcommand{\cran}[1]{\overline{{\rm{ran}}}(#1)}

\newcommand{\atensor}{\otimes_{{\rm{alg}}}}
\newcommand{\vtensor}{\overline{\otimes}}
\newcommand{\slim}[1]{{\rm{s}}\text{-}\lim_{#1}}
\newcommand{\disp}{\displaystyle}
\newcommand{\slimn}{{\rm{s}}\text{-}\lim_{n\to \infty}}

\newcommand{\wlimn}[1]{{\rm{w}}\text{-}\lim_{n\to \infty}}

\renewcommand{\ker}[1]{{\rm{ker}}(#1)}
\newcommand{\re}[1]{{\rm{Re}}(#1)}
\newcommand{\im}[1]{{\rm{Im}}(#1)}
\newcommand{\rc}[1]{{\rm{\mathscr{RC}}}(#1)}

\title{Lie Group-Lie Algebra Correspondences of\\ Unitary Groups in Finite von Neumann Algebras}

\author{Hiroshi Ando\\
Research Institute for Mathematical Sciences, Kyoto University\\
Kyoto, 606-8502, Japan\\
E-mail: andonuts@kurims.kyoto-u.ac.jp\\
\\
Yasumichi Matsuzawa$^{1,2, }$\thanks{
Supported by Research fellowships of the Japan Society for the Promotion of Science for Young Scientists 
(Grant No. 2100165000).
} \\
$^{1\ }$Mathematisches Institut, Universit\"{a}t Leipzig\\
Johannisgasse 26, 04103, Leipzig, Germany\\
$^{2\ }$Department of Mathematics, Hokkaido University\\
Kita 10, Nishi 8, Kita-ku, Sapporo, 060-0810, Japan\\
E-mail: matsuzawa@math.sci.hokudai.ac.jp
}

\begin{document}
\maketitle
\begin{abstract} 
We give an affirmative answer to the question whether there exist Lie algebras for suitable closed subgroups 
of the unitary group $U(\Hs)$ in a Hilbert space $\Hs$ with $U(\Hs)$ equipped with the strong operator topology. 
More precisely, for any strongly closed subgroup $G$ of the unitary group $U(\M)$ 
in a finite von Neumann algebra $\M$, 
we show that the set of all generators of strongly continuous one-parameter subgroups of $G$ 
forms a complete topological Lie algebra with respect to the strong resolvent topology.
We also characterize the algebra $\cM$ of all densely defined closed operators affiliated with $\M$ 
from the viewpoint of a tensor category.
\end{abstract}


\medskip


\medskip

\pagebreak 
\tableofcontents

\section{Introduction and Main Theorem}
\label{intro}
Lie groups played important roles in mathematics because of its close relations with the notion of symmetries. 
They appear in almost all branches of mathematics and have many applications.
While Lie groups are usually understood as finite dimensional ones, 
many infinite dimensional symmetries appear in natural ways: for instance, loop groups $C^{\infty}(S^1,G)$ \cite{Pres}, 
current groups $C^{\infty}_{c}(M,G)$ \cite{Alb}, 
diffeomorphism groups $\mathit{Diff}^{\infty}(M)$ of manifolds \cite{Banyaga} 
and Hilbert-Schmidt groups \cite{Gordina} are among well-known cases.
They have been extensively investigated in several concrete ways.   

In this context, it would be meaningful to consider a general theory of infinite dimensional Lie groups.
One of the most fundamental infinite dimensional groups are Banach-Lie groups.
They are modeled on Banach spaces and many theorems in finite dimensional cases
are also applicable to them. 
Since it has been shown that a Banach-Lie group cannot act transitively and effectively on a compact
manifold as a transformation group \cite{Omori2}, however, Banach-Lie groups are not sufficient for treating
infinite dimensional symmetries.
After the birth of Banach-Lie group theory, more general notions of infinite dimensional Lie groups have been scrutinized to date: 
locally convex Lie groups \cite{Neeb}, ILB-Lie groups \cite{Omori}, pro-Lie groups \cite{Hofmann, Hofmann2} and so on. 
While there are many interesting and important results about them, 
we note that not all theorems in finite dimensional cases remain valid in these categories 
and their treatments are complicated. 
For example, the exponential map might not be a local homeomorphism 
and the Baker-Campbell-Hausdorff formula may no longer be true \cite{Milnor}.

We understand that the one of the most fundamental class of finite dimensional Lie groups 
are the unitary groups $U(n)$ in such a sense that any compact Lie group can be realized 
as a closed subgroup of them. 
From this viewpoint, 
it would be important to study the infinite dimensional analogue of it; 
that is, we like to explicate the Lie theory for the unitary group $U(\Hs)$ of an infinite dimensional Hilbert space $\Hs$.
One of the most fundamental question is whether $\text{Lie}(G)$ defined as the set
of all generators of continuous one-parameter subgroups of a closed subgroup 
$G$ of $U(\mathcal{H})$ forms a Lie algebra or not.
For the infinite dimensional
Hilbert space $\mathcal{H}$, there are at least two topologies on $U(\Hs)$, 
(a) the norm topology and (b) the strong operator topology. 
We discuss the above 
topologies separately.
In the case (a), $U(\Hs)$ is a Banach-Lie group and for each closed subgroup the set $\text{Lie}(G)$ forms a Lie algebra.
But it is well known that there are not many ``nice'' continuous unitary representations of groups 
in $\Hs$, and hence, $U(\Hs)$ with the norm topology is very narrow.
On the other hand, $U(\Hs)$ with the strong operator topology (b) 
is important, because there are many ``nice" continuous unitary representations 
of groups in $\Hs$--say, diffeomorphism groups of compact manifolds, etc. 
However, the answer is
negative to the question whether there exists a corresponding Lie algebra or not.
Indeed, by the Stone theorem, the Lie algebra of $U(\Hs)$ coincides with 
the set of all (possibly unbounded) skew-adjoint operators on $\Hs$,
but we cannot define 
naturally a Lie algebra structure with addition and Lie bracket operations on it.
This arises from the problem of the domains of unbounded operators.
For two skew-adjoint operators $A,B$ on $\Hs$, $\dom{A+B}=\dom{A}\cap \dom{B}$ is not always dense.
Even worse, it can be $\{0\}$ (see Remark \ref{3domain is can be zero}).
Because of this, the Lie theory for $U(\Hs)$ has not been successful, although the group itself is a very natural object.
On the other hand, it is possible that even though the whole group $U(\Hs)$ does not have a Lie algebra, 
some suitable class of closed subgroups of it have ones.
Indeed their Lie algebras $\text{Lie}(G)$ are smaller than $\text{Lie}(U(\Hs))$.

We give an affirmative answer to the last question.
Furthermore we prove that for a suitable subgroup $G$, $\text{Lie}(G)$ 
is a complete topological Lie algebra with respect to some natural topology.
We outline below the essence of our
detailed discussions in the text.

First, a group $G$ to be studied in this paper is a closed subgroup of the unitary group $U(\M)$ 
of some finite von Neumann algebra $\M$ acting on a Hilbert space $\Hs$. 
Clearly it is also a closed subgroup of $U(\Hs)$. 
The key proposition is the following result of Murray-von Neumann (cf. Theorem \ref{2eta_*alg}):

\begin{thm}[Murray-von Neumann]
The set $\cM$ of all densely defined closed operators affiliated with a finite von Neumann algebra $\M$ on $\Hs$,
\begin{align*}
\cM := \left\{ 
A \ ;
\begin{array}{l} A\ \ \text{is \ a \ densely \ defined \ closed \ operator \ on }\ \Hs \\ 
\ \ \ \ \ \ \ \ \
\text{such \ that}\ \ uAu^*=A\ \ \text{for \ all}\ \ u \in U(\M'). 
\end{array} \right\},
\end{align*}
constitutes a *-algebra under the sum $\overline{A+B}$, 
the scalar multiplication $\overline{\alpha A}\ (\alpha \in \mathbb{C})$, 
the product $\overline{AB}$ and the involution $A^*$, where $\overline{X}$ denotes the closure of a closable operator $X$.
\end{thm}

The inclusion $G\subset U(\M)$ implies $\text{Lie}(G)\subset \cM$ and hence, for arbitrary two elements $A$, $B\in \text{Lie}(G)$,
the sum $\overline{A+B}$, the scalar multiplication $\overline{\alpha A}$, 
the Lie bracket $[A,B]:=\overline{AB-BA}$ are determined as elements of $\cM$.
We can prove that they are again elements of $\text{Lie}(G)$, which is not trivial. 
Therefore $\text{Lie}(G)$ indeed forms a Lie algebra which is infinite dimensional in general.
Thus if we do not introduce a topology, it is difficult to investigate it.
Then, what is the natural topology on $\text{Lie}(G)$?
Since $\text{Lie}(G)$ is a Lie algebra, it should be a  
vector space topology. Furthermore, 
in view of the 
correspondences between Lie groups and Lie algebras it is natural to require the continuity of the mapping
\begin{equation*}
\exp : \text{Lie}(G)\ni A\longmapsto e^{A}\in G,
\end{equation*}
where $G$ is equipped with the strong operator topology and $e^{A}$ is defined by the spectral theorem. 
Under these assumptions, a necessary condition for 
a sequence $\{A_n\}_{n=1}^{\infty}\subset \text{Lie}(G)$ to converge to $A\in \text{Lie}(G)$ is given by
\begin{equation*}
\slimn e^{tA_n}=e^{tA}, \ \ \ \ \text{for all}\ t\in \mathbb{R}.
\end{equation*}
This condition is equivalent to
\begin{equation*} 
\slimn (A_n+1)^{-1}=(A+1)^{-1}.
\end{equation*}
The latter convergence is well known in the field of (unbounded) operator theory 
as the convergence with respect to the strong resolvent topology. 
Therefore it seems natural to consider the strong resolvent topology for $\text{Lie}(G)$. 
However, there arises,
unfortunately, another troublesome question as to whether the vector space
operations and the Lie bracket operation are continuous with respect to the
strong resolvent topology of $\text{Lie}(G)$.
For example, even if sequences $\{A_n\}_{n=1}^{\infty}$, $\{B_n\}_{n=1}^{\infty}$ of skew-adjoint operators converge, 
respectively to skew-adjoint operators $A$, $B$ 
with respect to the strong resolvent topology, the sequences $\{A_n+B_n\}_{n=1}^{\infty}$ 
are not guaranteed to converge to $A+B$ (see Remark \ref{3not linear}). 
We can solve this difficulty by applying the noncommutative integration theory and proving that the Lie algebraic operations are continuous 
with respect to the strong resolvent topology and that $\text{Lie}(G)$ is complete as a uniform space.
Hence $\text{Lie}(G)$ forms a complete topological Lie algebra.
Finally, let us remark one point: remarkably, $\text{Lie}(G)$ is not locally convex in general. 
Most of infinite dimensional Lie theories assume the local convexity explicitly, 
but as soon as we consider such groups as natural infinite dimensional analogues of classical Lie groups, 
there appear non-locally convex examples.

We shall explain the contents of the paper. 
$\S$2 is a preliminary section. 
We recall the basic facts about closed operators affiliated with a finite von Neumann algebra 
and explain the generalization of the Murray-von Neumann theorem for a non-countably decomposable case. 
In $\S$3, we introduce three topologies on the set $\cM$ of all densely defined closed operators affiliated with 
a finite von Neumann algebra $\M$.
The first topology originates from (unbounded) operator theory, the second one is Lie theoretical 
and the last one derives from the noncommutative integration theory.
We discuss their topological properties and show that they do coincide on $\overline{\mathfrak{M}}$.
The main result of this section is Theorem \ref{3cta}
which states that 
$\overline{\mathfrak{M}}$ forms a complete topological *-algebra with respect to the strong resolvent topology.
In $\S$4 constituting the main
contents of the paper, we show that $\text{Lie}(G)$ is a complete
topological Lie algebra and discuss some aspects of it. The main result is
given in Theorem \ref{4main}.
In $\S$5, applying the results of $\S$3, we consider the following problem: 
{\it What kind of unbounded operator algebras can they be represented in the form of $\cM$?}
We give their characterization from the viewpoint of a tensor category. 
We show that $\mathscr{R}$ can be represented as $\cM$ if and only if 
it is an object of the category \textbf{fRng} (cf. Definition \ref{6def of fRng}).  
The main result is Theorem \ref{6main theorem}, which says that the category \textbf{fRng} 
is isomorphic to the category \textbf{fvN} of finite von Neumann algebras as a tensor category. 
In Appendix, we list up some fundamental definitions and results of the direct sums of operators, 
the strong resolvent convergence and the categories.

\section{Preliminaries}
\label{sec2}

In this section we review some basic facts about operator algebras and unbounded operators. 
For the details, see \cite{Reed-Simon, Tak}.
See also Appendix A for the direct sums of operators.

\subsection{von Neumann Algebras} 

Let $\Hs$ be a Hilbert space with an inner product $\nai{\xi}{\eta}$, which is linear with respect to $\eta$.
We denote the algebra of all bounded operators on $\Hs$ by $\mathfrak{B}(\Hs)$.
Let $\mf{M}$ be a von Neumann algebra acting on $\Hs$.
The set 
\begin{equation*}
\M':=\left\{x\in \mf{B}(\mathcal{H})\ ;\ xy=yx,\ {\rm for\ all}\ y\in \M\right\}
\end{equation*}
is called the \textit{commutant} of $\M$. 
The group of all unitary operators in $\mf{M}$ is denoted by $U(\mf{M})$. 
The lattice of all projections in $\M$ is denoted by $P(\M).$ 
The orthogonal projection onto the closed subspace $\mathcal{K}\subset \mathcal{H}$ is denoted by $P_{\mathcal{K}}$. 
For a projection $p$ in $\M$, we denote $1-p$ as $p^{\perp}$.

\begin{dfn}
Let $\M$ be a von Neumann algebra acting on a Hilbert space $\Hs$.
\begin{list}{}{}
\item[(1)] A von Neumann algebra with no non-unitary isometry is called \textit{finite}. 
\item[(2)] A von Neumann algebra is called \textit{countably decomposable} if it admits at most countably many non-zero orthogonal projections.
\item[(3)] A subset $\mathcal{D}$ of $\Hs$ is called \textit{separating} for $\M$
if $x\xi=0$, $x\in\M$ for all $\xi\in\mathcal{D}$ implies $x=0$.  
\end{list}
\end{dfn}

It is known that a von Neumann algebra $\M$ acting on a Hilbert space $\Hs$ is countably decomposable 
if and only if there exists a countable separating subset of $\Hs$ for $\M$.

\begin{dfn}
Let $\M$ be a von Neumann algebra.
\begin{list}{}{}
\item[(1)] A state $\tau$ on $\M$ is called \textit{tracial} if for all $x, y\in \M$,
\begin{equation*}
\tau(xy)=\tau(yx)
\end{equation*}
holds.
\item[(2)] A tracial state $\tau$ is called \textit{faithful} if $\tau(x^*x)=0$ ($x\in\M$) implies $x=0$.
\item[(3)] A tracial state $\tau$ is called \textit{normal} if it is $\sigma$-weakly continuous.
\end{list}
\end{dfn}

It is known that a von Neumann algebra is countably decomposable and finite if and only if 
there exists a faithful normal tracial state on it. 
For more informations about tracial states, see \cite{Tak}.

Let $\M$ be a von Neumann algebra and $p\in\M\cup\M'$ be a projection.
Define the set $\M_p$ of bounded operators on the Hilbert space $\ran{p}$ as
\begin{equation*}
\left\{px|_{\ran{p}}\ ;\ x\in\M \right\},
\end{equation*}
then $\M_p$ forms a von Neumann algebra acting on the Hilbert space $\ran{p}$ and $\left(\M_{p}\right)'=(\M')_p$ holds.

If $(\M, \Hs)$ and $(\mathfrak{N}, \mathcal{K})$ are von Neumann algebras 
and if there exists a unitary operator $U$ of $\Hs$ onto $\mathcal{K}$ such that
\begin{equation*}
U\M U^{*} = \mathfrak{N},
\end{equation*} 
then $(\M, \Hs)$ and $(\mathfrak{N}, \mathcal{K})$ are said to be {\it spatially isomorphic}.
The map $\pi$ of $\M$ onto $\mathfrak{N}$ defined by 
\begin{equation*}
\varphi (x) = UxU^{*}, \ \ \ \ x\in\M,
\end{equation*} 
in called a {\it spatial isomorphism}.
The next Lemma is useful.

\begin{lem}\label{2decompose}
Let $\left(\M,\Hs\right)$ be a finite von Neumann algebra.
Then there exists a family of countably decomposable finite von Neumann algebras $\left\{\left(\M_{\alpha},\Hs_{\alpha}\right)\right\}_{\alpha}$
such that $\left(\M,\Hs\right)$ is spatially isomorphic to the direct sum
$\left(\bigoplus_{\alpha}^b\M_{\alpha},\bigoplus_{\alpha}\Hs_{\alpha}\right)$. 
\end{lem}

A von Neumann algebra $\M$ is called \textit{atomic} if each non-zero projection in $\M$ majorizes a non-zero minimal projection.
It is known that a finite von Neumann algebra is atomic if and only if it is spatially isomorphic 
to the direct sum of finite dimensional von Neumann algebras $M_{n}(\mathbb{C})$ ($n\in\mathbb{N}$), 
where $M_{n}(\mathbb{C})$ is the algebra of all $n\times n$ complex matrices.    

A von Neumann algebra with no non-zero minimal projection is called \textit{diffuse}.
It is known that every von Neumann algebra is spatially isomorphic to the direct sum of 
some atomic von Neumann algebra $\M_{\rm atomic}$ and diffuse von Neumann algebra $\M_{\rm diffuse}$. 
These von Neumann algebras $\M_{\rm atomic}$ and $\M_{\rm diffuse}$ are unique up to spatial isomorphism. 
We call $\M_{\rm atomic}$ and $\M_{\rm diffuse}$ the \textit{atomic part} and the \textit{diffuse part} of $\M$, respectively.

\subsection{Murray-von Neumann's Result}

The domain of a linear operator $T$ on $\mathcal{H}$ is written as $\dom{T}$ and the range of it is written as $\ran{T}$.
If $T$ is a closable operator, we write $\overline{T}$ for the closure of $T$.
 
\begin{dfn}
A densely defined closable operator $T$ on $\Hs$ is said to be \textit{affiliated} with a von Neumann algebra $\mf{M}$ 
if for any $u\in U(\mf{M}')$, $uTu^*=T$ holds. 
If $T$ is affiliated with $\M$, so is $\overline{T}$. 
The set of all densely defined closed operators affiliated with $\M$ is denoted by $\overline{\M}.$
Each element in $\overline{\M}$ is called a {\it affiliated operator}.  
\end{dfn}

Note that $T$ is affiliated with $\M$ if and only if $xT\subset Tx$ for all $x\in \mf{M}'$. 
Next, we define algebraic structures of unbounded operators in the style of Murray-von Neumann \cite{vNMur}.

Let $x_1,y_1,x_2,y_2,\cdots$ be (finite or countable infinite number of) indeterminants. 
A non-commutative monomial with indeterminants $\{x_i,\ y_i\}_{i}$ is a formal product $z_1z_2\cdots z_n$, 
where all $z_k$ equal to $x_i\text{ or }y_i$. If $n=0$, we write this monomial as 1.
A non-commutative polynomial $p(x_1,y_1,\cdots )$ is a formal sum of finite number of monomials. 
$p(x_1,y_1,\cdots)$ has the following form:
\[p(x_1,y_1,\cdots )=\begin{cases}
\disp \sum_{\rho =1}^qa_{\rho}\cdot z_1^{(\rho)}\cdots z_{n_{\rho}}^{(\rho)} & (q =1,2,\cdots ),\\
0 & (q=0).
\end{cases}\]
Here, $a_{\rho}\in \mathbb{C}$ and we allow such a term as $0\cdot z_1z_2\cdots z_n$ in this expression. 
If there is a term with coefficient 0, it cannot be omitted in the representation. 
Hence $x_1$ is different from $x_1+0\cdot y_1$ as non-commutative polynomials . 
If there are two such terms as $a\cdot z_1\cdots z_n,\ b\cdot z_1\cdots z_n$, 
we identify the sum of them with the term $(a+b)\cdot z_1\cdots z_n$. 
The sum, the scalar multiplication and the multiplication of non-commutative polynomials are defined naturally, 
where we do not ignore the terms with 0 coefficients.

Once a non-commutative polynomial $p(x_1,y_1,\cdots )$ is given we obtain a new polynomial $p^{(r)}(x_1,y_1,\cdots )$ 
by omitting terms with coefficient $a_{\rho}=0$ in the representation of $p$. 
We call $p^{(r)}(x_1,y_1,\cdots )$ the reduced polynomial of $p$. 
We also define the adjoint element by $x_i^{+}:=y_i,\ y_i^{+}:=x_i$. 
We also define the conjugate polynomial of $p$ by
\[
p(x_1,y_1,\cdots )^{+}:=
\begin{cases}
\displaystyle \sum_{\rho =1}^q\ol{a_{\rho}} \cdot (z_{n_{\rho}}^{(\rho)})^{+}\cdots (z_{1}^{(\rho)})^{+} & (q=1,2,\cdots ),\\
0 & (q=0).
\end{cases}\]

Suppose there is a corresponding sequence $\{X_i\}_{i}$ of densely defined closed operators on $\mathcal{H}$. 
For all $i$, we assume $(x_i,y_i)$ corresponds to the pair of the closed operators $(X_i,X_i^*)$. 
In this case we define a new operator $p(X_1,X_1^*,\cdots )$ obtained by substituting each $\{x_i,y_i\}$ 
in the representation of $p(x_1,y_1,\cdots )$ of the pairs $(X_i,X_i^*).$ 
More precisely, the domain of $p(X_1,X_1^*,\cdots )$ is defined according to the following rules:
\begin{list}{}{}
\item[(1)] $\dom{0}=\dom{1}=\mathcal{H}$,\\ 
$0\xi:=0,\ 1\xi:=\xi$,\ \ \ \  for all $\xi \in \Hs$,
\item[(2)] $\dom{aX}:=\dom{X}$,\\ 
$(aX)\xi:=a(X\xi)$,\ \ \ \  for all $\xi \in \dom{aX}$,
\item[(3)] $\dom{X+Y}:=\dom{X}\cap \dom{Y}$,\\ $
(X+Y)\xi:=X\xi +Y\xi$,\ \ \ \ for all $\xi \in \dom{X+Y}$,
\item[(4)] $\dom{XY}:=\left\{\xi \in \dom{Y}\ ;\ Y\xi \in \dom{X}\right\}$,\\ 
$(XY)\xi:=X(Y\xi)$,\ \ \ \ for all $\xi \in \dom{XY}$,
\end{list}
where $X$ and $Y$ are densely defined closed operators on $\Hs$ and $a\in \mbb{C}$. 
In general, $\cM$ is not a *-algebra under these operations. 
This is the reason for the difficulty of constructing Lie theory in infinite dimensions. 
However, Murray and von Neumann proved, in the pioneering paper \cite{vNMur}, 
that for a finite von Neumann algebra $\M$, $\cM$ does constitute a *-algebra of unbounded operators, 
which we will explain more precisely in the sequel.

Murray-von Neumann proved the following results for a countably decomposable case. 
Since we need to apply these results for a general finite von Neumann algebra case, 
we shall offer the generalization of their proofs.
First of all, we recall the notion of complete density, which is important for later discussions.

\begin{dfn}
A subspace $\mathcal{D}\subset\Hs$ is said to be \textit{completely dense} 
for a finite von Neumann algebra $\M$ 
if there exists an increasing net $\{p_{\alpha}\}_{\alpha}\subset P(\M)$ of projections in $\M$ such that
\begin{list}{}{}
\item[(1)] $p_{\alpha}\nearrow 1$ (strongly).
\item[(2)] $p_{\alpha}\mathcal{H}\subset \mathcal{D}$ for any $\alpha$.
\end{list}
\end{dfn}

It is clear that a completely dense subspace is dense in $\Hs$. 
We often omit the phrase ``for $\M$" when the von Neumann algebra in consideration is obvious from the context.

\begin{rem} In \cite{vNMur}, Murray and von Neumann used the term ``strongly dense". 
However, this terminology is somewhat confusing. 
Therefore we tentatively use the term ``completely dense".
\end{rem}

\begin{lem}\label{2c-dense for sigma finite case} 
Let $\M$ be a countably decomposable, finite von Neumann algebra on a Hilbert space $\Hs$, 
$\tau$ be a faithful normal tracial state on $\M$. 
For a subspace $\mathcal{D}\subset \Hs$, the following are equivalent.
\begin{list}{}{}
\item[(1)] $\mathcal{D}$ is completely dense.
\item[(2)] There exists an increasing sequence $\{p_n\}_{n=1}^{\infty}\subset P(\M)$ such that 
\[p_n\nearrow 1 \ (\text{strongly}) ,\ \ \ \   \ran{p_n}\subset \mathcal{D}.\]
\item[(3)] For every $\varepsilon>0$, there exists $p\in P(\M)$ such that 
\[\tau(p^{\perp})<\varepsilon, \ \ \ \  p\Hs \subset \mathcal{D}.\]
\end{list}
\end{lem}

\begin{proof}
It is clear that (2)$\Rightarrow $(1)$\Rightarrow$(3) holds. 
We shall prove (3)$\Rightarrow $(2). 
By assumption, for all $n\in \mbb{N}$, there exists $p_n\in P(\M)$ such that 
$\tau(p_n^{\perp})<1/2^{n}$ and $p_n\Hs\subset \mathcal{D}$. 
Put
\[q_n:=\bigwedge_{k=n}^{\infty}p_k\in P(\M).\]
Since $q_n\le q_{n+1}$, the strong limit $s$-$\lim_{n\to \infty} q_n=:q\in P(\M)$ exists. 
It holds that
\eqa{
\tau(q^{\perp})&=\lim_{n\to \infty}\tau(q_n^{\perp})=\lim_{n\to \infty}\tau \left (\bigvee_{k=n}^{\infty}p_k^{\perp}\right )\\
&\le \lim_{n\to \infty}\sum_{k=n}^{\infty}\tau (p_k^{\perp})\le \lim_{n\to \infty}\sum_{k=n}^{\infty}\frac{1}{2^k}=0.
}
Therefore we have $q=1$.
\end{proof}

\begin{lem}\label{2alg-sum of c-dense spaces} 
Let $\left\{(\M_{\lambda},\Hs_{\lambda})\right\}_{\lambda \in \Lambda}$ be a family of countably decomposable, finite von Neumann algebras. 
Let
\[\M:=\bigoplus_{\lambda \in \Lambda}^b\M_{\lambda},\ \ \ \  \Hs:=\bigoplus_{\lambda \in \Lambda}\Hs_{\lambda}.\]
For each $\lambda \in \Lambda$, let $\mathcal{D}_{\lambda}\subset \Hs_{\lambda}$ be a completely dense subspace for $\M_{\lambda}$. 
Then $\widehat{\bigoplus}_{\lambda \in \Lambda}\mathcal{D}_{\lambda}\subset \Hs$ is a completely dense subspace for $\M$. 
\end{lem}

\begin{proof}
By Lemma \ref{2c-dense for sigma finite case}, for each $\lambda \in \Lambda$, 
there exists an increasing sequence $\{p_{\lambda,n}\}_{n=1}^{\infty}\subset P(\M_{\lambda})$ such that $p_{\lambda,n}\nearrow 1$ (strongly) 
and $\ran{p_{\lambda,n}}\subset \mathcal{D}_{\lambda}$. 
For a finite set $F\subset \Lambda$, define
\eqa{
p_{F,n}&:=\oplus_{\lambda}p_{F,n}^{(\lambda)},\\
p_{F,n}^{(\lambda)}&:=\begin{cases}
p_{\lambda,n} & (\lambda \in F),\\
0 & (\lambda \notin F).
\end{cases}
}
Then we have $p_{F,n}\in P(\M)$ and $\{p_{F,n}\}_{(F,n)}$ is an increasing net of projections. 
Here, we define $(F,n)\leq (F',n')$ by $F\subset F'$ and $n\le n'$. 
It is clear that $p_{F,n}\nearrow  1$ (strongly)
and $\ran{p_{F,n}}\subset \widehat{\bigoplus}_{\lambda \in \Lambda}\mathcal{D}_{\lambda}$. 
Hence $\widehat{\bigoplus}_{\lambda \in \Lambda}\mathcal{D}_{\lambda}$ is completely dense.
\end{proof}

\begin{rem}\label{2why net is necessary} 
Lemma \ref{2c-dense for sigma finite case} does not hold if $\M$ is not countably decomposable. 
We will show a counterexample. 
Let
\[\Hs:=\bigoplus_{t\in \mbb{R}}\ \ell^2(\mbb{N}),\ \ \ \ \M:=\bigoplus_{t\in \mbb{R}}^b\ \M_t,\ \ \ \ 
\mathcal{D}:=\widehat{\bigoplus_{t\in \mbb{R}}}\ \ell^2(\mbb{N})\]
Here, all $\M_t$ are isomorphic copies of some finite von Neumann algebra on $\ell^2(\mbb{N})$. 
By Lemma \ref{2alg-sum of c-dense spaces}, $\mathcal{D}$ is completely dense for $\M$. 
Suppose (2) of Lemma \ref{2c-dense for sigma finite case} holds. 
Then there exists $p_n\in P(\M)$ such that $\ran{p_n}\subset \mathcal{D}$ and $p_n\nearrow 1$ (strongly). 
Represent $p_n$ as $\oplus_{t}p_{t,n}$ $(p_{t,n}\in P(\M_t))$. 
Then we have
\[\bigoplus_{t\in \mbb{R}}\ran{p_{t,n}}=\ran{p_n}\subset \mathcal{D}=\widehat{\bigoplus_{t\in \mbb{R}}}\ \ell^2(\mbb{N}).\]
Therefore for each $n\in \mbb{N}$, there exists a finite set $F_n\subset \mbb{R}$ such that $p_{t,n}=0$ for $t\notin F_n$. 
Since $F:=\Bcu_{n=1}^{\infty}F_n\subset \mbb{R}$ is at most countable, there exists some $t_0\notin F$. 
Choose $\xi^{(t_0)}\in \ell^2(\mbb{N})$ to be a unit vector and $\xi^{(t)}:=0$ $(t\neq t_0)$. 
Then for $\xi=\left\{\xi^{(t)}\right\}_{t\in\mathbb{R}}\in \Hs$, it follows that
\eqa{
||p_n\xi-\xi||^2&=\sum_{t\in \mbb{R}}||p_{t,n}\xi^{(t)}-\xi^{(t)}||^2=||\underbrace{p_{t_0,n}}_{=0}\xi^{(t_0)}-\xi^{(t_0)}||^2\\
&=||\xi^{(t_0)}||^2=1.
}
On the other hand, we have $||p_n\xi-\xi||^2\to 0$, which is a contradiction.
\end{rem}

\begin{prop}[Murray-von Neumann \cite{vNMur}]\label{2countable intersection of c-dense} 
Let $\M$ be a finite von Neumann algebra on a Hilbert space $\Hs$. 
Let $\{\mathcal{D}_i\}_{i=1}^{\infty}\subset \Hs$ be a sequence of completely dense subspaces for $\M$. 
Then the intersection $\displaystyle \Bca_{i=1}^{\infty}\mathcal{D}_i$ is also completely dense. 
\end{prop}

The proof requires some lemmata.

\begin{lem}\label{2countable intersection for sigma finite} 
Proposition \ref{2countable intersection of c-dense} holds if $\M$ is countably decomposable.
\end{lem}

\begin{proof}
Let $\tau$ be a faithful normal tracial state on $\M$. 
By Lemma \ref{2c-dense for sigma finite case}, for each $\varepsilon>0$ and $i\in \mbb{N}$, there exists $p_i\in P(\M)$ such that
$\tau(p_i^{\perp})<\varepsilon/2^i$ and $p_i\Hs \subset \mathcal{D}_i$. 
Put
\[p:=\bigwedge_{i=1}^{\infty}p_i\in P(\M).\]
Then we have 
\eqa{
\tau(p^{\perp})&=\tau \left (\bigvee_{i=1}^{\infty}p_i^{\perp}\right )
\le \sum_{i=1}^{\infty}\tau(p_i^{\perp})\le \sum_{i=1}^{\infty}\frac{\varepsilon}{2^i}=\varepsilon,\\
p\Hs&=\bigcap_{i=1}^{\infty}(p_i\Hs)\subset \Bca_{i=1}^{\infty}\mathcal{D}_i.
}
Hence by Lemma \ref{2c-dense for sigma finite case}, the intersection $\Bca_{i=1}^{\infty}\mathcal{D}_i$ is completely dense.
\end{proof}

\begin{lem}\label{2decomposition of c-dense}
Let $\{(\M_{\lambda},\Hs_{\lambda})\}_{\lambda \in \Lambda}$ be a family of countably decomposable, finite von Neumann algebras. 
Put
\[\M:=\bigoplus_{\lambda \in \Lambda}^b\M_{\lambda},\ \ \ \ \Hs:=\bigoplus_{\lambda \in \Lambda}\Hs_{\lambda}.\]
Let $\mathcal{D}\subset \Hs$ be a completely dense subspace for $\M$. 
Then for each $\lambda \in \Lambda$, there exists some completely dense subspace $\mathcal{D}_{\lambda}\subset \Hs_{\lambda}$ 
for $\M_{\lambda}$ such that
\[\widehat{\bigoplus_{\lambda \in \Lambda}}\ \mathcal{D}_{\lambda}\subset \mathcal{D}.\]
\end{lem}

\begin{proof}
By the definition, there exists an increasing net $\{p_{\alpha}\}_{\alpha \in A}\subset P(\M)$ such that 
$p_{\alpha}\nearrow 1$ (strongly) and $\ran{p_{\alpha}}\subset \mathcal{D}$. 
Let $p_{\alpha}=:\oplus_{\lambda}p_{\lambda,\alpha}$ $(p_{\lambda, \alpha}\in P(\M_{\lambda}))$. 
Then it holds that $p_{\lambda,\alpha}\nearrow 1$ (strongly). 
Put 
\[\mathcal{D}_{\lambda}:=\bigcup_{\alpha \in A}\ran{p_{\lambda,\alpha}}\subset \Hs_{\lambda}.\]
We see that $\mathcal{D}_{\lambda}$ is completely dense for $\M_{\lambda}$. 
It is clear that $\widehat{\bigoplus}_{\lambda \in \Lambda}\mathcal{D}_{\lambda}\subset \mathcal{D}$ holds.
\end{proof}

\begin{proof}[\bf Proof of Proposition \ref{2countable intersection of c-dense}] 
Since $\M$ is finite, there exists a family of countably decomposable, finite von Neumann algebras 
$\{(\M_{\lambda},\Hs_{\lambda})\}_{\lambda \in \Lambda}$ and a unitary operator 
$U:\Hs\to \bigoplus_{\lambda \in \Lambda}\Hs_{\lambda}$ such that $U\M U^*=\bigoplus_{\lambda \in \Lambda}\M_{\lambda}$. 
Put $\mathcal{D}_i':=U\mathcal{D}_i$. 
To prove the proposition, it suffices to prove that $\Bca_{i=1}^{\infty}\mathcal{D}_i'$ 
is completely dense for $\bigoplus_{\lambda \in \Lambda}\M_{\lambda}$. 
By Lemma \ref{2decomposition of c-dense}, for each $i \in \mbb{N}$, 
there exist completely dense subspaces $\mathcal{D}_{\lambda,i}\subset \Hs_{\lambda}$ for $\M_{\lambda}$ 
such that $\mathcal{D}_i'\supset \widehat{\bigoplus}_{\lambda \in \Lambda}\mathcal{D}_{\lambda,i}$. 
Then it follows that 
\[\Bca_{i=1}^{\infty}\mathcal{D}_i'\supset \Bca_{i=1}^{\infty}\left (\widehat{\bigoplus_{\lambda \in \Lambda}}\ 
\mathcal{D}_{\lambda,i}\right )=\widehat{\bigoplus_{\lambda \in \Lambda}}\left (\Bca_{i=1}^{\infty}\ \mathcal{D}_{\lambda,i}\right ).\]
By Lemma \ref{2countable intersection for sigma finite}, $\Bca_{i=1}^{\infty}\mathcal{D}_{\lambda,i}$ is completely dense for $\M_{\lambda}$. 
Therefore by Lemma \ref{2alg-sum of c-dense spaces}, 
$\widehat{\bigoplus}_{\lambda \in \Lambda}\left (\Bca_{i=1}^{\infty}\mathcal{D}_{\lambda,i}\right )$ 
is completely dense for $\bigoplus_{\lambda\in \Lambda}\M_{\lambda}$, which implies $\Bca_{i=1}^{\infty}\mathcal{D}_i'$ 
is also completely dense for $\bigoplus_{\lambda\in \Lambda}\M_{\lambda}$. 
\end{proof}

\begin{prop}[Murray-von Neumann \cite{vNMur}]\label{2M-vN1}
Let $\M$ be a finite von Neumann algebra.
Then for each $X\in \cM$ and a completely dense subspace $\mathcal{D}$ for $\M$, the subspace 
\[\left\{\xi \in \dom{X}\ ;\ X\xi \in \mathcal{D}\right\}\] 
is also completely dense. 
In particular, $\dom{X}$ is completely dense for all $X\in \cM.$
\end{prop}

\begin{proof}
See \cite{vNMur}.
\end{proof}

\begin{prop}[Murray-von Neumann \cite{vNMur}]\label{2No-closed-extension}
Let $\M$ be a finite von Neumann algebra.
\begin{list}{}{}
\item[(1)] Every closed symmetric operator in $\cM$ is self-adjoint.
\item[(2)] There are no proper closed extensions of operators in $\cM$. 
Namely, if $X,\ Y\in \cM$ satisfy $X\subset Y$, then $X=Y$.
\item[(3)] Let $\disp \{X_i\}_{i}$ be a (finite or infinite) sequence in $\cM$. 
The intersection of domains 
\[\mathcal{D}_{\mathcal{P}}:=\bigcap_{p\in \mathcal{P}} \dom{p(X_1,X_1^*,X_2,X_2^*,\cdots )}\]
of all unbounded operators obtained by substituting $\{X_i\}_{i}$ 
into the non-commutative polynomial $p(x_1,y_1,\cdots )$ is completely dense for $\cM$, 
where $\mathcal{P}$ is the set of all non-commutative polynomials with indefinite elements $\{x_i,y_i\}_{i}$.
\end{list}
\end{prop}

\begin{proof}
See \cite{vNMur}.
\end{proof}

\begin{rem}\label{2finiteness}
Murray-von Neumann proved (1) of Proposition \ref{2No-closed-extension} using Cayley transform, 
but there is a simpler proof. We record it here.
\end{rem} 
\begin{proof}
Let $A\in \cM$ be a symmetric operator. 
It is easy to see that $A+i$ is injective. 
Let $A+i=u|A+i|$ be its polar decomposition. 
From the injectivity, $u^*u=P_{\ker{A+i}^{\perp}}=1_{\mathcal{H}}.$ 
Since $\M$ is finite and $uu^*=P_{\cran{A+i}}$, we see that $\cran{A+i}=\mathcal{H}.$ 
On the other hand, since $A$ is closed and symmetric, $\ran{A+i}$ is closed. 
Therefore we obtain $\ran{A+i}=\mathcal{H}$. 
By the same way, it holds that $\ran{A-i}=\mathcal{H}$, which means $A$ is a self-adjoint operator.
\end{proof}

Similarly, we see that for $X\in \cM$ the injectivity of $X$ is equivalent to the density of $\ran{X}$. 

\begin{lem}[Murray-von Neumann \cite{vNMur}]\label{2eta_alg} 
Let $\M$ be a finite von Neumann algebra and $\{X_i\}_i$ be a (finite or infinite) sequence in $\cM$. 
Let
\[p(x_1,y_1,x_2,y_2,\cdots ),\ \ \ \  q(x_1,y_1,x_2,y_2,\cdots ),\ \ \ \  r(x_1,y_1,x_2,y_2,\cdots )\]
be non-commutative polynomials and $p(X_1,X_1^*,X_2,X_2^*,\cdots )$ be an operator obtained by substituting $(x_i,y_i)$ by $(X_i,X_i^*)$.
\begin{list}{}{}
\item[(1)] $p(X_1,X_1^*,X_2,X_2^*,\cdots )$ is a densely defined closable operator on $\Hs$, and 
\[\overline{p(X_1,X_1^*,X_2,X_2^*,\cdots )}\in \cM. \]
\item[(2)] If $p^{(r)}(x_1,y_1,\cdots )=q^{(r)}(x_1,y_1,\cdots )$, then 
\[\overline{p(X_1,X_1^*,\cdots )}=\overline{q(X_1,X_1^*,\cdots)}.\]
Namely, the closure of the substitution of operators depends on a reduced polynomial only.
\item[(3)] If $p(x_1,y_1,\cdots )^{+}=q(x_1,y_1,\cdots )$, then 
\[\left \{\overline{p(X_1,X_1^*,\cdots )}\right \} ^*=\overline{q(X_1,X_1^*,\cdots )}.\]
\item[(4)] If $\alpha p(x_1,y_1,\cdots )=q(x_1,y_1,\cdots )\ (\alpha \in \mathbb{C})$, then 
\[\overline{\alpha \cdot \left \{ \overline{p(X_1,X_1^*,\cdots )}\right \} }=\overline{q(X_1,X_1^*,\cdots )}.\]
\item[(5)] If $p(x_1,y_1,\cdots )+q(x_1,y_1,\cdots )=r(x_1,y_1,\cdots )$, then   
\[\overline{\overline{p(X_1,X_1^*,\cdots) }+\overline{q(X_1,X_1^*,\cdots )}}=\overline{r(X_1,X_1^*,\cdots )}.\]
\item[(6)] If $p(x_1,y_1,\cdots )\cdot q(x_1,y_1,\cdots )=r(x_1,y_1,\cdots )$, then  
\[\overline{\overline{p(X_1,X_1^*,\cdots )}\cdot \overline{q(X_1,X_1^*,\cdots )}}=\overline{r(X_1,X_1^*,\cdots )}.\]
\end{list}
\end{lem}

\begin{proof}
See \cite{vNMur}.
\end{proof}

\begin{rem}\label{3domain is can be zero}
Lemma \ref{2eta_alg} (1) is not trivial.
Indeed one can construct a pair of densely defined closed operators whose intersection of domains is $\{0\}$.
See, e.g., \cite{Kosaki, Schmuedgen2}.
\end{rem}

In summary, we have the following theorem.

\begin{thm}[Murray-von Neumann \cite{vNMur}]\label{2eta_*alg} 
For an arbitrary finite von Neumann algebra $\M$, the set $\cM$ forms a *-algebra of unbounded operators, 
where the algebraic operations are defined by\footnote{
$\overline{\alpha X}$ equals $\alpha X$ when $\alpha \neq 0.$
However, $\dom{\overline{0\cdot X}}=\mathcal{H}\neq \dom{X}.$} 
\begin{align*}
(X,Y)&\mapsto \overline{X+Y}, & (\alpha ,X)&\mapsto \overline{\alpha X},\\
(X,Y)&\mapsto \overline{XY},  &           X&\mapsto X^*.
\end{align*}
\end{thm}

To conclude these preliminaries, we shall show a simple but useful lemma.

\begin{lem}\label{2core of A} 
Let $\M$ be a finite von Neumann algebra, $A$ be an operator in $\cM$. 
If $\mathcal{D}$ is a completely dense subspace of $\mathcal{H}$ contained in $\dom{A}$, 
then it is a core of $A$. That is, $\overline{A|_{\mathcal{D}}}=A$.
\end{lem} 

\begin{proof}
From the complete density of $\mathcal{D}$, there exists an increasing net of closed subspaces $\{M_{\alpha}\}_{\alpha}$ 
of $\mathcal{H}$ with $P_{\alpha}:=P_{M_{\alpha}}\in \M$ such that  \[\mathcal{D}_0:=\Bcu_{\alpha}M_{\alpha}\subset \mathcal{D}\]
is dense in $\mathcal{H}$.
Define $A_0:=A|_{\mathcal{D}_0}$.
Take an arbitrary $u\in U(\M').$ 
Let $\xi \in \mathcal{D}_0=\dom{A_0}$, so that there is some $\alpha$ such that $\xi \in M_{\alpha}$. 
Then we have
\eqa{
uA_0\xi&=uA\xi=Au\xi=AuP_{\alpha}\xi\\
&=AP_{\alpha}u\xi=A_0P_{\alpha}u\xi\ \\
&=A_0u\xi.
} 
Therefore $uA_0\subset A_0u$ holds. 
Since $u\in U(\M')$ is arbitrary, we have $uA_0u^*=A_0$. 
Taking the closure of both sides, we see that $\overline{A_0}=u\overline{A_0}u^*$. 
This means $\overline{A_0}\in \cM$. 
Therefore, it follows that
\[\overline{A_0}= \overline{A|_{\mathcal{D}}}\subset \overline{A}=A\]
Therefore by Proposition \ref{2No-closed-extension}, we have $\overline{A_0}=A.$
\end{proof}

\subsection{Converse of Murray-von Neumann's Result}

The converse of Theorem \ref{2eta_*alg} is also true.
We shall give a proof here.

\begin{lem}\label{23lem}
Let $\M$ be a von Neumann algebra acting on a Hilbert space $\Hs$.
Assume that, for all $A$, $B\in\M$, the domains $\dom{A+B}$ and $\dom{AB}$ are dense in $\Hs$.
Then $A+B$ and $AB$ are densely defined closable operators on $\Hs$ 
and the closures $\overline{A+B}$ and $\overline{AB}$ are affiliated with $\M$
for all $A$, $B\in\M$.
\end{lem}

\begin{proof}
By the assumption, $A+B$ is densely defined and
\begin{equation*}
(A+B)^* \supset A^*+B^*.
\end{equation*}
Since the right hand side is densely defined, $A+B$ is closable.
As same as the above, we see that $AB$ is closable.
Affiliation property is easy.
\end{proof}

\begin{rem}
Let $\M$ be a von Neumann algebra.
It is easy to check that $\alpha A$ ($\alpha\in\mathbb{C}$, $A\in\cM$) is always densely defined closable 
and its closure $\overline{\alpha A}$ is affiliated with $\M$.
Moreover $\cM$ is closed with respect to the involution $A\mapsto A^*$.
\end{rem}

\begin{thm}\label{23converse}
Let $\M$ be a von Neumann algebra acting on a Hilbert space $\Hs$.
Assume that, for all $A$, $B\in\M$, the domains $\dom{A+B}$ and $\dom{AB}$ are dense in $\Hs$.
If the set $\cM$ forms a *-algebra with respect to the sum $\overline{A+B}$, 
the scalar multiplication $\overline{\alpha{A}}$ ($\alpha\in\mathbb{C}$), the multiplication $\overline{AB}$ and the involution $A^*$,
then $\M$ is a finite von Neumann algebra.
\end{thm}

\begin{proof}
{\bf Step 1.} We first show that all closed symmetric operators affiliated with $\M$ are automatically self-adjoint.
Let $A$ be a closed symmetric operator affiliated with $\M$.
Define operators $B\in\cM$ and $C\in\cM$ as 
\begin{equation*}
B:=\frac{1}{2}\left(\overline{A+A^*}\right), \ \ \ \ C:=\frac{1}{2i}\left(\overline{A-A^*}\right),
\end{equation*} 
then $B$ and $C$ are self-adjoint and $A=\overline{B+iC}$ holds because $\cM$ is a *-algebra.
Since $A$ is symmetric, we see that 
\begin{equation*}
C \supset \frac{1}{2i}\left(A-A^*\right) \supset \frac{1}{2i}\left(A-A\right) = 0|_{\dom{A}}.
\end{equation*}
By taking the closure, we obtain $C=0$.
Hence $A=B$ is self-adjoint.\\

{\bf Step 2.} We shall prove that $\M$ is finite.
Let $v$ be an arbitrary isometry in $\M$.
By the Wold decomposition, there exists a unique projection $p\in\M$
such that $\ran{p}$ reduces $v$, $s:=v|_{\ran{p}}\in\M_{p}$ is a unilateral shift operator 
and $u:=v|_{\ran{p^{\bot}}}\in\M_{p^{\bot}}$ is unitary.
It is easy to see that 
\begin{equation*}
\ker{1-s}=\{0\}, \ \ \ \ \ker{1-s^{*}}=\{0\},
\end{equation*}
so that we can define the closed symmetric operator $T$ on $\ran{p}$ as follows:
\begin{equation*}
T:=i(1+s)(1-s)^{-1}.
\end{equation*}
We immediately see that $T$ is affiliated with the von Neumann algebra $\M_{p}$.  
Define the operator $A$ on $\Hs=\ran{p}\bigoplus\ran{p^{\bot}}$ by
\begin{equation*}
A:=T\oplus 0_{\ran{p^{\bot}}},
\end{equation*}
then $A$ is a closed symmetric operator and it is affiliated with $\M$.
From Step 1., $A$ is self-adjoint, so that $T$ is also self-adjoint. 
Since the Cayley transform of a self-adjoint operator is always unitary 
and the Cayley transform of $T$ is $s$, $s$ is unitary.
This implies $p=0$ because a unilateral shift operator admits 
no non-zero reducing closed subspace on which it is unitary.  
Hence $v=u$ is unitary. 
\end{proof}

\section{Topological Structures of $\cM$}

In this section we investigate topological properties of $\cM$.
We need these results in the next section.
We first endow $\cM$ with two topologies, called the strong resolvent topology
and the strong exponential topology.
The former is (unbounded) operator theoretic and the latter is Lie theoretic.
To show that these two topologies do coincide and $\cM$ forms a complete topological *-algebra with respect to them,
we introduce another topology, called the $\tau$-measure topology which originates from the noncommutative integration theory.
They seem quite different to each other, but in fact they also coincide.
The main topic of the present section is to study correlations between them.

\subsection{Strong Resolvent Topology}

First of all, we define the topology called the strong resolvent topology on the suitable subset of densely defined closed operators.
Let $\Hs$ be a Hilbert space.
We call a densely defined closed operator $A$ on $\Hs$ belongs to the {\it resolvent class} $\rc{\Hs}$
if $A$ satisfies the following two conditions:

\begin{list}{}{}
\item[(RC.1)] there exist self-adjoint operators $X$ and $Y$ on $\Hs$ such that 
the intersection $\dom{X}\cap\dom{Y}$ is a core of $X$ and $Y$,
\item[(RC.2)] $A=\overline{X+iY}$, \ $A^{*}=\overline{X-iY}$.
\end{list}

Note that (RC.1) implies $\dom{X}\cap\dom{Y}$ is dense, so $X+iY$ and $X-iY$ are closable.
Thus $\overline{X+iY}$ and $\overline{X-iY}$ are always defined.
Furthermore, we have
\begin{equation*}
\frac{1}{2}(A+A^{*})= \frac{1}{2}(\overline{X+iY} + \overline{X-iY}) \supset X|_{\dom{X}\cap\dom{Y}}.
\end{equation*}
Since $A+A^{*}$ is closable and by (RC.1), we get
\begin{equation*}
\frac{1}{2}\overline{A+A^{*}} \supset X.
\end{equation*}
As $X$ is self-adjoint, $X$ has no non-trivial symmetric extension, we have
\begin{equation*}
\frac{1}{2}\overline{A+A^{*}} = X.
\end{equation*}
Therefore, $X$ is uniquely determined.
As same as the above, $Y$ is also unique and 
\begin{equation*}
\frac{1}{2i}\overline{A-A^{*}} = Y.
\end{equation*}
We denote 
\begin{equation*}
\re{A} := X = \frac{1}{2}\overline{A+A^{*}}, \ \ \ \ \im{A} := Y = \frac{1}{2i}\overline{A-A^{*}}.
\end{equation*}
Also note that bounded operators and (possibility unbounded) normal operators belong to $\rc{\Hs}$.

Now we endow $\rc{\Hs}$ with the {\it strong resolvent topology} (SRT for short), 
the weakest topology for which the following mappings
\begin{equation*}
\rc{\Hs} \ni A \longmapsto \{\re{A}-i\}^{-1} \in (\mathfrak{B}(\Hs), SOT)
\end{equation*}
and
\begin{equation*}
\rc{\Hs} \ni A \longmapsto \{\im{A}-i\}^{-1} \in (\mathfrak{B}(\Hs), SOT)
\end{equation*}
are continuous.
Thus a net $\{A_{\alpha}\}_{\alpha}$ in $\rc{\Hs}$ converges to $A \in \rc{\Hs}$ 
with respect to the strong resolvent topology if and only if
\begin{equation*} 
\{\re{A_{\alpha}}-i\}^{-1}\xi\rightarrow \{\re{A}-i\}^{-1}\xi, \ \ \ \ \{\im{A_{\alpha}}-i\}^{-1}\xi \rightarrow \{\im{A}-i\}^{-1}\xi,
\end{equation*}
for each $\xi\in\Hs$.
This topology is well-studied in the field of unbounded operator theory and suitable for the operator theoretical study.
We denote the system of open sets of the strong resolvent topology by $\mathcal{O}_{{\rm{SRT}}}$.

Let $\M$ be a finite von Neumann algebra on a Hilbert space $\Hs$.
We shall show that $\cM$ is a closed subset of the resolvent class $\rc{\Hs}$.
This fact follows from Proposition \ref{2countable intersection of c-dense}, Theorem \ref{2eta_*alg}, 
Lemma \ref{2core of A} and the following lemmata.

\begin{lem}\label{3subset}
Let $\M$ be a finite von Neumann algebra on a Hilbert apace $\Hs$, $A$ be in $\cM$.
Then there exist unique self-adjoint operators $B$ and $C$ in $\cM$ such that
\begin{equation*}
A=\overline{B+iC}.
\end{equation*} 
\end{lem}

\begin{proof}
Put 
\begin{equation*}
B:=\frac{1}{2}\overline{A+A^*}, \ \ \ \ C:=\frac{1}{2i}\overline{A-A^*}.
\end{equation*}
Applying Proposition \ref{2countable intersection of c-dense}, $\dom{B}$ and $\dom{C}$ are dense in $\Hs$.
Hence $B$ and $C$ are closed symmetric operators affiliated with $\M$.
By Proposition \ref{2No-closed-extension}, in fact, $B$ and $C$ are self-adjoint. 
As $\cM$ is a *-algebra, we have
\begin{equation*}
A=\overline{B+iC}.
\end{equation*}
\end{proof}

\begin{lem}\label{3closed}
Let $\M$ be a finite von Neumann algebra. 
Then $\cM$ is closed with respect to the strong resolvent topology.
\end{lem}

\begin{proof}
Let $\{A_{\alpha}\}_{\alpha} \subset \cM$ be a net converging to $A \in \rc{\Hs}$ with respect to the strong resolvent topology.
Then, for all $u \in U(\M')$, we have
\begin{align*}
\{u\re{A}u^{*}-i\}^{-1} &= u\{\re{A}-i\}^{-1}u^{*} = \slim{\alpha} u\{\re{A_{\alpha}}-i\}^{-1}u^{*} \\ 
&= \slim{\alpha} \{u\re{A_{\alpha}}u^{*}-i\}^{-1} = \slim{\alpha} \{\re{A_{\alpha}}-i\}^{-1} \\
&=\{\re{A}-i\}^{-1}.
\end{align*}
This implies $\re{A}$ belongs to $\cM$.
As same as the above, we obtain $\im{A} \in \cM$.
Thus so is $A =\overline{\re{A}+\im{A}}$.
\end{proof}

\begin{rem}\label{3not linear}
In general, the strong resolvent topology is not linear.
Indeed, there exists sequences $\{A_n\}_{n=1}^{\infty}$, $\{B_n\}_{n=1}^{\infty}$ of self-adjoint operators
and self-adjoint operators $A$, $B$ 
such that the following conditions hold:
\begin{list}{}{}
\item[(1)] $\{A_n\}_{n=1}^{\infty}$ and $\{B_n\}_{n=1}^{\infty}$ converge to $A$ and $B$ 
in the strong resolvent topology, respectively.
\item[(2)] $A_n+B_n$ is essentially self-adjoint for each $n\in\mathbb{N}$.
\item[(3)] $A+B$ is essentially self-adjoint.
\item[(4)] $\left\{\overline{A_n+B_n}\right\}_{n=1}^{\infty}$ converges to some self-adjoint operator $C$ 
in the strong resolvent topology, but $C\not=\overline{A+B}$.
\end{list}
For the details, see \cite{Simon}.
However, as we see in the sequel, the strong resolvent topology is linear on $\cM$.
\end{rem}

The next lemma is important in our discussion.

\begin{lem}\label{Ametrizable}
Let $\M$ be a finite von Neumann algebra acting on a Hilbert space $\Hs$.
Then the following are equivalent:
\begin{list}{}{}
\item[(1)] $\M$ is countably decomposable,
\item[(2)] $(\cM, SRT)$ is metrizable as a topological space,
\item[(3)] $(\cM, SRT)$ satisfies the first countability axiom.
\end{list}
\end{lem}

\begin{proof}
$(1)\Rightarrow (2)$. Let $\{\xi_k\}_{k}$ be a countable separating family of unit vectors in $\Hs$ for $\M$.
For each $A$, $B \in \cM$, we define
\begin{align*}
d(A,B) &:= \sum_{k}\frac{1}{2^k}\|\{\re{A}-i\}^{-1}\xi_k-\{\re{B}-i\}^{-1}\xi_k\| \\
&\ \ \ \ + \sum_{k}\frac{1}{2^k}\|\{\im{A}-i\}^{-1}\xi_k-\{\im{B}-i\}^{-1}\xi_k\|.
\end{align*}
It is easy to see that the above $d$ is a distance function on the space $\cM$, 
and the topology induced by the distance function $d$ coincide with the strong resolvent topology on $\cM$. 

$(2)\Rightarrow (3)$ is trivial.

$(3)\Rightarrow (1)$. Let $S \subset P(\M)$ be a family of mutually orthogonal nonzero projections in $\M$.
Since $(\cM, SRT)$ satisfies the first countability axiom, 
the origin $0 \in \cM$ has a countable fundamental system of neighborhoods $\{V_k\}_k$.
Put 
\begin{equation*}
S_k:=\{p\in S\ ;\ p\notin V_k\},
\end{equation*}
then $S=\bigcup_k S_k$.
This follows from the Hausdorff property of the strong resolvent topology.
Next we show that each $S_k$ is a finite set.
Suppose $S_k$ is an infinite set, 
then we can take a countably infinite subset $\{p_n\ ;\ n\in\mathbb{N}\}$ of $S_k$.
Define
\begin{equation*}
p:= \slim{N \rightarrow \infty} \sum_{n=1}^{N}p_n.
\end{equation*}
For every $\xi\in\Hs$ we see that 
\begin{align*}
\|p_n\xi\| &= \|\sum_{i=1}^{n}p_n\xi-\sum_{i=1}^{n-1}p_n\xi\| \\
&\leq \|\sum_{i=1}^{n}p_n\xi-p\xi\| + \|p\xi-\sum_{i=1}^{n-1}p_n\xi\| \\
&\longrightarrow 0.
\end{align*}
Thus $p_n$ converges strongly to $0$.
By Lemma \ref{AcoreSRT}, this implies $p_n$ converges to $0$ with respect to the strong resolvent topology.
Hence there exists a number $n\in\mathbb{N}$ such that $p_n\in V_k$.
This is a contradiction to $p_n\in S_k$.
Therefore $S_k$ is a finite set.
From the above arguments, we conclude that $S=\bigcup_k S_k$ is at most countable.
\end{proof}

\begin{rem}\label{Auniform_metrizable}
As we see in the sequel, $(\cM, SRT)$ is a Hausdorff topological linear space.
Thus in the case that $\M$ satisfies conditions (1), (2) or (3) of Lemma \ref{Ametrizable}, 
$(\cM, SRT)$ is metrizable with a translation invariant distance function.
In particular, it is also metrizable as a uniform space.  
\end{rem}

Finally, we state one lemma.

\begin{lem}\label{3mt=sot}
Let $\M$ be a finite von Neumann algebra acting on a Hilbert space $\Hs$.
Then the strong resolvent topology and the strong operator topology coincide on 
the closed unit ball $\M_1$.
\end{lem}

\begin{proof}
Note that if a von Neumann algebra is finite, then the involution is strongly continuous on the closed unit ball.
The lemma follows immediately from this fact, Lemma \ref{AcoreSRT} and Lemma \ref{Asrt-sot}.
\end{proof}

See Appendix B for more informations of the strong resolvent topology.

\subsection{Strong Exponential Topology}

Next we introduce a Lie theoretic topology on $\cM$.
Let $\Hs$ be a Hilbert space.
For each $A\in\rc{\Hs}$, each SOT-neighborhood $V$ at $1\in\mathfrak{B}(\Hs)$ and each compact set $K$ of $\mathbb{R}$, 
we define $W(A;V,K)$ the subset of $\rc{\Hs}$ by

\begin{align*}
W(A;V,K) := \left\{B\in\rc{\Hs}\ ;\ 
\begin{array}{l} e^{-it\re{A}}e^{it\re{B}}\in V,\\
e^{-it\im{A}}e^{it\im{B}}\in V,\ \forall t \in K.
\end{array} \right\},
\end{align*}
then $\{W(A;V,K)\}_{A,V,K}$ is a fundamental system of neighborhoods on $\rc{\Hs}$.
We denote the system of open sets of the topology induced by this fundamental system of neighborhoods by $\mathcal{O}_{{\rm{SET}}}$,
and call this topology  the {\it strong exponential topology} (SET for short).
Note that a net $\{A_{\lambda}\}_{\lambda\in\Lambda}$ in $\rc{\Hs}$ converges to $A\in\rc{\Hs}$
in the strong exponential topology if and only if 
\begin{equation*}
e^{it\re{A_{\lambda}}}\xi\longrightarrow e^{it\re{A}}\xi, \ \ \ \ e^{it\im{A_{\lambda}}}\xi\longrightarrow e^{it\im{A}}\xi,
\end{equation*}
for each $\xi\in\Hs$, uniformly for $t$ in any finite interval.
This topology is important from the viewpoint of Lie theory.
Indeed it can be defined by the unitary group $U(\Hs)$ only.
Before stating the main theorem in this section, we study relations between the strong resolvent topology
and the strong exponential topology.

\begin{lem}\label{ASETmetrizable}
Let $\M$ be a countably decomposable finite von Neumann algebra acting on a Hilbert space $\Hs$.
Then $(\cM, SET)$ is metrizable as a topological space.
\end{lem}

\begin{proof}
Let $\{\xi_n\}_{n}$ be a countable separating family of unit vectors in $\Hs$ for $\M$.
For each $A$, $B \in \cM$ we define
\begin{align*}
d(A,B) &:= \sum_{n}\sum_{m=1}^{\infty}\frac{1}{2^{n+m}}\sup_{t\in [-m,m]}\|e^{it\re{A}}\xi_n-e^{it\re{B}}\xi_n\| \\
&\ \ \ \ + \sum_{n}\sum_{m=1}^{\infty}\frac{1}{2^{n+m}}\sup_{t\in [-m,m]}\|e^{it\im{A}}\xi_n-e^{it\im{B}}\xi_n\|.
\end{align*}
It is easy to see that the above $d$ is a distance function on the space $\cM$, 
and the topology induced by the distance function $d$ coincide with the strong exponential topology on $\cM$.
\end{proof}

\begin{lem}\label{3SRT=SET}
Let $\M$ be a countably decomposable finite von Neumann algebra.
Then the strong resolvent topology and the strong exponential topology coincide on $\cM$.
\end{lem}

\begin{proof}
This follows immediately from Lemma \ref{Ametrizable}, Lemma \ref{ASETmetrizable} and Lemma \ref{ASRT=exp}.
\end{proof}

\begin{rem}
Similar to the above argument, one can prove that 
the strong resolvent topology and the strong exponential topology coincide on $\rc{\Hs}$ 
if the Hilbert space $\Hs$ is separable.
But the authors do not know whether this is true or not if $\Hs$ is not separable.
However we can show the following theorem.
\end{rem}

The next is the main theorem in this section. 

\begin{thm}\label{3cta}
Let $\M$ be a finite von Neumann algebra acting on a Hilbert space $\Hs$.
Then $\cM$ is a complete topological *-algebra with respect to 
the strong resolvent topology. 
Moreover the strong resolvent topology and the strong exponential topology coincide on $\cM$.
\end{thm}

Throughout this section, we prove the above theorem.

\subsection{$\tau$-Measure Topology}
We first prove Theorem \ref{3cta} in a countably decomposable von Neumann algebra case.
In this case, we can use the nonmmutative integration theory thanks to a faithful normal tracial state.
We shall introduce the $\tau$-measure topology.
Let $\M$ be a countably decomposable finite von Neumann algebra acting on a Hilbert space $\Hs$. 
Fix a faithful normal tracial state $\tau$ on $\M$.
The {\it $\tau$-measure topology} (MT for short) on $\cM$ is the linear topology whose fundamental system of neighborhoods at 0 is
given by
\begin{align*}
N(\eps,\delta) := \left\{A \in \cM \ ; \begin{array}{l} \ {\rm{there \ exists \ a \ projection}} \ p\in\M \\
\ {\rm{such \ that}} \ \|Ap\|<\eps, \ \tau(p^{\perp})<\delta
\end{array} \right\},
\end{align*}
where $\eps$ and $\delta$ run over all strictly positive real numbers.
It is known that $\cM$ is a complete topological *-algebra with respect to this topology \cite{Nelson}.
We denote the system of open sets with respect to the $\tau$-measure topology by $\mathcal{O}_{\tau}$.
Note that the $\tau$-measure topology satisfies the first countability axiom.

\begin{rem}
In this context, the operators in $\cM$ are sometimes called {\it $\tau$-measurable operators} \cite{FackKosaki}.
\end{rem}

Thus there are two topologies on $\cM$, the strong resolvent topology 
and the $\tau$-measure topology.
It seems that these two topologies are quite different.
However, in fact, they coincide on $\cM$, i.e.,

\begin{lem}\label{3srt=mt}
Let $\M$ be a countably decomposable finite von Neumann algebra acting on a Hilbert space $\Hs$.
Then the strong resolvent topology and the $\tau$-measure topology coincide on $\cM$.
In particular, $\cM$ forms a complete topological *-algebra 
with respect to the strong resolvent topology.
Moreover the $\tau$-measure topology is independent of 
the choice of a faithful normal tracial state $\tau$.
\end{lem}

This lemma is the first step to our goal.

\subsection{Almost Everywhere Convergence}

To prove Lemma \ref{3srt=mt}, we define almost everywhere convergence.
Let $\M$ be a countably decomposable finite von Neumann algebra on a Hilbert space $\Hs$.

\begin{dfn}\label{3cae}
A sequence $\{A_n\}_{n=1}^{\infty} \subset \cM$ {\it converges almost everywhere} 
(with respect to $\M$) to $A\in\cM$ if
there exists a completely dense subspace $\mathcal{D}$ such that 
\begin{list}{}{}
\item[(i)] $\mathcal{D} \subset \bigcap_{n=1}^{\infty}\dom{A_n}\cap\dom{A}$,
\item[(ii)] $A_n\xi$ converges to $A\xi$ for each $\xi \in \mathcal{D}$.
\end{list}
\end{dfn}

We shall investigate the relations between the almost everywhere convergence 
and the other topologies.

\begin{lem}\label{3subseq}
Let $\{A_n\}_{n=1}^{\infty} \subset \cM$ be a sequence, $A\in\cM$.
Suppose $A_n$ converges to $A$ in the $\tau$-measure topology, 
then there exists a subsequence $\{A_{n_k}\}_{k=1}^{\infty}$ of $\{A_n\}_{n=1}^{\infty}$ such that 
$A_{n_k}$ converges almost everywhere to $A$.
\end{lem}

\begin{proof}
For all $j \in \mathbb{N}$, we can take $n_j \in \mathbb{N}$ and $p_j \in P(\M)$ which satisfy 
the following conditions:
\begin{equation*}
\|\overline{(A_{n_{j}}-A)}p_j\| < 1/2^j, \ \ \ \ \tau(p_j^{\perp})<1/2^j, \ \ \ \ n_j < n_{j+1}. 
\end{equation*} 
Put $p:=\bigvee_{l=1}^{\infty} \bigwedge_{k=l}^{\infty} p_k \in P(\M)$, 
then $\ran{p} = \overline{\bigcup_{l=1}^{\infty} \bigcap_{k=l}^{\infty} \ran{p_k}}$.
On the other hand, 
\begin{align*}
\tau(p^{\perp}) &= \lim_{l \rightarrow \infty} 
\tau\left(\bigvee_{k=l}^{\infty}p_k^{\perp}\right)
\leq \lim_{l \rightarrow \infty} \sum_{k=l}^{\infty}\tau(p_k^{\perp}) \\
&\leq \lim_{l \rightarrow \infty} \sum_{k=l}^{\infty} \frac{1}{2^k} = 0.
\end{align*}
Therefore, $\Hs = \ran{p} = \overline{\bigcup_{l=1}^{\infty} \bigcap_{k=l}^{\infty} \ran{p_k}}$.
This implies
\begin{equation*}
\mathcal{D}_0 := \bigcup_{l=1}^{\infty} \bigcap_{k=l}^{\infty} \ran{p_k}
\end{equation*}
is completely dense.
Let $\mathcal{D}_1$ be the intersection of the domains of all non-commutative polynomials of operators $\{A_{n_k},\ A,\ p_k\}_{k=1}^{\infty}$, where we do not take closure for each non-commutative polynomial of operators.
Then $\mathcal{D}_1$ is also completely dense and so is 
$\mathcal{D}:= \mathcal{D}_0 \cap \mathcal{D}_1$.
Take $\xi \in \mathcal{D}$, then there exists $k_0 \in \mathbb{N}$ such that $\xi \in \bigcap_{k=k_0}^{\infty}\ran{p_k}$.
Consequently, for all $k\geq k_0$, we get
\begin{equation*}
\xi = p_k \xi, \ \ \ \ p_k\xi \in \dom{A}\cap\dom{A_k}, \ \ \ \ \xi \in \dom{A}\cap\dom{A_k},
\end{equation*}
and 

\begin{align*}
\|\overline{(A_{n_{k}}-A)}\xi\| &= \|\overline{(A_{n_{k}}-A)}p_j\xi\| \\
&\leq \|\overline{(A_{n_{j}}-A)}p_j\| \cdot \|\xi\| \\
&\leq \frac{1}{2^k} \cdot \|\xi\| \longrightarrow 0.
\end{align*}
Thus $A_{n_k}$ converges almost everywhere to $A$.
\end{proof}

\begin{lem}\label{3almostsrt}
Let $\{A_n\}_{n=1}^{\infty}$ be a sequence in $\cM$ converging almost everywhere to $A\in\cM$.
Suppose $\{{A_n}^{*}\}_{n=1}^{\infty}$ also converges almost everywhere to $A^{*}$, 
then $\{A_n\}_{n=1}^{\infty}$ converges to $A$ in the strong resolvent topology.
\end{lem}

\begin{proof}
It is easy to check that $\re{A_n}$ and $\im{A_n}$ converge almost every where to $\re{A}$ and $\im{A}$, respectively.
Applying Lemma \ref{AcoreSRT} and Lemma \ref{2core of A} to $\re{A_n}$ and $\im{A_n}$, 
we see that $\re{A_n}$ and $\im{A_n}$ converge to $\re{A}$ and $\im{A}$ in the strong resolvent topology, respectively.
This implies $\{A_n\}_{n=1}^{\infty}$ converges to $A$ in the strong resolvent topology.
\end{proof}

The following is well-known:

\begin{lem}\label{3metricconv}
Let $X$ be a metric space,  $\{x_n\}_{n=1}^{\infty} \subset X$ be a sequence, $x\in X$.
Suppose for each subsequence $\{x_{n_k}\}_{k=1}^{\infty}$ of $\{x_n\}_{n=1}^{\infty}$ has 
a subsequence $\{x_{n_{k_l}}\}_{l=1}^{\infty}$ of $\{x_{n_k}\}_{k=1}^{\infty}$ which converges to $x$, 
then $x_n$ converges to $x$.
\end{lem}

\subsection{Proof of Lemma \ref{3srt=mt}}

We shall start to prove Lemma \ref{3srt=mt}.
We prove that the system of open sets of the strong resolvent topology $\mathcal{O}_{{\rm{SRT}}}$ 
and the system of open sets of the $\tau$-measure topology $\mathcal{O}_{\tau}$ coincide on $\cM$.
Let $\{A_n\}_{n=1}^{\infty} \subset \cM$ be a sequence, $A\in\cM$. \\
\\
$\mathcal{O}_{{\rm{SRT}}} \subset \mathcal{O}_{\tau}$:\ \ Suppose that 
$\{A_n\}_{n=1}^{\infty}$ converges to $A$ in the $\tau$-measure topology.
Let $\{A_{n_k}\}_{k=1}^{\infty}$ be an arbitrary subsequence of $\{A_n\}_{n=1}^{\infty}$.
By Proposition \ref{3subseq}, there exists a subsequence $\{A_{n_{k_l}}\}_{l=1}^{\infty}$ 
of $\{A_{n_k}\}_{k=1}^{\infty}$ such that 
$\{A_{n_{k_l}}\}_{l=1}^{\infty}$ and $\{{A_{n_{k_l}}}^{*}\}_{l=1}^{\infty}$ 
converge almost everywhere to $A$ and $A^{*}$, respectively.
Applying Lemma \ref{3almostsrt}, $A_{n_{k_l}}$ converges to $A$ in the strong resolvent topology.
This implies $A_n$ converges to $A$ in the strong resolvent topology, by Lemma \ref{3metricconv}.
Thus we get $\mathcal{O}_{{\rm{SRT}}} \subset \mathcal{O}_{\tau}$. \\
\\
$\mathcal{O}_{\tau} \subset \mathcal{O}_{{\rm{SRT}}}$:\ \ Suppose that 
$\{A_n\}_{n=1}^{\infty}$ converges to $A$ with respect to the strong resolvent topology.
First we consider the case that $A_n$ and $A$ are self-adjoint.
Let $|A_n| =: \int_0^{\infty}\lambda dE_n(\lambda)$ and 
$|A| =: \int_0^{\infty}\lambda dE(\lambda)$ be spectral resolutions of $|A_n|$ and $|A|$, respectively.
Fix an arbitrary positive number $\eps >0$.
It is clear that $\slim{\lambda \rightarrow \infty} E([0, \lambda))=1$,
so there exists a positive number $\Lambda>0$ such that $\tau(E([0, \Lambda))^{\perp})<\eps$, 
where we can take $\Lambda>0$ which is not a point spectrum of $|A|$.
Indeed, self-adjoint operators have at most countable point spectra, as $\M$ is countably decomposable.
Next we define a continuous function $\phi$ on $\mathbb{R}$ as follows:
\[
\phi(\lambda) := \left\{
\begin{array}{ll}
0 & \mbox{ if \ \ $\lambda \leq -2\Lambda$, } \\
-\lambda-2\Lambda & \mbox{ if\ \  $-2\Lambda \leq \lambda \leq -\Lambda$, } \\
\lambda & \mbox{ if\ \  $-\Lambda \leq \lambda \leq \Lambda$, } \\
-\lambda+2\Lambda & \mbox{ if \ \ $\Lambda \leq \lambda \leq 2\Lambda$, } \\
0 & \mbox{ if \ \ $2\Lambda \leq \lambda$. }
\end{array}
\right.
\]
Let $|\phi(A_n)-\phi(A)|=:\int_{0}^{\infty}\lambda dF_n(\lambda)$ be a spectral resolution of $|\phi(A_n)-\phi(A)|$,
$e$ be a spectral measure of $A$.
Note that $E([0, \Lambda))=e((-\Lambda, \Lambda))$.
For each $\xi\in\Hs$, 
\begin{align*}
\nai{\xi}{AE([0 ,\Lambda))\xi} &= \int_{(-\Lambda, \Lambda)}\lambda d\nai{\xi}{e(\lambda)\xi} \\
&= \int_{(-\Lambda, \Lambda)}\phi(\lambda) d\nai{\xi}{e(\lambda)\xi} \\
&= \int_{\mathbb{R}}\phi(\lambda) d\nai{\xi}{e(\lambda)E([0, \Lambda))\xi} \\
&= \nai{\xi}{\phi(A)E([0 ,\Lambda))\xi}.
\end{align*}
Thus we have $AE([0 ,\Lambda))\xi = \phi(A)E([0 ,\Lambda))\xi$.
Similar to the above argument, we get $A_nE_n([0 ,\Lambda))\xi = \phi(A_n)E_n([0 ,\Lambda))\xi$.
Therefore, for all $\xi\in\Hs$, we see that
\begin{align*}
&\|(A_n-A)\{E_n([0, \Lambda))\wedge E([0, \Lambda))\wedge F_n([0, \eps))\}\xi\|^2 \\
&= \|\{\phi(A_n)-\phi(A)\}\{E_n([0, \Lambda))\wedge E([0, \Lambda))\wedge F_n([0, \eps))\}\xi\|^2 \\
&= \||\phi(A_n)-\phi(A)|\{E_n([0, \Lambda))\wedge E([0, \Lambda))\wedge F_n([0, \eps))\}\xi\|^2 \\
&= \int_{[0, \eps)}\lambda^2 d\|F_n(\lambda)\{E_n([0, \Lambda))\wedge E([0, \Lambda))\wedge F_n([0, \eps))\}\xi\|^2 \\
&\leq \eps^2\|\xi\|^2.
\end{align*}
This implies 
\begin{equation*}
\|\overline{(A_n-A)}\{E_n([0, \Lambda))\wedge E([0, \Lambda))\wedge F_n([0, \eps))\}\| \leq \eps.
\end{equation*}
On the other hand, 
\begin{align*}
&\tau(\{E_n([0, \Lambda))\wedge E([0, \Lambda))\wedge F_n([0, \eps))\}^{\perp}) \\
&\leq \tau(E_n([0, \Lambda))^{\perp}) + \tau(E([0, \Lambda))^{\perp}) + \tau(F_n([0, \eps))^{\perp}) \\
&\leq \tau(E_n([0, \Lambda))^{\perp}) + \eps + \tau(F_n([0, \eps))^{\perp}).
\end{align*}
By Lemma \ref{Abf}, $|A_n|$ converges to $|A|$ in the strong resolvent topology, 
as the function 
\begin{equation*}
\mathbb{R} \ni \lambda \longmapsto (|\lambda|-i)^{-1} \in \mathbb{C} 
\end{equation*}
is bounded continuous.
By Lemma \ref{Aspec}, 
\begin{align*}
E_n([0, \Lambda)) &= E_n((-1, \Lambda)) \\
&\xrightarrow{SOT} E((-1, \Lambda)) = E([0, \Lambda)).
\end{align*}
Thus for all sufficiently large number $n\in\mathbb{N}$, 
\begin{align*}
\tau(E_n([0, \Lambda))^{\perp}) &= \tau(E([0, \Lambda))^{\perp}) + \tau(E([0, \Lambda))-E_n([0, \Lambda))) \\
&\leq \eps + \eps = 2\eps.
\end{align*}
Furthermore, by Lemma \ref{Abf}, $\phi(A_n)$ converges strongly to $\phi(A)$.
We obtain that for each $\xi\in\Hs$,
\begin{equation*}
\||\phi(A_n)-\phi(A)|\xi\| = \|\{\phi(A_n)-\phi(A)\}\xi\| \longrightarrow 0.
\end{equation*}
Applying Lemma \ref{AcoreSRT} and Lemma \ref{Aspec} to $|\phi(A_n)-\phi(A)|$, we see that
\begin{equation*}
F_n([0, \eps)) = F_n((-1, \eps)) \xrightarrow{SOT} 1.
\end{equation*}
Hence, for all sufficiently large numbers $n\in\mathbb{N}$, $\tau(F_n([0, \eps))^{\perp}) < \eps$.
Thus, for all sufficiently large numbers $n\in\mathbb{N}$, we have
\begin{equation*}
\tau(\{E_n([0, \Lambda))\wedge E([0, \Lambda))\wedge F_n([0, \eps))\}^{\perp}) \leq 4\eps.
\end{equation*} 
From the above argument, we conclude that $A_n$ converges to $A$ in the $\tau$-measure topology.
In a general case, self-adjoint operators $\re{A_n}$ and $\im{A_n}$ converge to $\re{A}$ and $\im{A}$ 
in the strong resolvent topology, respectively.
By the above argument, we see that $\re{A_n}$ and $\im{A_n}$ converge to $\re{A}$ and $\im{A}$ 
in the $\tau$-measure topology, respectively.
Since the addition is continuous with respect to the $\tau$-measure topology, 
$A_n$ converges to $A$ in the $\tau$-measure topology.
This implies $\mathcal{O}_{\tau} \subset \mathcal{O}_{{\rm{SRT}}}$.
Hence the proof of Lemma \ref{3srt=mt} is complete.

\begin{rem}\label{3ref}
We referred to the proof of Theorem 5.5 of the paper \cite{Stinespring} 
to prove the inclusion $\mathcal{O}_{\tau} \subset \mathcal{O}_{{\rm{SRT}}}$.
\end{rem}

\subsection{Direct Sums of Algebras of Unbounded Operators}

To prove Theorem \ref{3cta} in a general case, we show some facts about the direct sums of unbounded operators.
See Appendix A for the definition of the direct sums of unbounded operators.
The next lemma follows immediately from Lemma \ref{3unbdd}.

\begin{lem}\label{3convergence}
Let $\Hs_{\alpha}$ be a Hilbert space, $\Hs$ be the direct sum Hilbert space of $\{\Hs_{\alpha}\}_{{\alpha}}$. 
For each $\alpha$, we consider a net $\{A_{\alpha, \lambda}\}_{\lambda\in\Lambda}$ of self-adjoint operators on $\Hs_{\alpha}$
and self-adjoint operator $A_{\alpha}$ on $\Hs_{\alpha}$.
Set
\begin{equation*}
A_{\lambda} := \oplus_{\alpha}A_{\alpha, \lambda},
\end{equation*}
and
\begin{equation*}
A :=  \oplus_{\alpha}A_{\alpha},
\end{equation*}
on the Hilbert space $\Hs$.
\begin{list}{}{}
\item[(1)] $A_{\lambda}$ converges to $A$ in the strong resolvent topology if and only if
each $\{A_{\alpha, \lambda}\}_{\lambda\in\Lambda}$ converges to $A_{\alpha}$ in the strong resolvent topology.
\item[(2)] $A_{\lambda}$ converges to $A$ in the strong exponential topology if and only if
each $\{A_{\alpha, \lambda}\}_{\lambda\in\Lambda}$ converges to $A_{\alpha}$ in the strong exponential topology.
\end{list}
\end{lem}

\begin{proof}
(1) By Lemma \ref{3unbdd}, we have
\begin{equation*}
(A_{\lambda}-i)^{-1} = \oplus_{\alpha}(A_{\alpha, \lambda}-i)^{-1},\ \ \ \ 
(A-i)^{-1} = \oplus_{\alpha}(A_{\alpha}-i)^{-1}.
\end{equation*}
The necessary condition is trivial.
On the other hand, it is easy to see that $\{(A_{\lambda}-i)^{-1}\}_{\lambda\in\Lambda}$ 
converges to $(A-i)^{-1}$ on $\hat{\oplus}_{\alpha}\Hs_{\alpha}$. 
Since $\widehat{\bigoplus}_{\alpha}\Hs_{\alpha}$ is dense in $\bigoplus_{\alpha}\Hs_{\alpha}$ 
and $\{(A_{\lambda}-i)^{-1}\}_{\lambda\in\Lambda}$ is uniformly bounded, 
the sufficient condition follows. 

(2) Similar to the proof (1).
\end{proof}

The next lemma is the key to prove Theorem \ref{3cta}.

\begin{lem}\label{3key}
Let $\M_{\alpha}$ be a finite von Neumann algebra acting on $\Hs_{\alpha}$, and put 
\begin{equation*}
\M := \bigoplus_{\alpha}^{b}\M_{\alpha}.
\end{equation*}
Then 
\begin{equation*}
\cM = \bigoplus_{\alpha}\cM_{\alpha}
\end{equation*}
holds.
The sum, the scalar multiplication, the multiplication and the involution are given by 
\begin{align*}
\overline{\left(\oplus_{\alpha}A_{\alpha}\right)+\left(\oplus_{\alpha}B_{\alpha}\right)} 
&= \oplus_{\alpha}\left(\overline{A_{\alpha}+B_{\alpha}}\right), \\
\overline{\lambda \left(\oplus_{\alpha}A_{\alpha}\right)} 
&= \oplus_{\alpha}\left(\overline{\lambda A_{\alpha}}\right),\ \ \ \ for\ all\ \lambda \in \mathbb{C}, \\
\overline{\left(\oplus_{\alpha}A_{\alpha}\right)\left(\oplus_{\alpha}B_{\alpha}\right)} 
&= \oplus_{\alpha}\left(\overline{A_{\alpha}B_{\alpha}}\right), \\
\left(\oplus_{\alpha}A_{\alpha}\right)^* &=  \oplus_{\alpha}\left({A_{\alpha}}^*\right).
\end{align*}
In addition, if each $\cM_{\alpha}$ is countably decomposable, 
then $\cM$ is a complete topological *-algebra with respect to the strong resolvent topology,
and the strong resolvent topology coincides with the strong exponential topology on $\cM$.
\end{lem}

\begin{proof}
we shall prove this lemma step by step. \\

{\bf Step 1.} We first show that $\bigoplus_{\alpha}\cM_{\alpha} \subset \cM$.
Indeed let $\oplus_{\alpha}A_{\alpha}\in\bigoplus_{\alpha}\cM_{\alpha}$.
By Lemma \ref{3bdd} and Lemma \ref{3sum_vN}, each unitary operator $u\in U(\M')$ 
can be written as $u=\oplus_{\alpha}u_{\alpha}$, where $u_{\alpha}\in U(\M_{\alpha}')$.
Thus we have
\begin{equation*}
u\left(\oplus_{\alpha}A_{\alpha}\right) = \oplus_{\alpha}\left(u_{\alpha}A_{\alpha}\right) 
\subset \oplus_{\alpha}\left(A_{\alpha}u_{\alpha}\right) 
= \left(\oplus_{\alpha}A_{\alpha}\right)u.
\end{equation*}
This implies $\oplus_{\alpha}A_{\alpha}\in\cM$. \\

{\bf Step 2.} We show that the converse inclusion $\cM \subset \bigoplus_{\alpha}\cM_{\alpha}$.
For each $\beta$, we put
\begin{equation*}
q_{\beta} := \oplus_{\alpha}\left(\delta_{\alpha \beta}1_{\Hs_{\alpha}}\right) \in \M',
\end{equation*}
where $\delta_{\alpha \beta}$ is the Kronecker delta and $1_{\Hs_{\alpha}}$ is the identity operator on $\Hs_{\alpha}$.
From Lemma \ref{3bdd}, $q_{\beta}$ is a projection and 
\begin{equation*}
\ran{q_{\beta}} = \bigoplus_{\alpha}\left(\delta_{\alpha \beta}\Hs_{\alpha}\right) =: \tilde{\Hs}_{\beta}.
\end{equation*}
Let $A\in\cM$ be a self-adjoint operator.
We would like to prove $A\in\bigoplus_{\alpha}\cM_{\alpha}$.
Since $q_{\beta}\in\M'$, we have $q_{\beta}A\subset A q_{\beta}$ for all $\beta$.
This implies that each $\tilde{\Hs}_{\beta}$ reduces $A$.
We denote reduced part of $A$ to $\tilde{\Hs}_{\beta}$ by $A_{\tilde{\Hs}_{\beta}}$.
$A_{\tilde{\Hs}_{\beta}}$ is obviously self-adjoint.
For each $\beta$, we consider natural unitary operator $v_{\beta}:\Hs_{\beta}\longrightarrow\tilde{\Hs}_{\beta}$.
Then the operator $A_{\beta}:=v_{\beta}^*A_{\tilde{\Hs}_{\beta}}v_{\beta}$ is again self-adjoint.
To prove $A=\oplus_{\alpha}A_{\alpha}$, we take an arbitrary $\xi\in\widehat{\bigoplus}_{\alpha}\dom{A_{\alpha}}$.
Since $v_{\alpha}\xi^{(\alpha)}\in \dom{A_{\tilde{\Hs}_{\alpha}}}\subset \dom{A}$, 
we see that 
\begin{equation*}
\xi=\sum_{
\begin{subarray}{c} \alpha \\
{\rm finite\ sum}
\end{subarray}
}v_{\alpha}\xi^{(\alpha)}\in\dom{A}.
\end{equation*}
Therefore we obtain

\begin{align*}
\left(\oplus_{\alpha}A_{\alpha}\right)\xi &= \left\{v_{\alpha}^* Av_{\alpha}\xi^{(\alpha)}\right\}_{\alpha}
= \sum_{
\begin{subarray}{c} \alpha \\
{\rm finite\ sum}
\end{subarray}
}Av_{\alpha}\xi^{(\alpha)} \\
&= A \left(\sum_{
\begin{subarray}{c} \alpha \\
{\rm finite\ sum}
\end{subarray}
}v_{\alpha}\xi^{(\alpha)}\right) 
= A\xi.
\end{align*}
Hence $\left(\oplus_{\alpha}A_{\alpha}\right)|_{\widehat{\bigoplus}_{\alpha}\dom{A_{\alpha}}}\subset A$.
By Lemma \ref{3sum_core}, we have $\oplus_{\alpha}A_{\alpha}\subset A$.
On the other hand, both of $\oplus_{\alpha}A_{\alpha}$ and $A$ are self-adjoint 
and self-adjoint operators have no non-trivial self-adjoint extension.
These facts implies $\oplus_{\alpha}A_{\alpha} = A$.
Next we show that each $A_{\alpha}$ is in $\cM_{\alpha}$. 
Taking arbitrary unitary operators $u_{\alpha}\in U(\M_{\alpha}')$ and putting $u:=\oplus_{\alpha}u_{\alpha}$, 
then by Lemma \ref{3sum_vN}, $u$ is a unitary operator in $\M'$.
Since $A\in\cM$, we see that
\begin{equation*}
\left(\oplus_{\alpha}u_{\alpha}\right)\left(\oplus_{\alpha}A_{\alpha}\right) =uA 
\subset Au = \left(\oplus_{\alpha}A_{\alpha}\right)\left(\oplus_{\alpha}u_{\alpha}\right).
\end{equation*} 
Thus for all $\alpha$, $u_{\alpha}A_{\alpha} \subset A_{\alpha}u_{\alpha}$ holds.
This implies $A_{\alpha}\in\cM_{\alpha}$ for all $\alpha$.
Hence $A\in\bigoplus_{\alpha}\cM_{\alpha}$.

Next we consider an arbitrary element $A\in\cM$.
Putting $B:=\re{A}$, $C:=\im{A}$, then $A=\overline{B+iC}$.
Since $B$ and $C$ are self-adjoint, by the above argument, 
there exist operators $B_{\alpha}\in\cM_{\alpha}$ and $C_{\alpha}\in\cM_{\alpha}$
such that  $B = \oplus_{\alpha}B_{\alpha}$ and $C = \oplus_{\alpha}C_{\alpha}$ holds.
Set $\mathcal{D}:=\widehat{\bigoplus}_{\alpha}\left(\dom{B_{\alpha}}\cap\dom{C_{\alpha}}\right)$.
Since $\dom{B_{\alpha}}\cap\dom{C_{\alpha}}$ is a core of $\overline{B_{\alpha}+iC_{\alpha}}$, 
$\mathcal{D}$ is a core of $\oplus_{\alpha}\left(\overline{B_{\alpha}+iC_{\alpha}}\right)$ by Lemma \ref{3sum_core}.
We observe that 
\begin{equation*}
A = \overline{B+iC} \supset \left(\overline{B+iC}\right)|_{\mathcal{D}} = 
\left\{\oplus_{\alpha}\left(\overline{B_{\alpha}+iC_{\alpha}}\right)\right\}|_{\mathcal{D}},
\end{equation*}
so that $A \supset \oplus_{\alpha}\left(\overline{B_{\alpha}+iC_{\alpha}}\right)$ follows.
Now we use Step 1., then  we see that $\oplus_{\alpha}\left(\overline{B_{\alpha}+iC_{\alpha}}\right)\in\cM$ 
because $\overline{B_{\alpha}+iC_{\alpha}}\in\cM_{\alpha}$ for all $\alpha$.
Since $\M$ is a finite von Neumann algebra, Lemma \ref{2No-closed-extension} means that 
$A = \oplus_{\alpha}\left(\overline{B_{\alpha}+iC_{\alpha}}\right)\in\bigoplus_{\alpha}\cM_{\alpha}$.
Hence $\cM = \bigoplus_{\alpha}\cM_{\alpha}$ follows. \\

{\bf Step 3.} We shall show the formulae with respect to the sum, the scalar product, the product and the involution.
The formulae with respect to the scalar product and the involution are trivial.
We first prove the formula of the sum.
Let $\oplus_{\alpha}A_{\alpha}$, $\oplus_{\alpha}B_{\alpha}\in\bigoplus_{\alpha}\cM_{\alpha}=\cM$.
Put $\mathcal{D}_{+}:=\widehat{\bigoplus}_{\alpha}\left(\dom{A_{\alpha}}\cap\dom{B_{\alpha}}\right)$.
Since $\dom{A_{\alpha}}\cap\dom{B_{\alpha}}$ is a core of $\overline{A_{\alpha}+B_{\alpha}}$, 
$\mathcal{D}_{+}$ is a core of $\oplus_{\alpha}\left(\overline{A_{\alpha}+B_{\alpha}}\right)$.
We observe that 
\begin{equation*}
\overline{\left(\oplus_{\alpha}A_{\alpha}\right)+\left(\oplus_{\alpha}B_{\alpha}\right)} \supset
\left\{\oplus_{\alpha}\left(\overline{A_{\alpha}+B_{\alpha}}\right)\right\}|_{\mathcal{D}_{+}},
\end{equation*}
so that
\begin{equation*}
\overline{\left(\oplus_{\alpha}A_{\alpha}\right)+\left(\oplus_{\alpha}B_{\alpha}\right)} \supset
\oplus_{\alpha}\left(\overline{A_{\alpha}+B_{\alpha}}\right)
\end{equation*}
follows.
Since both sides are elements in $\cM$, we have
\begin{equation*}
\overline{\left(\oplus_{\alpha}A_{\alpha}\right)+\left(\oplus_{\alpha}B_{\alpha}\right)} =
\oplus_{\alpha}\left(\overline{A_{\alpha}+B_{\alpha}}\right).
\end{equation*}

Next, to show the formula of the product, 
we put $\mathcal{D}_{\times}:=\widehat{\bigoplus}_{\alpha}\dom{A_{\alpha}B_{\alpha}}$.
Since $\dom{A_{\alpha}B_{\alpha}}$ is a core of $\overline{A_{\alpha}B_{\alpha}}$, 
$\mathcal{D}_{\times}$ is a core of $\oplus_{\alpha}\left(\overline{A_{\alpha}B_{\alpha}}\right)$.
We observe that 
\begin{equation*}
\overline{\left(\oplus_{\alpha}A_{\alpha}\right)\left(\oplus_{\alpha}B_{\alpha}\right)} \supset
\left\{\oplus_{\alpha}\left(\overline{A_{\alpha}B_{\alpha}}\right)\right\}|_{\mathcal{D}_{+}},
\end{equation*}
so that
\begin{equation*}
\overline{\left(\oplus_{\alpha}A_{\alpha}\right)\left(\oplus_{\alpha}B_{\alpha}\right)} \supset
\oplus_{\alpha}\left(\overline{A_{\alpha}B_{\alpha}}\right)
\end{equation*}
follows.
Since both sides are elements in $\cM$, we have
\begin{equation*}
\overline{\left(\oplus_{\alpha}A_{\alpha}\right)\left(\oplus_{\alpha}B_{\alpha}\right)} =
\oplus_{\alpha}\left(\overline{A_{\alpha}B_{\alpha}}\right).
\end{equation*}
Hence the proof of Step 3. is complete.\\

In the sequel, we assume that each $\cM_{\alpha}$ is countably decomposable. 
Note that, by Lemma \ref{3srt=mt}, each $\cM_{\alpha}$ is a complete topological *-algebra 
with respect to the strong resolvent topology.\\

{\bf Step 4.} Let $\{A_{\lambda}\}_{\lambda\in\Lambda}$ be a net in $\cM$, $A$ be an element of $\cM$.
Corresponding to $\cM = \oplus_{\alpha}\cM_{\alpha}$, we can write them as follows:
\begin{alignat*}{2}
A_{\lambda} &= \oplus_{\alpha}A_{\alpha, \lambda}, & \qquad A_{\alpha, \lambda}&\in\cM, \\
A           &= \oplus_{\alpha}A_{\alpha},          &                 A_{\alpha}&\in\cM. 
\end{alignat*}
We shall show that $A_{\lambda}$ converges to $A$ with respect to the strong resolvent topology if and only if
each $\{A_{\alpha, \lambda}\}_{\lambda\in\Lambda}$ converges to $A_{\alpha}$ with respect to the strong resolvent topology.
From Step 3., we obtain
\begin{alignat*}{2}
\re{A_{\lambda}} &= \oplus_{\alpha}\re{A_{\alpha, \lambda}}, & \qquad \im{A_{\lambda}} &= \oplus_{\alpha}\im{A_{\alpha, \lambda}}, \\
\re{A}           &= \oplus_{\alpha}\re{A_{\alpha}},           &        \im{A}           &= \oplus_{\alpha}\im{A_{\alpha}},
\end{alignat*}
so that, by Lemma \ref{3convergence}, the above equivalence of convergence follows. \\
 
By Step 3. and Step 4., we see that $\cM$ forms a topological *-algebra with respect to the strong resolvent topology.
Next, to prove the completeness, we prepare some facts. \\

{\bf Step 5.} Fix an arbitrary $\alpha_0$ and let $V^{(\alpha_0)}$ be an arbitrary SOT-open set in $\M_{\alpha_0}$.
Set $V^{(\alpha)}:=\M_{\alpha}$ ($\alpha\not= \alpha_0$) and $V:=\bigoplus_{\alpha}^{b}V^{(\alpha)}$.
Then $V$ is a SOT-open set in $\M$.
Indeed, since for any $x=\oplus_{\alpha}x_{\alpha}\in V$, we have $x_{\alpha_0}\in V^{(\alpha_0)}$,
there exists a positive number $\varepsilon>0$ and finitely many vectors 
$\xi_1^{(\alpha_0)}\in\Hs_{\alpha_0}$,$\cdot\cdot\cdot$, $\xi_n^{(\alpha_0)}\in\Hs_{\alpha_0}$
such that
\begin{equation*}
\bigcap_{k=1}^{n}\left\{y\in\M_{\alpha_0}\ ;\ \|(y-x_{\alpha_0})\xi_k^{(\alpha_0)}\|<\eps\right\} \subset V^{(\alpha_0)}.
\end{equation*}
Set $\xi_k^{(\alpha)}:=0$ ($\alpha\not= \alpha_0$), 
then we get $\xi_k\in\bigoplus_{\alpha}\Hs_{\alpha}$ and
\begin{equation*}
x\in \bigcap_{k=1}^{n}\left\{y\in\M\ ;\ \|(y-x)\xi_k\|<\eps\right\} \subset V.
\end{equation*}
Since $\bigcap_{k=1}^{n}\left\{y\in\M\ ;\ \|(y-x)\xi_k\|<\eps\right\}$ is a SOT-open set in $\M$, V is a SOT-open set.\\

{\bf Step 6.} Fix an arbitrary $\alpha_0$. 
Let $W^{(\alpha_0)}$ be an arbitrary SRT-neighborhood at $0\in\cM_{\alpha_0}$.
Set $W^{(\alpha)}:=\cM_{\alpha}$ ($\alpha\not= \alpha_0$) and $W:=\bigoplus_{\alpha}W^{(\alpha)}$.
Then $W$ is a neighborhood at $0\in\cM$.
Indeed, $0\in W$ is trivial.
On the other hand, since $0\in\cM_{\alpha_0}$, there exists finitely many SOT-open sets 
$V_1^{(\alpha_0)}$, $\cdot\cdot\cdot$, $V_n^{(\alpha_0)}$in $\M_{\alpha_0}$
such that
\begin{equation*}
0\in\bigcap_{k=1}^{n}\left\{A\in \cM_{\alpha_0}\ ;\ 
\left\{\re{A}-i\right\}^{-1},\ \left\{\im{A}-i\right\}^{-1} \in V_k^{(\alpha_0)}\right\}
\subset W^{(\alpha_0)}.
\end{equation*}
Put $V_k^{(\alpha)}:=\M_{\alpha}$ ($\alpha\not= \alpha_0$) and $V_k:=\bigoplus_{\alpha}^{b}V_k^{(\alpha)}$,
then, by Step 5., each $V_k$ is a SOT-open set in $\M$ and 
\begin{equation*}
0\in\bigcap_{k=1}^{n}\left\{A\in \cM\ ;\ 
\left\{\re{A}-i\right\}^{-1},\ \left\{\im{A}-i\right\}^{-1} \in V_k\right\} 
\subset W.
\end{equation*}
By the definition of the strong resolvent topology, 
\begin{equation*}
\bigcap_{k=1}^{n}\left\{A\in \cM\ ;\ 
\left\{\re{A}-i\right\}^{-1},\ \left\{\im{A}-i\right\}^{-1} \in V_k\right\}
\end{equation*}
is a open set, so that $W$ is a SRT-neighborhood at $0\in\cM$. \\

{\bf Step 7.} We shall give a proof of the completeness of $\cM$.
Let $\{A_{\lambda}\}_{\lambda\in\Lambda}$ be a Cauchy net in $\cM$.
For each $\lambda\in\Lambda$ we can write as $A_{\lambda}=\oplus_{\alpha}A_{\alpha, \lambda}\in\bigoplus_{\alpha}\cM_{\alpha}$.
Fix an arbitrary $\alpha_0$ and let $W^{(\alpha_0)}$ be an arbitrary SRT-neighborhood at $0\in\cM_{\alpha_0}$.
Set $W^{(\alpha)}:=\cM_{\alpha}$ ($\alpha\not= \alpha_0$) and $W:=\bigoplus_{\alpha}W^{(\alpha)}$, 
then, by Step 6., $W$ is a SRT-neighborhood at $0\in\cM$.
Therefore there exists $\lambda_0 \in\Lambda$ such that $\overline{A_{\lambda}-A_{\mu}}\in W$ for all $\lambda$, $\mu\geq \lambda_0$.
Since $\overline{A_{\lambda}-A_{\mu}}=\oplus_{\alpha}\left(\overline{A_{\alpha,\lambda}-A_{\alpha,\mu}}\right)$,
this implies that $\overline{A_{\alpha_0, \lambda}-A_{\alpha_0, \mu}}\in W^{(\alpha_0)}$ for all $\lambda$, $\mu\geq \lambda_0$.
Hence $\{A_{\alpha_0, \lambda}\}_{\lambda\in\Lambda}$ is a Cauchy net in $\cM_{\alpha_0}$.
We now use the completeness of $\cM_{\alpha_0}$, then there exists an element 
$A_{\alpha_0}\in\cM_{\alpha_0}$ such that $A_{\alpha_0, \lambda}\rightarrow A_{\alpha_0}$.
Since $\alpha_0$ is arbitrary, so that this means that 
for each $\alpha$, there exists an element $A_{\alpha}\in\cM_{\alpha}$ such that $A_{\alpha, \lambda}\rightarrow A_{\alpha}$.
Put $A:=\oplus_{\alpha}A_{\alpha}$, then, by Step 4., we conclude that $A_{\lambda}\rightarrow A$.
Thus $\cM$ is complete.\\

{\bf Step 8.} The strong resolvent topology coincides with the strong exponential topology on $\cM$.
This fact follows from Lemma \ref{3SRT=SET} and Lemma \ref{3convergence}.
\end{proof}

\begin{lem}\label{6UXU^*}
Let $(\M, \Hs)$ and $(\mathfrak{N}, \mathcal{K})$ be spatially isomorphic finite von Neumann algebras.
If a unitary operator $U$ of $\Hs$ onto $\mathcal{K}$ induces the spatial isomorphism, 
then the map
\begin{equation*}
\Phi:\ \cM \to \overline{\mathfrak{N}},\ X\mapsto UXU^*
\end{equation*}
is a *-isomorphism. 
Moreover $\Phi$ is a homeomorphism with respect to the strong resolvent topology and the strong exponential topology.
\end{lem}

\begin{proof}
It is easy to see that $\Phi(X)\in\overline{\mathfrak{N}}$ for all $X\in\cM$. 
Thanks to Proposition \ref{2No-closed-extension} (2), 
it is not difficult to show that $\Phi$ is a unital *-homomorphism:
\eqa{
U\left(\ol{X+Y}\right)U^*&=\ol{UXU^*+UYU^*}\\
U\left(\ol{XY}\right)U^*&=\ol{UXU^*UYU^*}\\
UX^*U^*&=(UXU^*)^*.
}
Furthermore, it is straightforward to verify that $\Phi$ is invertible, 
the inverse of which is given by $\overline{\mathfrak{N}}\ni Y\mapsto U^*YU\in \cM$.
Topological property is trivial.
\end{proof}

\subsection{Proof of Theorem \ref{3cta}}

We shall give a proof of Theorem \ref{3cta}.
By Lemma \ref{2decompose}, there exists a family of countably decomposable finite von Neumann algebras 
$\{\M_{\alpha}\}_{\alpha}$ such that $\M$ is spatially isomorphic onto $\bigoplus_{\alpha}^{b}\M_{\alpha}$.
From Lemma \ref{6UXU^*}, there exists a *-isomorphism of $\cM$ onto $\bigoplus_{\alpha}\cM_{\alpha}$ 
which is homeomorphic with respect to the strong resolvent topology and the strong exponential topology.
By Lemma \ref{3key}, $\bigoplus_{\alpha}\cM_{\alpha}$ is complete topological *-algebra,
so that so is $\cM$.
Hence the proof of Theorem \ref{3cta} is complete.

\subsection{Local Convexity}

We study the local convexity of $\left(\cM,SRT\right)$ here.

\begin{prop}\label{4local convex iff atomic} 
Let $\M$ be a finite von Neumann algebra. 
Then the following are equivalent.
\begin{list}{}{}
\item[(1)] $(\cM,SRT)$ is locally convex.
\item[(2)] $\M$ is atomic.
\end{list}
\end{prop}

We need some lemmata to prove the above proposition.

\begin{lem}\label{3atomic is locally convex}
Let $\M$ be an atomic finite von Neumann algebra, then $\left(\cM,SRT\right)$ is locally convex. 
\end{lem}

\begin{proof}
Every atomic finite von Neumann algebra is spatially isomorphic to the direct sum of matrix algebras 
$\left\{M_{n_{\lambda}}(\mathbb{C})\right\}_{\lambda\in\Lambda}$,
where each $M_{n_{\lambda}}(\mathbb{C})$ is the algebra of all $n_{\lambda}\times n_{\lambda}$ complex matrices.
Thus we should only prove this lemma in the case that $\M$ is equal to $\bigoplus_{\lambda\in\Lambda}^{b}M_{n_{\lambda}}(\mathbb{C})$.
Note that 
\begin{equation*}
\cM=\overline{\bigoplus_{\lambda\in\Lambda}^{b}M_{n_{\lambda}}(\mathbb{C})}
=\bigoplus_{\lambda\in\Lambda}\overline{M_{n_{\lambda}}(\mathbb{C})}
=\bigoplus_{\lambda\in\Lambda}M_{n_{\lambda}}(\mathbb{C}).
\end{equation*}
Let $p_{\lambda}$ be a semi-norm on $\cM$ defined by
\begin{equation*}
p_{\lambda}(x):=\|x_{\lambda}\|, \ \ \ \ 
x=\oplus_{\lambda\in\Lambda}x_{\lambda}\in\cM=\bigoplus_{\lambda\in\Lambda}M_{n_{\lambda}}(\mathbb{C}).
\end{equation*}
Then the strong resolvent topology on $\cM$ coincides with 
the locally convex topology induced by the semi-norms $\{p_{\lambda}\}_{\lambda\in\Lambda}$
because there is only one Hausdorff linear topology on a finite dimensional linear space.
Hence the proof is complete.
\end{proof}

\begin{lem}\label{3diffuse}
Let $\M$ be a diffuse finite von Neumann algebra, 
then there exists no non-zero SRT-continuous linear functional on $\cM$.
\end{lem}

\begin{proof}
Suppose there exists a non-zero SRT-continuous linear functional $f$ on $\cM$ and we shall show a contradiction.
Since, by Lemma \ref{3mt=sot}, the restriction of $f$ onto $\M$ is a $\sigma$-strongly continuous linear functional on $\M$ 
and $\M$ is SRT-dense in $\cM$,
there exists a projection $e_0$ in $\M$ such that $f(e_0)\not=0$.\\

{\bf Step 1.} For any orthogonal family of non-zero projections $\{e_n\}_{n=1}^{\infty}$ of $\M$, 
$f(e_n)=0$ except at most finitely many $n\in\mathbb{N}$.
Indeed, put 
\begin{equation*}
A:= \sum_{n=1}^{\infty}a_ne_n\in\cM, \ \ \ \ 
a_n:= \left\{
\begin{array}{ll}
\frac{1}{f(e_n)} & \mbox{ if \ \ $f(e_n)\not=0$, } \\
0 & \mbox{ if\ \  $f(e_n)=0$, } \\
\end{array}
\right.
\end{equation*}
where convergence of $A$ is in the strong resolvent topology.
Then we have
\begin{equation*}
f(A)=\sum_{n=1}^{\infty}a_nf(e_n) =\sum_{f(e_n)\not=0}1<\infty,
\end{equation*}
so that $f(e_n)=0$ except at most finitely many $n\in\mathbb{N}$.\\

{\bf Step 2.} For any $e\in P(\M)$ with $f(e)\neq 0$, there exists $e'\in P(\M)$ such that $0\neq e'\le e$ and $f(e')=0$.
Indeed, since $\M$ is diffuse, there exists an orthogonal family of non-zero projections 
$\{e_n\}_{n=1}^{\infty}$ in $\M$ such that $e=\sum_{n\ge 1}e_n$. 
By Step 2., $J:=\left\{n\in \mbb{N}\ ;\ f(e_n)\neq 0\right\}$ is a finite set. 
In particular, 
\[e':=e-\sum_{n\in J}e_n\neq 0\]
satisfies $f(e')=0$.\\

{\bf Step 3.} We shall get a contradiction.
By Step 2., we can take a maximal orthogonal family of non-zero projections $\{e_{\alpha}\}_{\alpha\in A}$ in $\M$ 
such that $e_{\alpha}\le e_0$ and $f(e_{\alpha})=0$.
Let $e:=\sum_{\alpha \in A}e_{\alpha}$. 
The maximality of $\{e_{\alpha}\}_{\alpha\in A}$ and Step 2. implies $e=e_0$. 
Thus we have 
\[0\neq f(e_0)=\sum_{\alpha \in A}f(e_{\alpha})=0,\]
which is a contradiction. 
Hence there exists no non-zero SRT-continuous linear functional on $\cM$. 
\end{proof}

\begin{lem}\label{3conjugate}
Let $\M_{\rm a}$ be an atomic finite von Neumann algebra, $\M_{\rm d}$ be a diffuse finite von Neumann algebra
and $\M:=\M_{\rm a}\bigoplus^{b}\M_{\rm d}$ be the direct sum von Neumann algebra.
Denote the conjugate spaces of $\left(\overline{\M_{\rm a}},SRT\right)$ and $\left(\cM,SRT\right)$ by 
$\left(\overline{\M_{\rm a}}\right)^*$ and $\left(\cM\right)^*$ respectively.
For each $f\in\left(\overline{\M_{\rm a}}\right)^*$, we define $I(f)\in\left(\cM\right)^*$ as 
\begin{equation*}
I(f)(A\oplus B):= f(A), \ \ \ \ A\oplus B\in \cM=\overline{\M_{\rm a}}\bigoplus\overline{\M_{\rm d}},
\end{equation*} 
then $I$ is a bijection between $\left(\overline{\M_{\rm a}}\right)^*$ onto $\left(\cM\right)^*$.
\end{lem}

\begin{proof}
This follows immediately from Lemma \ref{3diffuse}.
\end{proof}

\begin{proof}[\bf Proof of Proposition \ref{4local convex iff atomic}]
We have only to prove (1)$\Rightarrow$(2).
Since $\M$ is spatially isomorphic to the direct sum of an atomic von Neumann algebra $\M_{\rm atomic}$
and a diffuse von Neumann algebra $\M_{\rm diffuse}$, it is enough to show that $\M_{\rm diffuse}=\{0\}$.
Suppose $\M_{\rm diffuse}\not=\{0\}$ and take $y\in\M_{\rm diffuse}\backslash\{0\}$.
Then, by local convexity of $\cM$, 
there exists a SRT-continuous linear functional $f$ on $\cM$ 
such that $f(0\oplus y)\not=0$.
However this is a contradiction by Lemma \ref{3conjugate}.
\end{proof}

Similarly one can prove the following proposition.

\begin{prop}\label{3diffuse2}
Let $\M$ be a finite von Neumann algebra. 
Then the following are equivalent.
\begin{list}{}{}
\item[(1)] There exists no non-zero SRT-continuous linear functional on $\cM$.
\item[(2)] $\M$ is diffuse.
\end{list}
\end{prop}

\section{Lie Group-Lie Algebra Correspondences}
\label{sec4}
In this section we state and prove the main result of this paper. 
As explained in the introduction, Lie theory for $U(\Hs)$ is a difficult issue. 
What one has to resolve for discussing the Lie group-Lie algebra correspondence is 
a domain problem of the generators of one parameter subgroups of $G\subset U(\Hs)$. 
The second to be discussed is a continuity of the Lie algebraic operations. 
However we can show that, for any strongly closed subgroup $G$ of unitary group $U(\M)$ of some finite von Neumann algebra $\M$, 
there exists canonically a  complete topological Lie algebra. 
Since there are continuously many non-isomorphic finite von Neumann algebras on $\Hs$, there are also varieties of such groups. 
We hope that the ``Lie Groups-Lie Algebras Correspondences" will play some important roles in the infinite dimensional Lie theory. 
We study the SRT-closed subalgebra of $\cM$, too.

\subsection{Existence of Lie Algebra}

Let $\M$ be a finite von Neumann algebra acting on a Hilbert space $\Hs$. 
Recall that a densely defined closable operator $A$ is called a {\it skew-adjoint operator} if $A^*=-A$, 
and $A$ is called {\it essentially skew-adjoint} if $\ol{A}$ is skew-adjoint.

\begin{rem}
In general, the strong limit of unitary operators is not necessarily unitary. 
It is known that $U(\M)$ is strongly closed in $\mathfrak{B}(\Hs)$ if and only if $\M$ is a finite von Neumann algebra. 
\end{rem}

\begin{dfn}For a strongly closed subgroup $G$ of $U(\M)$, the set 
\[\mf{g}=\text{Lie}(G):=\{A\ ;\ A^*=-A\ \text{on}\ \mathcal{H},\ e^{tA}\in G,\ \text{for all }t\in \mathbb{R}\}\]
is called the Lie algebra of $G$.\ The complexification $\mf{g}_{\mbb{C}}$ of $\mf{g}$ is defined by
\[\mf{g}_{\mbb{C}}:=\left\{\overline{A+iB}\ ;\ A,B\in \mf{g}\right\}.\]
If $G=U(\M)$, we sometimes write $\mf{g}$ as $\mf{u}(\M)$.
\end{dfn}

At first sight, it is not clear whether we can define algebraic operations on $\mf{g}$. However,
\begin{lem} 
Under the above notations, $\mf{g}\subset \cM$ holds.
\end{lem}

\begin{proof} Let $u\in U(\M')$ and $A\in \mf{g}$. 
By definition, we have $e^{tA}u=ue^{tA}$. 
Taking the strong derivative on each side, we have $uA\subset Au$. 
Since $u$ is arbitrary we obtain $uA=Au$, which implies $A\in \cM$.
\end{proof}

Therefore the sum $\overline{A+B}$ and the Lie bracket $\overline{AB-BA}$ are well-defined operations in $\cM$, 
but it is not clear whether they belong to $\mf{g}$ again.
The following Lemma \ref{4Lie alg} guarantees the validity of the name ``Lie algebra". 
The former part of the proof is based on the two lemmata established by Trotter-Kato and E. Nelson, 
which are of importance in their own.

\begin{lem}[Trotter-Kato, Nelson \cite{Nelson2}]\label{4TKN}
Let $A,B$ be skew-adjoint operators on a Hilbert space $\mathcal{H}$.
\begin{list}{}{}
\item[(1)] If $A+B$ is essentially skew-adjoint on $\dom{A}\cap \dom{B}$, 
then it holds that
\[e^{t(\overline{A+B})}=\slimn \brac{e^{tA/n}e^{tB/n}}^n,\]
for all $t\in \mbb{R}$.
\item[(2)] If $(AB-BA)$ is essentially skew-adjoint on
\[\dom{A^2}\cap \dom{AB}\cap \dom{BA}\cap \dom{B^2},\] 
then it holds that
\[e^{t[A,B]}=\slimn \brac{e^{-\sqrt{\frac{t}{n}}A}e^{-\sqrt{\frac{t}{n}}B}e^{\sqrt{\frac{t}{n}}A}e^{\sqrt{\frac{t}{n}}B}}^{n^2},\]
for all $t>0$, where $[A,B]:=\overline{AB-BA}$.
\end{list}
\end{lem}

\begin{lem}\label{4Lie alg}
Let $G$ be a strongly closed subgroup of $U(\M)$. 
Then $\mf{g}$ is a real Lie algebra with the Lie bracket $[X,Y]:=\overline{XY-YX}.$
\end{lem}

\begin{proof}
Let $A,B\in \mf{g}.$ 
It suffices to prove that $\ol{A+B}$ and $\ol{AB-BA}$ belong to $\mf{g}$. 
Since $\dom{A}\cap \dom{B}$ is completely dense, $A+B$ is essentially skew-adjoint. 
Therefore by Lemma \ref{4TKN} (1), we have $e^{t(\ol{A+B})}\in \ol{G}^{s}=G$ for all $t\in \mbb{R}.$ 
This implies $\ol{A+B}\in \mf{g}.$ 
It is clear that $\lambda A\in \mf{g}$ for all $\lambda \in \mbb{R}.$ 
On the other hand, as $\ol{AB-BA}$ is essentially skew-adjoint on 
\[\mathcal{D}:=\dom{AB}\cap \dom{BA}\cap \dom{A^2}\cap \dom{B^2},\]
since $\mathcal{D}$ is completely dense by Proposition \ref{2countable intersection of c-dense} and $\ol{AB-BA}\in \cM.$ 
Therefore by Proposition \ref{4TKN} (2), we have $e^{t(\ol{AB-BA})}\in G$ for all $t>0.$ 
Thanks to the unitarity, this equality is also valid for $t<0.$ 
Thus we obtain $[A,B]\in \mf{g}.$ 
The associativity of the algebraic operations follows from the Murray-von Neumann's Theorem \ref{2eta_*alg}.
\end{proof}

Now we state the main result of this paper, whose proof is almost completed in the previous arguments.

\begin{thm}\label{4main} 
Let $G$ be a strongly closed subgroup of the unitary group $U(\M)$ of a finite von Neumann algebra $\M$.
Then $\mf{g}$ is a complete topological real Lie algebra with respect to the strong resolvent topology.
Moreover, $\mf{g}_{\mbb{C}}$ is a complete topological Lie $^*$-algebra. 
\end{thm}

\begin{proof}
The Lie algebraic properties are proved in Lemma \ref{4Lie alg}. 
By Lemma \ref{3cta}, we see that $\mf{g}$ and $\mf{g}_{\mbb{C}}$ are SRT-closed Lie subalgebras of $\cM$. 
As the algebraic operations $(X,Y)\mapsto \ol{X+Y},\ [X,Y]$ are continuous with respect to the strong resolvent topology 
, so that the topological properties follow. 
\end{proof}

\begin{rem}
It is easy to see that for $G=U(\M)$, its Lie algebra $\mf{u}(\M)$ is equal to $\{A\in \cM; A^*=-A\}$ and the exponential map
\[\exp : \mf{u}(\M)\to U(\M)\]
is continuous by Lemma \ref{ASRT=exp} and surjective by the spectral theorem.
\end{rem}

\begin{prop} 
Let $\M_1,\ \M_2$ be finite von Neumann algebras on Hilbert spaces $\mathcal{H}_1$, $\mathcal{H}_2$ respectively. 
Let $G_i$ be a strongly closed subgroup of $U(\M_i)$ ($i=1,2$). 
For any strongly continuous group homomorphism $\varphi:G_1\to G_2$, 
there exists a unique SRT-continuous Lie algebra homomorphism $\Phi:\ \text{Lie}(G_1)\to \text{Lie}(G_2)$ 
such that $\varphi(e^{A})=e^{\Phi (A)}$ for all $A\in \text{Lie}(G_1).$ 
In particular, if $G_1$ is isomorphic to $G_2$ as a topological group, 
then $\text{Lie}(G_1)$ and $\text{Lie}(G_2)$ are isomorphic as a topological Lie algebra.
\end{prop}

\begin{proof}
Let $X$ be an element in $\text{Lie}(G_1)$. 
From the strong continuity of $\varphi$, $t\mapsto \varphi(e^{tX})$ is a strongly continuous one-parameter unitary group. 
Therefore by Stone theorem, there exists uniquely a skew-adjoint operator $\Phi(X)$ on $\mathcal{H}_2$ 
such that $\varphi(e^{tX})=e^{t\Phi(X)}.$ 
This equality implies $\Phi(X)\in \text{Lie}(G_2).$ 
Since $\varphi$ is strongly continuous, thanks to Lemma \ref{4TKN}, we see that
\eqa{
e^{t\Phi([X,Y])}&=\varphi (e^{t\cdot [X,Y]})\\
&=\varphi \brac{\slimn \left [e^{-\sqrt{\frac{t}{n}}X}e^{-\sqrt{\frac{t}{n}}Y}e^{\sqrt{\frac{t}{n}}X}e^{\sqrt{\frac{t}{n}}Y}\right ]^n}\\
&=\slimn \left [\varphi\brac{e^{-\sqrt{\frac{t}{n}}X}}\varphi\brac{e^{-\sqrt{\frac{t}{n}}Y}}
\varphi\brac{e^{\sqrt{\frac{t}{n}}X}}\varphi\brac{e^{\sqrt{\frac{t}{n}}Y}}\right ]^n\\
&=\slimn \left [e^{-\sqrt{\frac{t}{n}}\Phi(X)}e^{-\sqrt{\frac{t}{n}}\Phi(Y)}e^{\sqrt{\frac{t}{n}}\Phi(X)}e^{\sqrt{\frac{t}{n}}\Phi(Y)}\right ]^n\\
&=e^{t\cdot [\Phi(X),\Phi(Y)]},
}
for all $t>0$.
Taking the inverse of unitary operators, the equality $e^{t\Phi([X,Y])}=e^{t[\Phi(X),\Phi(Y)]}$ is also valid for all $t<0$. 
Therefore from the uniqueness of a generator of one-parameter group, we have $\Phi([X,Y])=[\Phi(X),\Phi(Y)].$ 
Similarly, we can prove that $\Phi$ is linear. 
Thus, $\Phi$ is a Lie algebra homomorphism. 
The SRT-continuity of $\Phi$ follows from the continuity of $\varphi$ and the definition of the strong exponential topology. 
\end{proof}

As above, $G$ has finite dimensional characters. 
On the other hand, it also has an infinite dimensional character.

\begin{prop}
Let $\M$ be a finite von Neumann algebra, then the following are equivalent.
\begin{list}{}{}
\item[(1)] The exponential map $\exp:\ \mf{u}(\M)\ni X\mapsto \exp(X)\in U(\M)$ is locally injective. 
Namely, the restriction of the map onto some SRT-neighborhood of $0\in \cM$ is injective. 
\item[(2)] $\M$ is finite dimensional.
\end{list}
\end{prop}

\begin{proof}
(2)$\Rightarrow$(1) is trivial.
We should only prove that (1)$\Rightarrow$(2).\\

{\bf Step 1.} For each orthogonal family of non-zero projections in $\M$, its cardinal number is finite. 
Indeed if there exists a orthogonal family of non-zero projections in $\M$ whose cardinal number is infinite,
we can take a countably infinite subset of it and write it as $\{p_n\}_{n=1}^{\infty}$.
Since $p_n$ converges strongly to $0$, it also converges to $0$ in the strong resolvent topology.   
Define $x_n:=2\pi ip_n\neq 0$. 
Since the spectral set of $p_n$ is $\{0,1\}$, 
we have $e^{x_n}=1$ for all $n\in\mathbb{N}$, while $x_n$ converges to $0$ in the strong resolvent topology. 
This implies that the map $\exp(\cdot)$ is not locally injective, which is a contradiction.\\

{\bf Step 2.} $\M$ is atomic.
Indeed if $\M$ is not atomic, the diffuse part of it is not $\{0\}$.
Thus we can take an infinite sequence of non-zero mutually orthogonal projections in $\M$.
But this is a contradiction to Step 1..\\

{\bf Step 3.} We shall show that $\M$ is finite dimensional.
By Step 2., $\M$ is spatially isomorphic 
to the direct sum of a family $\left\{M_{n_{\lambda}}(\mathbb{C})\right\}_{\lambda\in\Lambda}$ ($n_{\lambda}\in\mathbb{N}$), 
where $M_{n_{\lambda}}(\mathbb{C})$ is the algebra of all $n_{\lambda}\times n_{\lambda}$ complex matrices.
By Step 1., the cardinal number of $\Lambda$ is finite.
Hence $\M$ is finite dimensional.
\end{proof}

\begin{rem}
$\text{Lie}(G)$ is not always locally convex, whereas most of infinite dimensional Lie theories, by contrast, assume local convexity.
Indeed, by Proposition \ref{4local convex iff atomic}, $\mf{u}(\M)$ is locally convex if and only if $\M$ is atomic.
\end{rem}

\subsection{Closed Subalgebras of $\cM$}

Next, we characterize closed *-subalgebras of $\cM$.

\begin{prop}\label{5main}
Let $\M$ be a finite non Neumann algebra on a Hilbert space $\Hs$, 
$\mathscr{R}$ be a SRT-closed *-subalgebra of $\cM$ with $1_{\Hs}$.
Then there exists a unique von Neumann subalgebra $\mathfrak{N}$ of $\M$ 
such that $\mathscr{R}=\overline{\mathfrak{N}}$.
\end{prop}

\begin{rem}\label{5rem}
A von Neumann subalgebra of a finite von Neumann algebra is also finite.
\end{rem}

\begin{proof}[{\rm{\textbf{Proof of Proposition \ref{5main}}}}]
Put
\begin{equation*}
\mathfrak{N}:=\{x\in \mathscr{R} \ ; \ x \ {\rm \ is \ bounded} \}.
\end{equation*} 
Since $0,1 \in \mathfrak{N}$, $\mathfrak{N}$ is not empty.
We first show that $\mathfrak{N}$ is a von Neumann algebra.
It is clear that $\mathfrak{N}$ is a subalgebra of $\M$.
Thus it is enough to check that $\mathfrak{N}$ is closed 
with respect to the strong* operator topology.
Let $\{x_{\alpha}\}$ be a net in $\mathfrak{N}$ converging to $x \in \mathfrak{N}$ 
with respect to the strong* operator topology.
So we have 
\begin{equation*}
\re{x_{\alpha}} \longrightarrow \re{x}, \ \ \ \ \im{x_{\alpha}} \longrightarrow \im{x}
\end{equation*}
with respect to the strong* operator topology.
By Lemma \ref{AcoreSRT}, 
\begin{equation*}
\re{x_{\alpha}} \longrightarrow \re{x}, \ \ \ \ \im{x_{\alpha}} \longrightarrow \im{x}
\end{equation*}
with respect to the strong resolvent topology.
As $\re{x_{\alpha}} \in \mathscr{R}$, $\im{x_{\alpha}} \in \mathscr{R}$ 
and $\mathscr{R}$ is SRT-closed, we see that 
$\re{x} \in \mathscr{R}$ and $\im{x} \in \mathscr{R}$.
Therefore
\begin{equation*}
x=\re{x}+i\im{x}\in\mathscr{R}.
\end{equation*}
Since $x$ is bounded, $x$ belongs to $\mathfrak{N}$.
Thus $\mathfrak{N}$ is a von Neumann algebra.

Next, we show that $\mathscr{R} \subset \overline{\mathfrak{N}}$.
Let $A$ be an element of $\mathscr{R}$. 
It is enough to consider the case that $A$ is self-adjoint.
Put 
\begin{equation*}
\mathcal{C}_{A}:= \bigcup_{n=1}^{\infty} \ran{E_A([-n, n])},
\end{equation*}
where $E_A(\cdot)$ is the spectral measure of $A$. 
$\mathcal{C}_{A}$ is completely dense and all elements of $\mathcal{C}_{A}$ are 
entire analytic vectors for $A$.
Thus we have for all $\xi \in \mathcal{C}_{A}$,
\begin{equation*}
e^{itA} = \lim_{J \rightarrow \infty} \sum_{j=1}^{J}\frac{(itA)^j}{j!}\xi.
\end{equation*}
Therefore the sequence 
\begin{equation*}
\left\{\overline{\sum_{j=1}^{J}\frac{(itA)^j}{j!}}\right\}_{J=1}^{\infty} \subset\mathscr{R}\subset\cM
\end{equation*}
converges almost everywhere to $e^{itA}$.
By Proposition \ref{3almostsrt}, 
it converges to $e^{itA}$ with respect to the strong resolvent topology.
Since $\mathscr{R}$ is SRT-closed and $e^{itA}$ is bounded, we get $e^{itA} \in \mathfrak{N}$.
This implies $A$ belongs to $\overline{\mathfrak{N}}$.

On the other hand, by the definition of $\mathfrak{N}$, $\mathfrak{N}\subset\mathscr{R}$.
Since $\overline{\mathfrak{N}}$ is a SRT-closure of $\mathfrak{N}$, 
we see that $\overline{\mathfrak{N}} \subset \mathscr{R}$.
Thus we conclude that $\overline{\mathfrak{N}} = \mathscr{R}$.

Finally, we show that the uniqueness for $\mathfrak{N}$.
Let $\mathfrak{L}$ be a von Neumann subalgebra of $\M$ satisfying 
$\overline{\mathfrak{L}} = \mathscr{R}$.
Then, we have
\begin{align*}
\mathfrak{N} &= \{x\in \overline{\mathfrak{N}} \ ; \ x \ {\rm \ is \ bounded} \} \\
&= \{x\in \mathscr{R} \ ; \ x \ {\rm \ is \ bounded} \} \\
&= \{x\in \overline{\mathfrak{L}} \ ; \ x \ {\rm \ is \ bounded} \} \\
&= \mathfrak{L}.
\end{align*}
Thus $\mathfrak{N}$ is unique.
\end{proof}

\begin{cor}\label{5cor}
Let $\M$ be a finite non Neumann algebra on a Hilbert space $\Hs$, 
$\mathfrak{g}$ be a real SRT-closed Lie subalgebra of $\mathfrak{u}(\M)$.
Then the following are equivalent:
\begin{list}{}{}
\item[(1)] there exists a von Neumann subalgebra $\mathfrak{N}$ of $\M$ 
such that $\mathfrak{g} = \mathfrak{u}(\mathfrak{N})$,
\item[(2)] $1_{\Hs}\in\mathfrak{g}$ and for all $A, B\in\mathfrak{g}$, $i\left(\overline{AB+BA}\right)\in\mathfrak{g}$.
\end{list}
In the above case, $\mathfrak{N}$ is unique.
\end{cor}

\begin{proof}
First of all, we shall show $(1) \Rightarrow (2)$.
Since $\mathfrak{u}(\mathfrak{N}) \subset \overline{\mathfrak{N}}$, 
we have $i\left(\overline{AB+BA}\right)\in\overline{\mathfrak{N}}$.
On the other hand, 
\begin{equation*}
\mathfrak{u}(\mathfrak{N}) = \{X\in\overline{\mathfrak{N}} \ ; \ X^*=-X\}.
\end{equation*}
Thus $i\left(\overline{AB+BA}\right)\in\mathfrak{u}(\mathfrak{N})=\mathfrak{g}$.
Next we shall show $(2) \Rightarrow (1)$. 
By direct computations, we see that 
\begin{equation*}
\mathscr{R} := \left\{\overline{X+iY}\in\overline{\mathfrak{M}} \ ; \ X, \ Y \in \mathfrak{g}\right\}
\end{equation*}
is a SRT-closed *-subalgebra of $\overline{\mathfrak{M}}$.
Thus, by Proposition \ref{5main}, there is a von Neumann subalgebra $\mathfrak{N}$ of 
$\mathfrak{M}$ such that $\mathscr{R}=\overline{\mathfrak{N}}$.
Then we see that
\begin{align*}
\mathfrak{g} &= \left\{X\in\mathscr{R} \ ; \ X^*=-X \right\} \\ 
&= \left\{X\in\overline{\mathfrak{N}} \ ; \ X^*=-X \right\} \\ 
&= \mathfrak{u}(\mathfrak{N}).
\end{align*}
Finally, we show the uniqueness for $\mathfrak{N}$.
Let $\mathfrak{L}$ be a von Neumann subalgebra of $\M$ satisfying 
$\mathfrak{u}(\mathfrak{L}) = \mathfrak{g}$.
Then, we have
\begin{align*}
\overline{\mathfrak{N}} 
&= \left\{\overline{X+iY} \ ; \ X, \ Y \in \mathfrak{u}(\mathfrak{N})\right\} \\
&= \left\{\overline{X+iY} \ ; \ X, \ Y \in \mathfrak{u}(\mathfrak{L})\right\} \\
&= \overline{\mathfrak{L}} 
\end{align*}
By the uniqueness of Proposition \ref{5main}, we get $\mathfrak{N}=\mathfrak{L}$.
\end{proof}

\section{Categorical Characterization of $\cM$}

\subsection{Introduction}

In this last section we turn the point of view and consider some categorical aspects of the *-algebra $\cM.$ 
Especially, we determine when a *-algebra $\mathscr{R}$ of unbounded operators on a Hilbert space $\Hs$ 
turns out to be of the form $\cM$, without any reference to von Neumann algebraic structure in advance. 
As is well known, there are many examples of *-algebra of unbounded operators that is not of the form $\cM$. 
For example, many O$^*$-algebras \cite{Schmuedgen} are not related to any affiliated operator algebra. 
Therefore, the appropriate condition to single out suitable class of *-algebras of unbounded operators are necessary. 
For this purpose, we define the category \textbf{fRng} of unbounded operator algebras 
and compare this category with the category \textbf{fvN} of finite von Neumann algebras 
and show that both of them have natural tensor category structures (cf. Appendix C). 
Furthermore, we will see that they are isomorphic as a tensor category, 
in spite of the fact that the object in \textbf{fRng} is not locally convex in general while the one in \textbf{fvN} is a Banach space. 
However, the algebraic structures of $\cM$ and $\M$ are very similar and in fact they constitute isomorphic categories. 
To begin with, let us introduce the structure of tensor category into \textbf{fRng}.

\subsection{\textbf{fvN} and \textbf{fRng} as Tensor Categories}

Now we turn to the question of characterizing the category \textbf{fRng} of *-algebras of unbounded 
operators which are realized as $\cM$, where $\M$ is a von Neumann algebra acting on a Hilbert space. 
It is well known 
that the usual tensor product $(\M_1,\M_2)\mapsto \M_1\vtensor \M_2$ of von Neumann algebras and 
the tensor product of $\sigma$-weakly continuous homomorphisms $(\phi_1,\phi_2)\mapsto \phi_1\tens \phi_2$ makes the category 
of finite von Neumann algebras a tensor category. 
Therefore we define:

\begin{dfn} 
The category \textbf{fvN} is a category whose objects are pairs $(\M,\mc{H})$ of a finite von Neumann algebra $\M$ 
acting on a Hilbert space $\Hs$ and whose morphisms are $\sigma$-weakly continuous unital *-homomorphisms. 
The unit object is $(\mathbb{C}1_{\mathbb{C}},\mathbb{C})$. 
The tensor functor is the usual tensor product functor of von Neumann algebras. 
The definition of left and right unit constraint functors might be obvious.
\end{dfn}

If we are to characterize the objects in \textbf{fRng}, 
we must settle some subtleties due to the fact that we cannot use von Neumann algebraic structure from the outset. 
However, this difficulty can be overcome thanks to the the notion of the strong resolvent topology and the resolvent class 
whose definitions are independent of von Neumann algebras (See $\S3$). 
We define \textbf{fRng} as follows.

\begin{dfn}\label{6def of fRng} 
The category \textbf{fRng} is a category whose objects $(\msc{R},\mc{H})$ consist of 
a SRT-closed subset $\msc{R}$ of the resolvent class $\msc{RC}(\mc{H})$ on a Hilbert space $\mc{H}$ with the following properties:
\begin{list}{}{}
\item[(1)] $X+Y$ and $XY$ are closable for all $X$, $Y\in\msc{R}$.
\item[(2)] $\overline{X+Y}$, $\overline{\alpha X}$, $\overline{XY}$ and $X^*$ again belong to 
$\msc{R}$ for all $X$, $Y\in\msc{R}$ and $\alpha\in\mathbb{C}$.
\item[(3)] $\msc{R}$ forms a *-algebra with respect to the sum $\overline{X+Y}$, the scalar multiplication $\overline{\alpha X}$,
the multiplication $\overline{XY}$ and the involution $X^*$.
\item[(4)] $1_{\Hs}\in \msc{R}$.
\end{list}
The morphism set between $(\msc{R}_1,\Hs_1)$ and $(\msc{R}_2,\Hs_2)$ consists of SRT-continuous unital *-homomorphisms 
from $\msc{R}_1$ to $\msc{R}_2$.
\end{dfn}

\begin{rem}\label{6*alg} 
From the definition of \textbf{fRng}, it is not clear whether, for each objects in \tb{fRng}, 
its algebraic operations are continuous or not.
However, the next lemma shows that $\msc{R}$ is a complete topological *-algebra.
\end{rem}

\begin{lem}\label{6characterization of R} 
Let ($\msc{R}, \Hs$) be an object in \textbf{fRng}. 
Then there exists a unique finite von Neumann algebra $\M$ on $\Hs$ such that $\msc{R}=\cM$. 
Furthermore, $\M=\msc{R}\cap\mathfrak{B}(\Hs)$ holds. 
\end{lem}

\begin{proof}
Define $\M:=\msc{R}\cap \mf{B}(\mc{H})$. 
Then one can prove that $\M$ is von Neumann algebra by the same way as in Proposition \ref{5main}.

We next show that $\msc{R}\subset\cM$.
Let $A\in \msc{R}$ be a self-adjoint operator.
Define the dense subspace $C_A$ according to the spectral decomposition of $A$:
\[C_A:=\Bcu_{n=1}^{\infty}\ran{E_A([-n,n])},\]
where
\[A=\int_{\mathbb{R}} \lambda dE_A(\lambda)\]
is the spectral decomposition of $A$.  
Since all $\xi \in C_A$ is an entire analytic vector for $A$, we have 
\[\disp e^{itA}\xi=\limn \sum_{k=0}^n\frac{(itA)^k}{k!}\xi,\]
for all $t\in\mathbb{R}$.  
Let $\M_A$ be a von Neumann algebra generated by $\{E_A(J)\ ;\ J\in \mc{B}(\mathbb{R})\}$,
where $\mc{B}(\mathbb{R})$ is the one dimensional Borel field.
Since $\M_A$ is abelian, it is a finite von Neumann algebra. 
It is also clear that 
\[B_n:=\overline{\sum_{k=0}^n\frac{(itA)^k}{k!}}\in \overline{(\M_A)}\cap \msc{R}\] 
and $e^{itA}\in \M_A$.
Since $C_A$ is completely dense for $\M_A$, $B_n$ converges almost everywhere to $e^{itA}$ in $\overline{(\M_A)}$. 
As $\M_A$ is finite, we see that $B_n$ converges to $e^{itA}$ in the strong resolvent topology. 
On the other hand, $\msc{R}$ is SRT-closed and therefore $e^{itA}\in \msc{R}\cap \mf{B}(\mc{H})=\M$, for all $t\in\mathbb{R}$.
This implies $A\in \cM$. 
For a general operator $B\in \msc{R}$, using real-imaginary part decomposition $B=\overline{\re{B}+i\im{B}}$, 
we have $B\in \cM$. 

We shall show that $\cM\subset\msc{R}$.
Let $A\in\cM$ and $A=U|A|$ be its polar decomposition, 
then $U\in\M\subset\msc{R}$ and $|A|\in\cM$.
Let $|A|=:\int_{0}^{\infty}\lambda dE_{|A|}(\lambda)$ be the spectral decomposition of $|A|$.
Put 
\[x_n:=\int_{0}^{n}\lambda dE_{|A|}(\lambda)\in\M\subset\msc{R},\] 
then $x_n$ converges to $|A|$ in the strong resolvent topology.
Thus $|A|\in\msc{R}$.
Therefore $A=U|A|\in\msc{R}$.

The finiteness of $\M$ follows immediately from Theorem \ref{23converse}.
\end{proof}

Note that for each finite von Neumann algebra $\M$ on a Hilbert space $\Hs$, $(\cM,\Hs)$ is an object in \textbf{fRng}.
 
The main result of this section is the next theorem.

\begin{thm}\label{6main theorem} 
The category \textbf{fRng} is a tensor category. 
Moreover, \textbf{fRng} and \textbf{fvN} are isomorphic as a tensor category.
\end{thm}

To prove this theorem, we need many lemmata. 
The proof is divided into several steps.

Next, we will define the tensor product $\msc{R}_1\vtensor \msc{R}_2$ of objects 
$\msc{R}_i\ (i=1,2)$ in \tb{fRng} (cf. Definition \ref{6def of tensor in fRng}). 
For this purpose, let us review the notion of the tensor product of closed operators.  
Let $A,B$ be densely defined closed operators on Hilbert spaces $\mc{H},\ \mc{K}$, respectively. 
Let $A\tens_0B$ be an operator defined by 
\begin{align*}
&\dom{A\tens_0B}:=\dom{A}\atensor \dom{B}, \\
&(A\tens_0B)(\xi\tens \eta):=A\xi\tens B\eta,\ \ \ \ \xi \in \dom{A},\ \eta \in \dom{B}.
\end{align*}
It is easy to see that $A\tens_0B$ is closable and denote its closure by $A\tens B$.
 
\begin{lem}\label{6affiliation of tensor product}
Let $\M_1$, $\M_2$ be finite von Neumann algebras acting on Hilbert spaces $\Hs_1$, $ \Hs_2$, respectively. 
Let $A\in \cM_1$ and $B\in \cM_2$. 
Then we have $A\tens B\in \ol{\M_1\vtensor \M_2}$.
\end{lem}

\begin{proof}
Let $x_i\in \M_i'\ (i=1,2)$. 
For any $\xi\in \dom{A}$ and $\eta\in \dom{B}$, we have $(x_1\tens x_2)(\xi\tens \eta)\in \dom{A\tens_0B}$ and 
\[\{(x_1\tens x_2)(A\tens_0B)\}(\xi\tens \eta)=\{(A\tens_0B)(x_1\tens x_2)\}(\xi\tens \eta).\]
Therefore, by the linearity, we have $(x_1\tens x_2)(A\tens_0B)\subset (A\tens _0B)(x_1\tens x_2)$. 
Since $(\M_1\vtensor \M_2)'=\M_1'\vtensor \M_2'$ is the strong closure of $\M_1'\atensor \M_2'$, we have
\[y(A\tens_0B)\subset (A\tens B)y,\ \text{for all}\ y\in (\M_1\vtensor \M_2)'.\]
Therefore by the limiting argument, we have $y(A\tens B)\subset (A\tens B)y$, 
which implies $A\tens B$ is affiliated with $\M_1\vtensor\M_2$.
\end{proof}

\begin{lem}\label{6tensor of cores} 
Let $A,B$ be densely defined closed operators on Hilbert spaces $\mc{H},\ \mc{K}$ 
with cores $\mc{D}_A,\mc{D}_B$ respectively. 
Then $\mc{D}:=\mc{D}_A\atensor \mc{D}_B$ is a core of $A\otimes B$.
\end{lem}

\begin{proof}
From the definition of $A\otimes B$, 
for any $\zeta \in \dom{A\otimes B}$ and for any $\eps>0$, 
there exists some $\zeta_{\eps}=\sum_{i=1}^n\xi_i\otimes \eta_i\in \dom{A}\atensor \dom{B}$ 
such that
\begin{equation*}
||\zeta-\zeta_{\eps}||<\eps,\hspace{0.4cm} ||(A\otimes B)\zeta-(A\otimes B)\zeta_{\eps}||<\eps.
\end{equation*}
Put 
\[C:=\max_{1\le i\le n}\{||\xi_i||,\ ||A\xi_i||\}+1>0.\]
Since $D_B$ is a core of $B$, there exists $\eta_i^{\eps}\in D_B$ such that 
\[||\eta_i-\eta_i^{\eps}||<\frac{\eps}{nC},\hspace{0.4cm} ||B\eta_i-B\eta_i^{\eps}||<\frac{\eps}{nC}.\]
Put 
\[C':=\max_{1\le i\le n}\{||\eta_i^{\eps}||,\ ||B\eta_i^{\eps}||\}+1>0.\]
Similarly, since $\mc{D}_A$ is a core of $A$, there exists $\xi_i^{\eps}\in \mc{D}_A$ such that 
\[ ||\xi_i-\xi_i^{\eps}||<\frac{\eps}{nC'},\hspace{0.4cm} ||A\xi_i-A\xi_i^{\eps}||<\frac{\eps}{nC'}.\]
Define $\disp \zeta^{\eps}:=\sum_{i=1}^n\xi_i^{\eps}\otimes \eta_i^{\eps}\in \mc{D}$.
Then we have
\eqa{
||\zeta-\zeta^{\eps}||&\le ||\zeta-\zeta_{\eps}||+||\zeta_{\eps}-\zeta^{\eps}||\\
&\le \eps +\sum_{i=1}^n||\xi_i\otimes \eta_i-\xi_i^{\eps}\otimes \eta_i^{\eps}||\\
&\le \eps +\sum_{i=1}^n||\xi_i\otimes \eta_i-\xi_i\otimes \eta_i^{\eps}||+\sum_{i=1}^n||\xi_i\otimes \eta_i^{\eps}-\xi_i^{\eps}\otimes \eta_i^{\eps}||\\
&\le \eps +\sum_{i=1}^n\norm{\xi_i}\norm{\eta_i-\eta_i^{\eps}}+\sum_{i=1}^n\norm{\xi_i-\xi_i^{\eps}}\norm{\eta_i^{\eps}}\\
&\le \eps +\sum_{i=1}^nC\cdot \frac{\eps}{nC}+\sum_{i=1}^n\frac{\eps}{nC'}\cdot C'\\
&=3\eps.\ \cdots (*)
}
Furthermore,
\eqa{
&||(A\otimes B)\zeta-  (A\otimes B)\zeta^{\eps}||\\
&\le \norm{(A\otimes B)\zeta-(A\otimes B)\zeta_{\eps}}+\norm{(A\otimes B)\zeta_{\eps}-(A\otimes B)\zeta^{\eps}}\\
&\le \eps+\sum_{i=1}^n\norm{A\xi_i\otimes B\eta_i-A\xi_i^{\eps}\tens B\eta_i^{\eps}}\\
&\le \eps+\sum_{i=1}^n\norm{A\xi_i\tens B\eta_i-A\xi_i\tens B\eta_i^{\eps}}+\sum_{i=1}^n\norm{
A\xi_i\tens B\eta_i^{\eps}-A\xi_i^{\eps}\tens B\eta_i^{\eps}}\\
&\le \eps +\sum_{i=1}^n\norm{A\xi_i}\norm{B\eta_i-B\eta_i^{\eps}}+\sum_{i=1}^n\norm{A\xi_i-A\xi_i^{\eps}}\norm{B\eta_i^{\eps}}\\
&\le \eps +\sum_{i=1}^nC\cdot \frac{\eps}{nC}+\sum_{i=1}^n\frac{\eps}{nC'}\cdot C'\\
&=3\eps\ \cdots (**)
}
$(*)$ and $(**)$ implies $\mc{D}$ is a core of $A\tens B$.
\end{proof}

Next lemma says that the tensor product of algebras of affiliated operators has a natural *-algebraic structures.

\begin{lem}\label{6tensor *-alg}
Let $\M,\ \mf{N}$ be finite von Neumann algebras acting on Hilbert spaces $\mc{H},\ \mc{K}$ respectively.
Let $A,C\in \cM,\ B,D\in \cN$.
Then we have
\begin{list}{}{}
\item[(1)] $\overline{(A\tens B)(C\tens D)}=\overline{AC}\tens \overline{BD}$.
\item[(2)] $(A\tens B)^*=A^*\tens B^*.$
\item[(3)] $\ol{A+C}\tens \ol{B+D}=\ol{A\tens B+A\tens D+C\tens B+C\tens D}.$
\item[(4)] $\ol{\lambda (A\tens B)}=\ol{\lambda A}\tens B=A\tens \ol{\lambda B}\ (\lambda \in \mbb{C}).$
\end{list}
\end{lem}

\begin{proof}
(1)  From Proposition \ref{2M-vN1}, $\mc{D}_1:=\{\xi \in \dom{C};C\xi\in \dom{A}\}$ 
is a core of $\overline{AC}$ and $\mc{D}_2:=\{\eta\in \dom{D};D\eta\in \dom{B}\}$ 
is a core of $\overline{BD}$. 
Define $\mc{D}:=\mc{D}_1\atensor \mc{D}_2$, which is a core of $\overline{AC}\tens \overline{BD}$.
Since 
\[\disp \dom{\overline{(A\tens B)(C\tens D)}}\supset \dom{(A\tens B)(C\tens D)}\supset \mc{D},\]
it holds that for any $\disp \zeta=\sum_{i=1}^n\xi_i\tens \eta_i\in \mc{D}$, we have
\eqa{
\overline{(A\tens B)(C\tens D)}\zeta&=\sum_{i=1}^nAC\xi_i\tens BD\eta_i=(\overline{AC}\tens \overline{BD})\sum_{i=1}^n\xi_i\tens \eta_i\\
&=(\overline{AC}\tens \overline{BD})\zeta.
}
Therefore $\disp \overline{(A\tens B)(C\tens D)}\supset (\overline{AC}\tens \overline{BD})|_{\mc{D}}$.
Since $\mc{D}$ is a core of $\overline{AC}\tens \overline{BD}$, we have (by taking the closure)
\[\overline{(A\tens B)(C\tens D)}\supset \overline{AC}\tens \overline{BD}.\]
Since both operators belong to $\overline{\M\vtensor \mf{N}}$ by Lemma \ref{6affiliation of tensor product}, 
we have 
\[\overline{(A\tens B)(C\tens D)}=\overline{AC}\tens \overline{BD}.\] 
by Proposition \ref{2No-closed-extension}(2).

(2) It is easy to see that $(A\tens B)^*\supset A^*\tens B^*$. 
Since $(A\tens B)^*$ and $A^*\tens B^*$ are closed operators belonging to $\ol{\M\vtensor \mf{N}}$, 
we have $(A\tens B)^*=A^*\tens B^*$ by Proposition \ref{2No-closed-extension} (2).

(3) and (4) can be easily shown in a similar manner as in (1).
\end{proof}

Now we shall define the tensor product $\msc{R}_1\vtensor \msc{R}_2$ of $(\msc{R}_1, \Hs_1)$ and 
$(\msc{R}_2, \Hs_2)$ in Obj(\textbf{fRng}).
Let $\M_i$ be finite von Neumann algebras on $\Hs_i$ such that $\msc{R}_i=\cM_i\ (i=1,2)$, 
respectively (cf. Lemma \ref{6characterization of R}). 
From Lemma \ref{6tensor *-alg}, the linear space $\msc{R}_1\tens_{\text{alg}}\msc{R}_2$ spanned by 
$\{A_1\tens A_2\ ;\ A_i\in \msc{R}_i,\ i=1,2\}$ is a *-algebra. 
Since $\msc{R}_1\tens_{\text{alg}}\msc{R}_2$ is a subset of $\ol{\M_1\vtensor \M_2}$, it belongs to $\msc{RC}(\Hs_1\tens \Hs_2)$. 
Therefore:

\begin{dfn}\label{6def of tensor in fRng}
Under the above notations, we define $\msc{R}_1\vtensor \msc{R}_2$ to be the SRT-closure 
(for $\mc{H}_1\tens \mc{H}_2$) of $\msc{R}_1\atensor \msc{R}_2$.
\end{dfn}

\begin{lem}\label{6tensor product of fRng}
Let $\msc{R}_i$ $(i=1,2)$ be as above. 
Then $\msc{R}_1\vtensor \msc{R}_2$ is also an object in \textbf{fRng}. 
More precisely, if $\msc{R}_i=\cM_i$, where $\M_i$ is a finite von Neumann algebra $(i=1,2)$, 
then $\cM_1\vtensor \cM_2=\ol{\M_1\vtensor\M_2}$.
\end{lem}

\begin{proof}
We first show that $\cM_1\vtensor \cM_2 \subset \ol{\M_1\vtensor \M_2}$. 
Let $T_i\in \cM_i\ (i=1,2)$. 
Then we can show that $T_1\tens T_2\in \ol{\M_1\vtensor \M_2}$ by Lemma \ref{6affiliation of tensor product}. 
Therefore by the linearity, we obtain 
\[\cM_1\atensor \cM_2\subset \ol{\M_1\vtensor \M_2} .\]
As the left hand side is SRT-closed in $\rc{\Hs_1\tens \Hs_2}$, 
we have $\cM_1\vtensor \cM_2 \subset \ol{\M_1\vtensor \M_2}$.
Next we prove that $\ol{\M_1\vtensor \M_2}\subset \cM_1\vtensor \cM_2$. 
It is clear that 
$\M_1\otimes_{\text{alg}} \M_2 \subset \cM_1 \overline{\otimes} \cM_2$.
By the Kaplansky density theorem and Lemma \ref{3mt=sot}, 
we have $\M_1\overline{\otimes}\M_2 \subset \cM_1 \overline{\otimes} \cM_2$.
By taking the SRT-closure, we obtain $\overline{\M_1\overline{\otimes}\M_2}\subset \cM_1\vtensor \cM_2$. 
\end{proof}

The above Lemma says that $(\msc{R}_1\overline{\otimes}\msc{R}_2, \Hs_1\otimes \Hs_2)$ is again an object in \textbf{fRng}.

Next, we discuss the extension of morphisms in \tb{fvN} to ones in \tb{fRng}. 
It requires some steps.

\begin{lem}\label{6SRT convergence of tensor product}
Let $(\M_1,\mc{H}_1)$, $(\M_2,\mc{H}_2)$ be finite von Neumann algebras. 
Then the mapping 
\begin{align*}
(\overline{\M_1},SRT)\times (\overline{\M_2},SRT) &\longrightarrow 
(\overline{\M_1}\vtensor\overline{\M_2},SRT), \\
(A,B)&\longmapsto A\tens B,
\end{align*}
is continuous.
\end{lem}

\begin{proof}
Let $\{A_{\alpha}\}_{\alpha}\subset\overline{\M_1}$, 
$\{B_{\alpha}\}_{\alpha}\subset\overline{\M_2}$ be SRT-converging nets 
and $A\in\overline{\M_1}$, $B\in\overline{\M_2}$ be their limits respectively.
We should only show that the net $\left\{A_{\alpha}\tens B_{\alpha}\right\}_{\alpha}$ converges to $A\tens B$ 
in the strong resolvent topology.\\

{\bf Step 1.} The above claim is true if all $A_{\alpha}$, $B_{\alpha}$, $A$ and $B$ are self-adjoint.
Indeed, since 
\begin{equation*}
e^{it(A_{\alpha}\tens 1)} = e^{itA_{\alpha}}\tens 1, \ \ \ \ e^{it(A\tens 1)}=e^{itA}\tens 1
\end{equation*}
hold, we easily see that $A_{\alpha}\tens 1$ converges to $A\tens 1$ in the strong exponential topology.
Thus, by Theorem \ref{3cta}, the SRT-convergence of $A_{\alpha}\tens 1$ to $A\tens 1$ follows.
Similarly $1\tens B_{\alpha}$ converges to $1\tens B$ in the strong resolvent topology.
Therefore, by Lemma \ref{6tensor *-alg} and the SRT-continuity of the multiplication, we have
\begin{equation*}
A_{\alpha}\tens B_{\alpha} = \overline{\left(A_{\alpha}\tens 1\right)\left(1\tens B_{\alpha}\right)} 
\rightarrow \overline{\left(A\tens 1\right)\left(1\tens B\right)} 
= A\tens B.
\end{equation*}
\\

{\bf Step 2.} In a general case, by Lemma \ref{6tensor *-alg}, we obtain
\begin{align*}
A_{\alpha}\tens B_{\alpha} &= \left(\overline{\re{A_{\alpha}}+i\im{A_{\alpha}}}\right)\tens \left(\overline{\re{B_{\alpha}}+i\im{B_{\alpha}}}\right) \\
&= \overline{\re{A_{\alpha}}\tens\re{B_{\alpha}} + i\re{A_{\alpha}}\tens\im{B_{\alpha}}} \\
&\ \ \ \ \ \ \ \ \overline{+ i\im{A_{\alpha}}\tens\re{B_{\alpha}}-\im{A_{\alpha}}\tens\im{B_{\alpha}}}\\
&\rightarrow \overline{\re{A}\tens\re{B} + i\re{A}\tens\im{B}} \\
&\ \ \ \ \ \ \ \ \overline{+ i\im{A}\tens\re{B}-\im{A}\tens\im{B}}\\
&= A\tens B.
\end{align*}
Hence the proof of Lemma \ref{6SRT convergence of tensor product} is complete.
\end{proof}

\begin{lem}\label{6reduction is finite}
Let $\mathfrak{M}$ be a finite von Neumann algebra on a Hilbert space $\mathcal{H}$ 
and $e$ is a projection in $\mathfrak{M}'$, then $\mathfrak{M}_{e}$ is also finite. 
\end{lem}

\begin{proof}
Well-known.
\end{proof}

\begin{lem}\label{6reduction of closed op}
Let $A$ be a densely defined closed operator on a Hilbert space $\mc{H}$, $\mc{K}$ be a closed subspace of $\mc{K}$ such that 
$P_{\mc{K}}A\subset AP_{\mc{K}}$. 
Then the operator $B:=A|_{\dom{A}\cap \mc{K}}$ is a densely defined closed operator on $\mc{K}$.
\end{lem}

\begin{proof}
This is a straightforward verification.
\end{proof}

The next proposition guarantees the existence and the uniqueness of the extension of morphisms in \tb{fvN} to the morphisms in \tb{fRng}. 
Note that the claim is not trivial, 
because many $\sigma$-weakly continuous linear mappings between finite von Neumann algebras cannot be extended 
SRT-continuously to the algebra of affiliated operators. 
Indeed, we can not extend any $\sigma$-weakly continuous state on a finite von Neumann algebra $\M$ SRT-continuously onto $\cM$ if $\M$ is diffuse. 

\begin{prop}\label{6exntension of morphisms} 
Let $\M_1,\M_2$ be finite von Neumann algebras on Hilbert spaces $\Hs_1$, $\Hs_2$ respectively.
\begin{list}{}{}
\item[(1)] For each SRT-continuous unital *-homomorphism $\Phi:\cM_1\to \cM_2$, 
the restriction $\varphi$ of $\Phi$ onto $\M_1$ 
is a $\sigma$-weakly continuous unital *-homomorphism from $\M_1$ to $\M_2$.
\item[(2)] Conversely, for each $\sigma$-weakly continuous unital *-homomorphism $\varphi:\M_1\to \M_2$, 
there exists a unique SRT-continuous unital *-homomorphism $\Phi:\cM_1\to \cM_2$ such that $\Phi|_{\M_1}=\varphi.$
\end{list}
\end{prop}

\begin{proof}
(1) We have to prove that $\Phi$ maps all bounded operators to bounded operators. 
For any $u\in U(\M_1)$ and $\xi\in \dom{\Phi(u)^*\Phi(u)}$, we have
\eqa{
||\Phi(u)\xi||^2&=\nai{\xi}{\Phi(u)^*\Phi(u)\xi}=\nai{\xi}{\Phi(u^*u)\xi}\\
&=\nai{\xi}{\Phi(1)\xi}=||\xi||^2.
}
Since $\dom{\Phi(u)^*\Phi(u)}$ is a (completely) dense subspace, $\Phi(u)\in \M_2$ and $\Phi(u)$ is an isometry. 
Therefore the finiteness of $\M_2$ implies $\Phi(u)\in U(\M_2)$. 
Thus, we see that $\Phi(U(\M_1))\subset U(\M_2)$. 
Since any element in $\M_1$ is a linear combination of $U(\M_1)$, $\Phi$ maps $\M_1$ into $\M_2$. 
To show that $\varphi$ is $\sigma$-weakly continuous, 
it is sufficient to prove the ($\sigma$-) strong continuity on the unit ball, because it is a homomorphism. 
Since the strong resolvent topology coincides with the strong operator topology 
on the closed unit ball by Lemma \ref{3mt=sot}, $\varphi$ is strongly continuous on the closed unit ball. 
Therefore $\varphi$ is a $\sigma$-weakly continuous homomorphism.

(2) Regard $\varphi$ as a composition of a surjection $\varphi': \M_1\to \varphi(\M_1)$ 
and the inclusion map $\iota:\ \varphi(\M_1)\hookrightarrow  \M_2.$ 
Note that the $\sigma$-weak continuity of $\varphi$ implies $\varphi(\M_1)$ is a von Neumann algebra. 
Since $\varphi'$ is surjective, from Theorem IV.5.5 of \cite{Tak}, 
there exists a Hilbert space $\mc{K}$, a projection $e'\in P(\M_1'\overline{\otimes}\mf{B}(\mc{K}))$ 
and a unitary operator $U:\ e'(\mc{H}_1\otimes \mc{K})\stackrel{\sim}{\to}\mc{H}_2$ such that 
\[\varphi'(x)=U(x\otimes 1_{\mc{K}})_{e'}U^*\]
for all $x\in \M_1$.
Now we would like to define the extension $\Phi'$ of $\varphi'$ to $\cM_1\to \ol{\varphi(\M_1)}$. 
Then we define $\Phi'$ as follows:
\[\Phi'(X)=U(X\otimes 1_{\mc{K}})_{e'}U^*,\ X\in \cM_1.\]
\[\xymatrix{
\cM_1\ar[d]^{\cdot \tens 1}\ar@{}[rrd]|\circlearrowleft \ar@{-->}[rr]^{\Phi'}& &\ol{\varphi(\M_1)}\\
\cM_1\tens \mbb{C}1_{\mc{K}}\ar[rr]^{\text{reduction by }e'}& &(\cM_1\otimes \mbb{C}1_{\mc{K}})_{e'}\ar[u]_{U\cdot U^*}
}\]
More precisely, we define  
\[Z=(X\tens 1)_{e'}:=e'(X\tens 1)|_{\ran{e'}\cap \dom{X\tens 1}},\ \ \ \  \Phi'(X):=UZU^*.\]
We have $Z\in \overline{(\M_1\otimes \mathbb{C}1_{\mathcal{K}})_{e'}}.$ 
Indeed, since $e'$ commutes with $\M\tens \mbb{C}1_{\mc{K}}$, it reduces the operator $X\tens 1$ 
and therefore by Lemma \ref{6reduction of closed op}, 
$(X\tens 1)_{e'}$ is a densely defined closed operator on $\ran{e'}$. 
Since $(\mf{N}_f)'=(\mf{N}')_{f}$ for each von Neumann algebra $\mf{N}$ and $f\in P(\mf{N}')$, the affiliation property is manifest. 
In addition, by Lemma \ref{6reduction is finite}, $(\M\tens \mbb{C}1_{\mc{K}})_{e'}$ is a finite von Neumann algebra. 
Next, we prove the map $\cM\ni X\mapsto (X\tens 1)_{e'}\in \overline{(\M_1\otimes \mathbb{C}1_{\mathcal{K}})_{e'}}$ 
is a SRT-continuous unital *-homomorphism. 
The continuity follows from Lemma \ref{6SRT convergence of tensor product}. 
To prove that it is a *-homomorphism, we have to show that for $X$, $Y\in \cM$, 
\begin{align*}
(\ol{(X+Y)}\tens 1)_{e'}&=\ol{(X\tens 1)_{e'}+(Y\tens 1)_{e'}},\\
(\ol{XY}\tens 1))_{e'}&=\ol{(X\tens 1)_{e'}(Y\tens 1)_{e'}},\\
((X\tens 1)_{e'})^*&=(X^*\tens 1)_{e'}.
\end{align*}
To prove the first equality, by Lemma \ref{6tensor *-alg}, we see that
\begin{align*}
\left((\overline{X+Y})\otimes 1\right)_{e^{'}} 
&= \left(\overline{X\otimes 1+Y\otimes 1}\right)_{e^{'}} \\
&\supset (X\otimes 1)_{e^{'}}+(Y\otimes 1)_{e^{'}}.
\end{align*}
Taking the closure, by Lemma \ref{2No-closed-extension}, we have
\begin{equation*}
\left((\overline{X+Y})\otimes 1\right)_{e^{'}} = \overline{(X\otimes 1)_{e^{'}}+(Y\otimes 1)_{e^{'}}}.
\end{equation*}
The others are proved in a similar manner. Next, by Lemma \ref{6UXU^*}, 
the correspondence $\cM_1\ni X\mapsto U(X\otimes 1_{\mc{K}})_{e'}U^*\in \ol{\varphi(\M_1)}\subset \cM_2$ 
defines a SRT-continuous unital *-homomorphism $\Phi'$ which is clearly an extension of $\varphi'$. 
Therefore by considering $\Phi:=\iota'\circ \Phi':\cM_1\to \cM_2$ is the desired extension of $\varphi$, 
where $\iota' : \Phi'(\cM_1)\hookrightarrow \cM_2$ is the mere inclusion. 
Finally, we prove the uniqueness of the extension. 
Let $\Psi$ be another SRT-continuous unital *-homomorphism such that $\Psi|_{\M_1}=\varphi$. 
Let $X\in \cM_1$. 
Then from the SRT-density of $\M_1$ in $\cM_1$, 
there exists a net $\{x_{\alpha}\}\subset \M_1$ such that $\lim_{\alpha} x_{\alpha}=X$ in the strong resolvent topology. 
Therefore we have
\eqa{
\Psi(X)&=\lim_{\alpha} \Psi(x_{\alpha})=\lim_{\alpha} \varphi(x_{\alpha})\\
&=\lim_{\alpha} \Phi(x_{\alpha})=\Phi(X).
}
\end{proof}

The next lemmata, together with Lemma \ref{6tensor product of fRng}, implies that \textbf{fRng} is a tensor category.

\begin{lem}\label{6tensor product of arrows in fRng} 
Let $\msc{R}_i$, $\msc{S}_i$ $(i=1,2)$ be objects in Obj(\textbf{fRng}). 
If $\Psi_1:\msc{R}_1\to \msc{S}_1$, $\Psi_2:\msc{R}_2\to \msc{S}_2$ are SRT-continuous unital *-homomorphisms, 
then there exists a unique SRT-continuous unital *-homomorphism 
$\Psi: \msc{R}_1\vtensor \msc{R}_2\to \msc{S}_1\vtensor \msc{S}_2$ such that 
$\Psi(A\tens B)=\Psi_1(A)\tens \Psi_2(B),$ for all $A\in \msc{R}_1$ and $B\in \msc{R}_2$. 
We define $\Psi_1\tens \Psi_2$ to be the map $\Psi$.
\end{lem}

\begin{proof}
Let $\psi_i$ be the restrictions of $\Psi_i$ onto $\mathfrak{M}_i$ $(i=1,2)$.
Then $\psi_i$  is a $\sigma$-weakly continuous unital *-homomorphism 
from $\mathfrak{M}_i$ to $\mathfrak{N}_i$, 
where $\overline{\mathfrak{N}_i}=\mathscr{S}_i$.
Thus there exists a $\sigma$-weakly continuous unital *-homomorphism $\psi$
from $\mathfrak{M}_1\overline{\otimes}\mathfrak{M}_2$ to 
$\mathfrak{N}_1\overline{\otimes}\mathfrak{N}_2$ such that
\begin{equation*}
\psi(x\otimes y) = \psi_1(x)\otimes\psi_2(y), 
\ \ \ \ x\in\mathfrak{M}_1,\ y\in\mathfrak{M}_2.
\end{equation*}
By Proposition \ref{6exntension of morphisms}, there exists a SRT-continuous unital *-homomorphism $\Psi$ 
from $\mathscr{R}_1\overline{\otimes}\mathscr{R}_2$ to 
$\mathscr{S}_1\overline{\otimes}\mathscr{S}_2$  
whose restriction to $\mathfrak{M}_1\overline{\otimes}\mathfrak{M}_2$ 
is equal to $\psi$.
For all $A\in\mathscr{R}_1$, $B\in\mathscr{R}_2$, we can take nets 
$\{x_{\alpha}\}_{\alpha}\subset\mathfrak{M}_1$, $\{y_{\alpha}\}_{\alpha}\subset\mathfrak{M}_2$ 
converging to $A$, $B$ in the strong resolvent topology, respectively.
Therefore, by Proposition \ref{6SRT convergence of tensor product}, we have
\begin{align*}
\Psi(A\otimes B)&=\lim_{\alpha}\Psi(x_{\alpha}\otimes y_{\alpha}) 
= \lim_{\alpha}\psi_1(x_{\alpha})\otimes \psi_2(y_{\alpha}) \\
&= \lim_{\alpha}\Psi_1(x_{\alpha})\otimes \Psi_2(y_{\alpha}) 
=\Psi_1(A)\otimes \Psi_2(B).
\end{align*}
\end{proof}

\begin{lem}\label{6associativity in fRng} 
Let $(\msc{R}_i,\Hs_i)\ (i=1,2,3)$ be objects in \tb{fRng}. 
Then we have a unique *-isomorphism which is homeomorphic with respect to the strong resolvent topology:
\eqa{
(\msc{R}_1\ol{\tens}\msc{R}_2)\ol{\tens}\msc{R}_3 & \cong \msc{R}_1\ol{\tens}(\msc{R}_2\ol{\tens}\msc{R}_3)\\ 
(X_1\tens X_2)\tens X_3 & \mapsto X_1\tens (X_2\tens X_3), \text{for all}\ X_i\in \msc{R}_i
} 
We denote the map as $\alpha_{\msc{R}_1,\msc{R}_2,\msc{R}_3}$. 
\end{lem}

\begin{proof}
Let $\M_i$ be a finite von Neumann algebra such that $\msc{R}_i=\cM_i$ $(i=1,2,3)$. 
Let $\alpha_0$ be the *-isomorphism from $(\M_1\vtensor \M_2)\vtensor \M_3$ 
onto $\M_1\vtensor (\M_2\vtensor \M_3)$ defined by $(x_1\tens x_2)\tens x_3\mapsto x_1\tens (x_2\tens x_3)$. 
By Lemma \ref{6tensor product of fRng}, both $(\cM_1\vtensor \cM_2)\vtensor \cM_3$ and 
$\cM_1\vtensor (\cM_2\vtensor \cM_3)$ are generated by $(\M_1\vtensor \M_2)\vtensor \M_3$ 
and $\M_1\vtensor (\M_2\vtensor \M_3)$, respectively. 
Therefore by Proposition \ref{6exntension of morphisms}, $\alpha_0$ can be extended to 
the desired *-isomorphism $\alpha_{\msc{R}_1,\msc{R}_2,\msc{R}_3}.$ 
\end{proof}

\begin{prop}\label{6fRng is a tensor category} 
\textbf{fRng} is a tensor category.
\end{prop}

\begin{proof}
We define the tensor product $\tens:\tb{fRng}\times \tb{fRng}\to \tb{fRng}$ by 
\[
(\msc{R}_1,\Hs_1)\tens (\msc{R}_2,\Hs_2):=(\msc{R}_1\vtensor \msc{R}_2,\Hs_1\tens \Hs_2)
\]
and for two morphisms $\Psi_i:(\msc{R}_i,\Hs_i)\to (\msc{S}_i,\mc{K}_i)\ (i=1,2)$, 
define $\Psi_1\tens \Psi_2$ according to Lemma \ref{6tensor product of arrows in fRng}. 
The unit object is $I:=(\mbb{C}1_{\mbb{C}},\mbb{C})$. 
The associative constraint $\alpha_{\msc{R}_1,\msc{R}_2,\msc{R}_3}$ is the map defined in Lemma \ref{6associativity in fRng}. 
The naturality of $\alpha_{\mathscr{R}_1,\mathscr{R}_2,\mathscr{R}_3}$ follows from Proposition \ref{6exntension of morphisms}. 
The definition of left (resp. right) constraint $\lambda_{\cdot}$ (resp. $\rho_{\cdot}$) might be clear. 
Now it is a routine task to verify that the data $(\tb{fRng},\tens,I,\alpha,\lambda,\rho)$ constitutes a tensor category.  
\end{proof}

Now we will prove that \textbf{fvN} is isomorphic to \textbf{fRng} as a tensor category. 
Define two functors $\mc{E}:$ \textbf{fvN}$\to $ \textbf{fRng}, $\mc{F}:$ \textbf{fRng}$\to $\textbf{fvN}.

\begin{dfn}
Define two correspondences $\mc{E}$, $\ \mc{F}$ as follows:
\begin{list}{}{}
\item[(1)] For each object $(\M,\mc{H})$ in \tb{fvN}, 
\[\mc{E}(\M,\mc{H}):=(\cM,\mc{H}),\]
which is an object in \tb{fRng}. 
For each morphism $\varphi:\M_1\to \M_2$ in \tb{fvN}, $\mc{E}(\varphi):\cM_1\to \cM_2$ 
is the unique SRT-continuous extension of $\varphi$ to $\cM_1$, 
so that $\mc{E}(\varphi)$ is a morphism in \tb{fRng} by Proposition \ref{6exntension of morphisms}.
\item[(2)] For each object $(\msc{R},\mc{H})$ in \tb{fRng}, 
\[\mc{F}(\msc{R},\mc{H}):=(\msc{R}\cap \mf{B}(\mc{H}),\mc{H}).\]
For each morphism $\Phi:\msc{R}_1\to \msc{R}_2$ in \tb{fRng}, 
$\mc{F}(\Phi):=\Phi|_{\msc{R}_1\cap \mf{B}(\Hs)}$, 
which is a morphism in \tb{fvN} by Proposition \ref{6exntension of morphisms}.
\end{list}
\end{dfn}

\begin{lem}\label{6 E and F are tensor functors}
$\mc{E}$ and $\mc{F}$ are tensor functors.
\end{lem}

\begin{proof}
We define the tensor functor $(\mc{E},h_1,h_2)$, where 
\eqa{
&h_1: \ (\mathbb{C}1_{\mathbb{C}}, \mathbb{C}) \stackrel{\text{id}}{\longrightarrow} 
(\mathbb{C}1_{\mathbb{C}}, \mathbb{C}) = 
\mathcal{E}((\mathbb{C}1_{\mathbb{C}}, \mathbb{C})),\\
&h_2((\M_1,\Hs_1),(\M_2,\Hs_2)): \cM_1\overline{\otimes}\cM_2 \stackrel{\text{id}}
{\longrightarrow} \overline{\M_1\overline{\otimes}\M_2},
}
can be taken to be identity morphisms thanks to Lemma \ref{6tensor product of arrows in fRng}. 
It is clear that $\mc{E}(1_{\M})=1_{\cM},$ where $1_{\M}$ and $1_{\cM}$ are identity map of $\M$ and $\cM$, respectively.
Let $\M_1\stackrel{\varphi_1}{\longrightarrow }\M_2\stackrel{\varphi_2}{\longrightarrow }\M_3$ be a sequence of morphisms in \tb{fvN}. 
Let $x\in \M_1$. 
It holds that
\eqa{
\mc{E}(\varphi_2\circ \varphi_1)(x)
&=(\varphi_2\circ \varphi_1)(x)=\mc{E}(\varphi_2)(\varphi_1(x))\\
&=\mc{E}(\varphi_2)(\mc{E}(\varphi_1)(x))=\left \{\mathcal{E}(\varphi_2)\circ \mathcal{E}(\varphi_1)\right \}(x). 
}
By Proposition \ref{6exntension of morphisms} (2), 
we have $\mc{E}(\varphi_2\circ \varphi_1)=\mc{E}(\varphi_2)\circ \mc{E}(\varphi_1)$. 
Therefore $\mc{E}$ is a functor. 
The conditions for $(\mc{E},h_1,h_2)$ to be a tensor functor are described as the following three diagrams, 
the commutativity of which are almost obvious by Proposition \ref{6exntension of morphisms} 
and ``$\sim $" symbols are followed from Lemma \ref{6associativity in fRng}.
\[
\xymatrix{
(\cM_1\vtensor \cM_2)\vtensor \cM_3\ar@{}[rdd]|\circlearrowleft \ar[d]_{\text{id}}\ar[r]^{\sim } 
& \cM_1\vtensor (\cM_2\vtensor \cM_3)\ar[d]^{\text{id}}\\
(\ol{\M_1\vtensor \M_2})\vtensor \cM_3\ar[d]_{\text{id}} 
& \cM_1\vtensor (\ol{\M_2\vtensor \M_3})\ar[d]^{\text{id}}\\
\ol{(\M_1\vtensor \M_2)\vtensor \M_3}\ar[r]^{\sim } & \ol{\M_1\vtensor (\M_2\vtensor \M_3)}
}
\]
\[
\xymatrix{
\mbb{C}\vtensor \cM\ \ar[d]_{\text{id}}\ar@{}[rd]|\circlearrowleft\ar[r]^{1\tens X\mapsto X} 
& \cM & \cM\vtensor \mbb{C}\ar[d]_{\text{id}}\ar@{}[rd]|\circlearrowleft\ar[r]^{X\tens 1\mapsto X} & \cM\\
\mbb{C}\vtensor \cM\ar[r]^{\text{id}} & \ol{\mbb{C}\vtensor \M}\ar[u] & \cM\vtensor \mbb{C}\ar[r]^{\text{id}} 
& \ol{\M\vtensor \mbb{C}}\ar[u]
}
\]  
Thus, $(\mc{E},h_1,h_2)$ is a tensor functor. 
The proof that $(\mc{F},h_1',h_2')$ is a tensor functor, including the definitions of $h_1',h_2'$ are easier.
\end{proof}

Now we are able to prove the main theorem easily.

\begin{proof}[{\rm{\textbf{Proof of Theorem \ref{6main theorem}}}}]
We will show that $\mc{E}$ and $\mc{F}$ are the inverse tensor functor of each other. 
By Lemma \ref{6 E and F are tensor functors}, they are tensor functors. 
Let $(\M_i,\Hs_i)$ $(i=1,2)$ be in Obj(\tb{fvN}). 
Let $\varphi: \M_1\to \M_2$ be a morphism in \tb{fvN}. 
Proposition \ref{6exntension of morphisms} implies $\varphi=(\mc{F}\circ \mc{E})(\varphi)$. 
By Proposition \ref{6characterization of R}, we have
\[(\M_i,\Hs_i)=(\cM_i\cap \mf{B}(\Hs_i),\Hs_i)=(\mc{F}\circ \mc{E})(\M_i,\Hs_i),\]
therefore $\mc{F}\circ \mc{E}=\text{id}_{\tb{\text{fvN}}}.$

Let $(\msc{R}_i,\Hs_i)\ (i=1,2)$ be objects in \tb{fRng}, 
$\Phi: (\msc{R}_1,\Hs_1)\to (\msc{R}_2,\Hs_2)$ be a morphism in \tb{fRng}. 
By Proposition \ref{6characterization of R}, we have $\msc{R}_i=\cM_i$ for a unique $(\M_i,\Hs_i)$ in Obj(\tb{fvN}). 
Similarly, we can prove that 
\[
(\msc{R}_i,\Hs_i)=(\mc{E}\circ \mc{F})(\msc{R}_i,\Hs_i),\hspace{0.4cm} (\mc{E}\circ \mc{F})(\Phi)=\Phi,
\] hence $\mc{E}\circ \mc{F}=\text{id}_{\tb{\text{fRng}}}$.
\end{proof}

Finally, we remark the correspondence of factors in \tb{fvN} and ones in \tb{fRng}.
Recall that, for a *-algebra $\mathscr{A}$, its {\it center} $Z(\mathscr{A})$ is defined by
\begin{equation*}
Z(\mathscr{A}) := \left\{x\in\mathscr{A}\ ;\ xy=yx, \ \text{for\ all}\ y\in\mathscr{A}\right\}.
\end{equation*}
$Z(\mathscr{A})$ is also a *-algebra.

\begin{prop}\label{6XY=1 implies YX=1} 
Let $\M$ be a finite von Neumann algebra on $\mc{H}$. The following conditions are equivalent.
\begin{list}{}{}
\item[(1)] The center $Z(\cM)$ of $\cM$ is trivial. 
I.e., $Z(\cM)=\mathbb{C}1_{\Hs}$.
\item[(2)] The center $Z(\M)$ of $\M$ is trivial.
\end{list}
\end{prop}

\begin{proof}
$(1)\Rightarrow (2)$ is evident.

$(2)\Rightarrow (1)$. Let $A\in \cM$ be a self-adjoint element of the center $Z(\cM)$. 
For any $u\in U(\M')$, we have  $uAu^*=A$. 
Therefore from the unitary covariance of the functional calculus, 
it holds that $u(A-i)^{-1}u^*=(A-i)^{-1}$ and $(A-i)^{-1}\in \M\cap \M'=\mathbb{C}1.$ Hence $(A-i)^{-1}=\alpha 1$ for some $\alpha \in \mathbb{C}$. 
By operating $A-i$ on both sides, we see that $A\in \mathbb{C}1.$ 
For a general closed operator $A\in Z(\cM)$, we know that there is a canonical decomposition $A=\overline{\re{A}+i\ \im{A}}$. 
Since $A$ belongs to $Z(\cM)$, $\disp \re{A},\ \im{A}$ also belong to $Z(\M)=\mathbb{C}1.$ 
Therefore $A\in \mathbb{C}1$.
\end{proof}

\appendix
\section{Direct Sums of Operators}

We recall the theory of direct sums of operators and show some facts.
We do not give proofs for well-known facts. 
See e.g., \cite{Arai}.

Let $\{\Hs_{\alpha}\}_{\alpha}$ be a family of Hilbert spaces and 
$\Hs = \bigoplus_{\alpha}\Hs_{\alpha}$ be the direct sum Hilbert space of $\{\Hs_{\alpha}\}_{\alpha}$, i.e.,

\begin{equation*}
\Hs := \left\{ \xi=\{\xi^{(\alpha)}\}_{\alpha}\ ;\ \xi^{(\alpha)} \in \Hs_{\alpha},\ \sum_{\alpha} \|\xi^{\alpha}\|^2 < \infty. \right\}.
\end{equation*}
For a subspace $\mathcal{D}_{\alpha}$ of $\Hs_{\alpha}$, we set 

\begin{equation*}
\widehat{\bigoplus}_{\alpha}\mathcal{D}_{\alpha} := 
\left\{ \xi=\{\xi^{(\alpha)}\}_{\alpha}\in \Hs \ ;\ \xi^{(\alpha)} \in \mathcal{D}_{\alpha},\ 
\xi^{(\alpha)}=0 \ {\rm except\ finitely\ many\ } \alpha. \right\}.
\end{equation*} 
It is known that $\widehat{\bigoplus}_{\alpha}\mathcal{D}_{\alpha}$ is dense in $\Hs$ whenever each $\mathcal{D}_{\alpha}$ is dense in $\Hs_{\alpha}$.

Next we recall the direct sum of unbounded operators.
Let $A_{\alpha}$ be a (possibly unbounded) linear operator on $\Hs_{\alpha}$. 
We define the liner operator $A = \oplus_{\alpha}A_{\alpha}$ on $\Hs$ as follows:

\begin{align*}
\dom{A} &:= \left\{ \xi=\{\xi^{(\alpha)}\}_{\alpha}\in \Hs \ ;\ 
\xi^{(\alpha)} \in \dom{A_{\alpha}},\ \sum_{\alpha} \|A_{\alpha}\xi^{\alpha}\|^2 < \infty. \right\}, \\
(A\xi)^{(\alpha)} &:= A_{\alpha}\xi^{(\alpha)}, \ \ \ \ \xi \in \dom{A}.
\end{align*}
$A$ is said to be the {\it direct sum} of $\{A_{\alpha}\}_{\alpha}$.
It is easy to see that if each $A_{\alpha}$ is a densely defined closed operator then so is $A$.
In this case,
\begin{equation*}
A^* = \oplus_{\alpha}{A_{\alpha}}^{*}
\end{equation*}
holds.
The following lemmata are well-known.

\begin{lem}\label{3bdd}
Assume the above notations.
\begin{list}{}{}
\item[(1)] $A\in \mathfrak{B}(\Hs)$ if and only if each $A_{\alpha}$ is in $\mathfrak{B}(\Hs_{\alpha})$ and $\sup_{\alpha}\|A_{\alpha}\|<\infty$.
In this case, 
\begin{equation*}
\|A\| = \sup_{\alpha}\|A_{\alpha}\|
\end{equation*}
holds.
\item[(2)] $A$ is unitary if and only if each $A_{\alpha}$ is unitary.
\item[(3)] $A$ is projection if and only if each $A_{\alpha}$ is projection.
In this case,
\begin{equation*}
\ran{A} = \bigoplus_{\alpha}\ran{A_{\alpha}}
\end{equation*}
holds.
\end{list}
\end{lem} 

\begin{lem}\label{3sum_core}
Assume that each $A_{\alpha}$ is closed.
Let $\mathcal{D}_{\alpha}$ be a core of $A_{\alpha}$.
Then $\widehat{\bigoplus}_{\alpha}\mathcal{D}_{\alpha}$ is a core of $A$.
\end{lem}

\begin{lem}\label{3unbdd}
Assume that each $A_{\alpha}$ is (possibly unbounded) self-adjoint.
\begin{list}{}{}
\item[(1)] $A$ is self-adjoint.
\item[(2)] For any complex valued Borel function $f$ on $\mathbb{R}$, 
\begin{equation*}
f(A) = \oplus_{\alpha}f(A_{\alpha})
\end{equation*}
holds.
\end{list}
\end{lem}

Finally, we study the direct sum of algebras of operators.
Let $\mathscr{S}_{\alpha}$ be a set of densely defined closed operators on $\Hs_{\alpha}$.
Put
\begin{equation*}
\bigoplus_{\alpha}\mathscr{S}_{\alpha} := \left\{\oplus_{\alpha}A_{\alpha}\ ; \ A_{\alpha}\in\mathscr{S}_{\alpha} \right\}.
\end{equation*}
Note that each element in $\bigoplus_{\alpha}\mathscr{S}_{\alpha}$ is a densely defined closed operator on $\Hs = \bigoplus_{\alpha}\Hs_{\alpha}$.
If each $\mathscr{S}_{\alpha}$ consists only of bounded operators, we also define
\begin{equation*}
\bigoplus_{\alpha}^{b}\mathscr{S}_{\alpha} := 
\left\{\oplus_{\alpha}x_{\alpha}\ ; \ x_{\alpha}\in\mathscr{S}_{\alpha},\ \sup_{\alpha}\|x_{\alpha}\|<\infty. \right\}.
\end{equation*} 
By Lemma \ref{3bdd}, each element in $\bigoplus_{\alpha}^{b}\mathscr{S}_{\alpha}$ is bounded.
The following is also well-known.

\begin{lem}\label{3sum_vN}
Let $\M_{\alpha}$ be a von Neumann algebra acting on $\Hs_{\alpha}$, and put 
\begin{equation*}
\M := \bigoplus_{\alpha}^{b}\M_{\alpha}.
\end{equation*}
Then $\M$ is a von Neumann algebra acting on $\Hs=\bigoplus_{\alpha}\Hs_{\alpha}$.
The sum, the scalar multiplication, the multiplication and the involution are given by 
\begin{align*}
\left(\oplus_{\alpha}x_{\alpha}\right)+\left(\oplus_{\alpha}y_{\alpha}\right) &= \oplus_{\alpha}\left(x_{\alpha}+y_{\alpha}\right), \\
\lambda \left(\oplus_{\alpha}x_{\alpha}\right) &= \oplus_{\alpha}\left(\lambda x_{\alpha}\right),\ \ \ \ for\ all\ \lambda \in \mathbb{C}, \\
\left(\oplus_{\alpha}x_{\alpha}\right)\left(\oplus_{\alpha}y_{\alpha}\right) &= \oplus_{\alpha}\left(x_{\alpha}y_{\alpha}\right), \\
\left(\oplus_{\alpha}x_{\alpha}\right)^* &=  \oplus_{\alpha}\left({x_{\alpha}}^*\right).
\end{align*}
Furthermore the followings hold.
\begin{list}{}{}
\item[(1)] $\M'= \bigoplus_{\alpha}^{b}\M_{\alpha}'$.
\item[(2)] $\M$ is a finite von Neumann algebra if and only if each $\M_{\alpha}$ is a finite von Neumann algebra.
\end{list}
\end{lem}

We call $\bigoplus_{\alpha}^{b}\M_{\alpha}$ the {\it direct sum von Neumann algebra} of $\left\{\M_{\alpha}\right\}_{\alpha}$.

\section{Fundamental Results of SRT}

Let $\Hs$ be a Hilbert space.
The following lemmata are well-known \cite{Reed-Simon}:
\begin{lem}\label{AcoreSRT}
Let $\{A_{\lambda}\}_{\lambda\in\Lambda}$ be a net of self-adjoint operators on $\Hs$, $A$ be a self-adjoint operator on $\Hs$,
and $\mathcal{D}$ be a dense subspace of $\Hs$ which is a core of $A$ and 
$\mathcal{D} \subset \bigcap_{\lambda\in\Lambda}\dom{A_{\lambda}}\cap\dom{A}$.
Suppose for all $\xi \in \mathcal{D}$, $\lim_{\lambda\in\Lambda} A_{\lambda}\xi = A\xi$, 
then $A_{\lambda}$ converges to $A$ in the strong resolvent topology.
\end{lem}

\begin{lem}\label{ASRT=exp}
Let $\{A_n\}_{n=1}^{\infty}$ be a sequence of self-adjoint operators on $\Hs$, $A$ be a self-adjoint operator on $\Hs$.
Then $A_n$ converges to $A$ in the strong resolvent topology if and only if 
$e^{itA_n}$ converges strongly to $e^{itA}$ for all $t \in \mathbb{R}$. 
In this case, the strong convergence of $e^{itA_n}$ to $e^{itA}$ is uniform on every finite interval of $t$.
\end{lem}

\begin{lem}\label{Aspec}
Let $\{A_n\}_{n=1}^{\infty}$ be a sequence of self-adjoint operators on $\Hs$, $A$ be a self-adjoint operator on $\Hs$.
Suppose $A_n$ converges to $A$ in the strong resolvent topology, 
then $E_{A_n}((a,b))$ converges strongly to $E_{A}((a,b))$ for each $a,b \in \mathbb{R}$ with $a<b$ and $a, b \notin \sigma_p(A)$, where $\sigma_p(A)$ is the set of point spectra of $A$.  
\end{lem}

\begin{lem}\label{Abf}
Let $\{A_n\}_{n=1}^{\infty}$ be a sequence of self-adjoint operators on $\Hs$, $A$ be a self-adjoint operator on $\Hs$.
Suppose $A_n$ converges to $A$ in the strong resolvent topology, 
then for all complex valued bounded continuous function $f$ on $\mathbb{R}$,
$f(A_n)$ converges strongly to $f(A)$.  
\end{lem}

\begin{lem}\label{Asrt-sot}
Let $\{x_{\lambda}\}_{\lambda\in\Lambda}$ be a net of bounded self-adjoint operators on $\Hs$, $x$ be a bounded self-adjoint operator on $\Hs$.
Suppose that 
\begin{equation*}
\sup_{\lambda\in\Lambda}\|x_{\lambda}\|<\infty,
\end{equation*} 
and $x_{\lambda}$ converges to $x$ in the strong resolvent topology, 
then $x_{\lambda}$ converges strongly to $x$.
\end{lem}

\section{Tensor Categories}

We briefly review the definition of tensor categories.\ For more details about category theory, 
see MacLane\ \cite{MacLane} (we follow the style in Kassel\ \cite{Kassel},\ Chapter XI).

\begin{dfn}\label{Bnaturality}
Let $\msc{C},\msc{C}'$ be categories, $\mathcal{F}, \mathcal{G}$ be functors from $\msc{C}$ to $\msc{C}'$. 
A \textit{natural transformation} $\theta:\mathcal{F}\to \mathcal{G}$ is a function 
which assigns to each object $A$ in $\msc{C}$ a morphism 
$\theta(A):\mathcal{F}(A)\to \mathcal{G}(A)$ of $\msc{C}'$ 
in such a way that for every morphism $f:A\to B$ in $\msc{C}$, 
the following diagram commutes:
\[
\xymatrix{
\mathcal{F}(A)\ar@{}[rd]|\circlearrowright \ar[d]_{\mathcal{F}(f)}\ar[r]^{\theta(A)}& \mathcal{G}(A)\ar[d]^{\mathcal{G}(f)}\\
\mathcal{F}(B)\ar[r]^{\theta(B)}& \mathcal{G}(B)
}
\]
If $\theta(A)$ is an invertible morphism for every $A$, we call $\theta$ a \textit{natural isomorphism}.
\end{dfn} 

\begin{dfn}\label{Btensor category}
A {\it{tensor category}} $(\msc{C},\tens,I,\alpha,\lambda,\rho)$ is a category $\msc{C}$\ equipped with 
\begin{list}{}{}
\item[(1)] a bifunctor $\tens: \msc{C}\times \msc{C}\to \msc{C}$ called a \textit{tensor product}
\footnote{
This implies $(f'\tens g')\circ (f\tens g)=(f'\circ f)\tens (g'\circ g)$ for all morphisms in $\msc{C}$, 
and $1_A\tens 1_B=1_{A\tens B}$ for all objects in $\msc{C}$.
},
\item[(2)] an object $I$ in $\msc{C}$ called a \textit{unit object}, 
\item[(3)] a natural isomorphism $\alpha : \tens (\tens \times 1_{\msc{C}})$
\footnote{
$\tens (\tens \times 1_{\msc{C}})$ is the composition of the functors 
$\tens \times 1_{\msc{C}}:(\msc{C}\times \msc{C})\times \msc{C}\to \msc{C}\times \msc{C}$ 
and $\tens:\msc{C}\times \msc{C}\to \msc{C}$.
}
$\to \tens (1_{\msc{C}}\times \tens)$
called an \textit{associativity constraint}.
\end{list} 
(3) means for any objects $A,B,C$ in $\msc{C}$, there is an isomorphism $\alpha_{A,B,C}:(A\tens B)\tens C\to A\tens (B\tens C)$ 
such that the diagram
\[
\xymatrix{
(A\tens B)\tens C\ar@{}[rd]|\circlearrowright \ar[d]_{(f\tens g)\tens h}\ar[r]^{\alpha_{A,B,C}} 
& A\tens (B\tens C)\ar[d]^{f\tens (g\tens h)}\\
(A'\tens B')\tens C'\ar[r]^{\alpha_{A',B',C'}} & A'\tens (B'\tens C')
}
\]
commutes for all morphisms $f,g,h$ in $\msc{C}$.
\begin{list}{}{}
\item[(4)] a natural isomorphism $\lambda: \tens(I\times 1_{\msc{C}})$
\footnote{
$I\times 1_{\msc{C}}$ is the functor from $\msc{C}$ to $\msc{C}\times \msc{C}$ 
given by $A\mapsto (I,A)$ for all objects in $\mathscr{C}$ 
and $f\mapsto (1_I,f)$ for all morphisms in $\msc{C}$.}$\to 1_{\msc{C}}$ 
(resp. $\rho: \tens(1_{\msc{C}}\times I)\to 1_{\msc{C}}$) called a {\it left} 
(resp. {\it right}) {\it unit constraint} with respect to $I$.
\end{list}
(4) means for any object $A$ in $\msc{C}$, there is an isomorphism 
$\lambda_{A}: I\tens A\to A$ (resp. $\rho_A:A\tens I\to A)$ such that the following two diagrams commute:
\[
\xymatrix{
I\tens A\ar@{}[rd]|\circlearrowright \ar[d]_{1_I\tens f}\ar[r]^{\lambda_A} 
& A\ar[d]^f & A\tens I\ar@{}[rd]|\circlearrowright \ar[d]_{f\tens 1_I}\ar[r]^{\rho_A} & A\ar[d]^{f}\\
I\tens A'\ar[r]^{\lambda_{A'}} & A' & A'\tens I\ar[r]^{\rho_{A'}} & A'
}
\]
for each morphism $f:A\to A'$ in $\msc{C}$.
These functors and natural isomorphisms satisfy the {\it Pentagon Axiom} and the {\it Triangle Axiom}. 
Namely, for all objects $A,B,C$ and $D$, the following diagrams commute: 
\[
\xymatrix{
((A\tens B)\tens C)\tens D\ar@{}[rrdd]|\circlearrowleft \ar[dd]_{\alpha_{A\tens B,C,D}} \ar[rr]^{\alpha_{A,B,C}\tens 1_{D}}
& &(A\tens (B\tens C))\tens D\ar[d]^{\alpha_{A,B\tens C,D}}\\
& & A\tens ((B\tens C)\tens D)\ar[d]^{1_A\tens \alpha_{B,C,D}}\\
(A\tens B)\tens (C\tens D)\ar[rr]^{\alpha_{A,B,C\tens D}} & & A\tens (B\tens (C\tens D))
}
\]
\[
\xymatrix{
(A\tens I)\tens B\ar[dr]_{\rho_A\tens 1_B}\ar@{}[rrd]|\circlearrowright \ar[rr]^{\alpha_{A,I,B}} & & 
A\tens (I\tens B)\ar[ld]^{1_A\tens \lambda_B}\\
& A\tens B & 
}
\]
\end{dfn}

\begin{dfn}\label{6 def of tensor functor} 
Let $(\msc{C},\tens,I,\alpha,\lambda,\rho),\ (\msc{C}',\tens,I',\alpha',\lambda',\rho')$ be tensor categories. 
\begin{list}{}{}
\item[(1)] A triple $(\mc{F},h_1,h_2)$ is called a {\it tensor functor} from $\msc{C}$ to $\msc{C}'$ 
if $\mc{F}:\msc{C}\to \msc{C}'$ is a functor, $h_1$ is an isomorphism $I'\stackrel{\sim }{\to}\mc{F}(I)$ 
and $h_2$ is a natural isomorphism $\otimes (\mathcal{F}\times \mathcal{F})$
\footnote{
$\otimes (\mathcal{F}\times \mathcal{F})$ is a functor $\msc{C}\times \msc{C}\to \msc{C}$ 
which assigns $\mathcal{F}(A)\otimes \mathcal{F}(B)$ for each object $(A,B)$ in $\msc{C}\times \msc{C}$ 
and $\mathcal{F}(f)\otimes \mathcal{F}(g)$ for each morphism $(f,g)$ 
in $\msc{C}\times \msc{C}$
}
$\stackrel{\sim}{\rightarrow }\mathcal{F}\otimes $, and they satisfy 
\[
\xymatrix{
(\mc{F}(A)\tens \mc{F}(B))\tens \mc{F}(C)\ar@{}[rrdd]|\circlearrowleft 
\ar[rr]^{\alpha_{\mc{F}(A),\mc{F}(B),\mc{F}(C)}} \ar[d]_{h_2(A,B)\tens 1_{\mc{F}(C)}}& 
&\mc{F}(A)\tens (\mc{F}(B)\tens \mc{F}(C))\ar[d]^{1_{\mc{F}(A)}\tens h_2(B,C)}\\
\mc{F}(A\tens B)\tens \mc{F}(C)\ar[d]_{h_2(A\tens B,C)} & & \mc{F}(A)\tens \mc{F}(B\tens C)\ar[d]^{h_2(A,B\tens C)}\\
\mc{F}((A\tens B)\tens C)\ar[rr]_{\mc{F}(\alpha_{A,B,C})} & & \mc{F}(A\tens (B\tens C))
}\]
\[
\xymatrix{
I'\tens \mc{F}(A)\ar@{}[rd]|\circlearrowleft \ar[d]_{h_1\tens 1_{\mc{F}(A)}}\ar[r]^{\lambda'_{\mc{F}(A)}} & 
\mc{F}(A) & \mc{F}(A)\tens I'\ar@{}[rd]|\circlearrowleft \ar[d]_{1_{\mc{F}(A)}\tens h_1}\ar[r]^{\rho'_{\mc{F}(A)}} & \mc{F}(A)\\
\mc{F}(I)\tens \mc{F}(A)\ar[r]_{h_2(I,A)} & \mc{F}(I\tens A)\ar[u]_{\mc{F}(\lambda_A)} & 
\mc{F}(A)\tens \mc{F}(I)\ar[r]_{h_2(A,I)} & \mc{F}(A\tens I)\ar[u]_{\mc{F}(\rho_A)}
}
\]
for all objects $A,B,C$ in $\msc{C}$.
\item[(2)] A {\it natural tensor transformation} $\eta: (\mc{F},h_1,h_2)\to (\mc{F}',h_1',h_2')$ 
between tensor functors from $\msc{C}$ to $\msc{C}'$ is a natural transformation $\mc{F}\to \mc{F}'$ 
such that the following diagrams commute:
\[
\xymatrix{
  & \mc{F}(I)\ar[dd]^{\eta(I)} & \mc{F}(A)\tens \mc{F}(B)\ar@{}[rdd]|\circlearrowleft 
\ar[dd]_{\eta(A)\tens \eta(B)}\ar[r]^{h_2(A,B)}  & \mc{F}(A\tens B)\ar[dd]^{\eta(A\tens B)}\\
I\ar@{}[r]|\circlearrowleft \ar[ur]_{h_1}\ar[dr]^{h'_1} &           &                           &                 \\
  & \mc{F}'(I)& \mc{F}'(A)\tens \mc{F}'(B)\ar[r]^{h'_2(A,B)} & \mc{F}'(A\tens B)
}
\]
for all objects $A,B$ in $\msc{C}$. 
If $\eta$ is also a natural isomorphism, it is called a {\it natural tensor isomorphism.} 
\item[(3)] A {\it tensor equivalence} between tensor categories $\msc{C},\msc{C}'$ is a tensor functor 
$\mc{F} : \msc{C}\to \msc{C}'$ such that there exists a tensor functor 
$\mc{F}': \msc{C}'\to \msc{C}$ and natural tensor isomorphisms 
$\eta : 1_{\msc{C}'}\stackrel{\sim }{\to }\mc{F}\circ \mc{F}'$ and 
$\theta : \mc{F}'\circ \mc{F}\stackrel{\sim}{\to }1_{\msc{C}}.$ 
If $\eta$ and $\theta$ can be taken to be identity transformations, 
then we say $\msc{C}$ is {\it isomorphic} to $\msc{C}'$ as a tensor category. 
\end{list}
\end{dfn}

\section*{Acknowledgement}
The authors would like to express their sincere thanks to Professor Asao Arai at Hokkaido University, 
Professor Izumi Ojima at Kyoto University for the fruitful discussions, insightful comments and encouragements. 
H.A. also thanks to his colleagues: Mr. Ryo Harada, Mr. Takahiro Hasebe, Mr. Kazuya Okamura and Mr. Hayato Saigo 
for their useful comments and discussions during the seminar.
Y.M. also thanks to Professor Konrad Schm\"{u}dgen at Universit\"{a}t Leipzig for informing us of the paper \cite{Schmuedgen2} and  
Mr. Yutaka Shikano at MIT for his professional advice about LaTeX.
Finally, the authors thank again to Professor Izumi Ojima for his careful proofreading and suggestions.

\end{document}